\documentclass[authoryear]{elsarticle}

\usepackage{booktabs}
\usepackage{multirow}
\usepackage{graphicx}  
\usepackage[font=small,format=plain,labelfont=bf,up,textfont=it,up]{caption}

\usepackage{subfigure}
\usepackage{geometry}
\usepackage{amsmath}
\usepackage{amssymb}
\usepackage{amsfonts}
\usepackage{amsthm}

\usepackage{float}

\usepackage{pgf,tikz}
\usetikzlibrary{shapes,arrows,calc,automata}

\usepackage{fancyhdr}
\definecolor{berkeley_lab_blue}{RGB}{0,57,90}

\newtheorem{myass}{Assumption}

\newtheorem{myprop}{Proposition}

\title{\textbf{Fully Implicit Multidimensional Hybrid Upwind Scheme for Coupled Flow and Transport}}
\date{}

\author[lbl]{Fran\c cois P. Hamon\corref{cor1}}
\ead{fhamon@lbl.gov}
\author[chevron]{Bradley T. Mallison}

\cortext[cor1]{Corresponding author, now with TOTAL Exploration \& Production USA}
\address[lbl]{Center for Computational Sciences and Engineering, Lawrence Berkeley National Laboratory, Berkeley, USA}
\address[chevron]{Chevron ETC, 6001 Bollinger Canyon Rd, San Ramon, USA}

\begin{document}

\begin{abstract} 
Robust and accurate fully implicit finite-volume schemes applied to Darcy-scale multiphase flow and transport in porous media 
are highly desirable. Recently, a smooth approximation of the saturation-dependent flux coefficients based on 
Implicit Hybrid Upwinding (IHU) has been proposed to improve the 
nonlinear convergence in fully implicit simulations with buoyancy. Here, we design a truly multidimensional extension 
of this approach that retains the simplicity and robustness of IHU while reducing the sensitivity of the results 
to the orientation of the computational (Cartesian) grid. This is achieved with the introduction of an adaptive, local coupling 
between the fluxes that takes the flow pattern into account. We analyze the mathematical properties 
of the proposed methodology to show that the scheme is monotone in the presence of competing viscous and buoyancy forces 
and yields saturations remaining between physical bounds. Finally, we demonstrate the efficiency and accuracy of the scheme 
on challenging two-dimensional two-phase examples with buoyancy, with an emphasis on the reduction of the grid orientation effect.
\end{abstract}


\maketitle

\section{\label{section_introduction}Introduction}

The numerical simulation of subsurface flow requires the design
of accurate and robust discretization schemes for the highly 
nonlinear partial differential equations (PDEs) governing
coupled flow and transport in porous media. The high heterogeneity 
characterizing geological porous media constitutes a challenge
for the computational models used in practice. In realistic 
models, the rock properties (permeability and porosity) can
vary by several orders of magnitude, which results in a vast range 
of flow velocities and time scales in the hyperbolic transport
of species. High localized flow velocities impose severe 
restrictions on the time step size for explicit time integration
methods. Therefore, a fully implicit, backward-Euler method 
guaranteeing unconditional stability is often preferred in the 
case of strongly nonlinear problems with high heterogeneity.

However, constructing accurate and efficient fully implicit
schemes for subsurface flow simulation is a challenging task.
In particular, the nonlinearity of the saturation-dependent
coefficients present in the flux terms makes the application of
high-resolution schemes challenging despite some encouraging
recent efforts \citep{mykkeltvedt2017fully,arbogast2019neumann}.
Therefore, these coefficients are
often approximated with robust first-order upwind schemes such
as the commonly used 
Phase-Potential Upwinding \citep[PPU,][]{peaceman2000fundamentals,aziz1979petroleum,sammon1988analysis,brenier1991upstream}.
Despite limitations with respect to resolution, first-order schemes
are preferred in practice because they can accommodate large time
steps when combined with Newton or quasi-Newton nonlinear iterations.
But, even limited to the family of first-order schemes, the
choice of an upwinding strategy has a significant impact on the
accuracy of the flow predictions and on the computational cost of
the numerical scheme. 

Previous authors have shown that nonlinear convergence issues can
result from the approximation of the interfacial phase flux
\citep{wang2013trust,li2014nonlinearity} and can consequently
undermine the performance of the scheme since the fully implicit
discretization requires solving large nonlinear systems at each
time step. Specifically, when the approximation is based on PPU,
convergence difficulties may arise when a small change in the
primary variables -- pressure and saturation -- causes an abrupt
change from cocurrent flow to countercurrent flow, and vice versa.
This issue, referred to as the ``flip-flopping'' of the phase
fluxes, was recently addressed with Implicit Hybrid Upwinding (IHU) 
in \cite{lee2015hybrid,lee2016c,lee2018hybrid,hamon2016analysis,hamon2016implicit,hamon2016capillary,moncorge2018consistent}. The IHU
strategy is based on a separate evaluation of the viscous, buoyancy,
and capillary parts to achieve a differentiable flux in the highly
nonlinear transport problem. This reduces the flip-flopping of the
phase fluxes and significantly improves the nonlinear convergence of 
Newton's method while producing a similar accuracy as that of the
standard PPU scheme.

The computational challenges are even greater in multiphase flows
where the displacing phase has greater mobility than the displaced
phase. Both the standard PPU and IHU upwinding procedures have
limited accuracy for adverse mobility ratios because the
saturation-dependent coefficients are approximated
dimension-by-dimension. That is, the construction of the flux
at an interface is entirely based on the saturation information
contained in the two adjacent control volumes. A well-known issue
resulting from this approach is the high sensitivity of the flow
pattern to the orientation of the computational grid. In practice,
these dimension-by-dimension schemes are therefore inaccurate
when the flow is not aligned with the orientation of the
computational grid. This unphysical grid orientation effect has
been studied extensively for porous media problems 
\citep{yanosik1979nine,pruess1983seven,shubin1984analysis,brand1991grid}. So-called multidimensional schemes taking into account the
saturation information in an extended stencil around each
interface have been designed to overcome this limitation of
two-point upwinding. Following the work of 
\cite{colella1990multidimensional,koren1991low,roe1992optimum,leveque1997wave,berger2003h} in computational fluid dynamics, the
multidimensional schemes tailored for subsurface flow simulation
include 
\cite{chen1993minimization,arbogast2006fully,velasco2007quadrilateral,lamine2010higher,lamine2013higher,lamine2015multidimensional,eymard2012grid}.

But, few of these schemes can be efficiently combined with a
fully implicit discretization targeting highly nonlinear problems
with competing viscous and buoyancy forces. In \cite{kozdon2011multidimensional,kozdon2011batstoicmame,keilegavlen2012multidimensional},
the authors design a multidimensional scheme based on PPU that
remains robust in fully implicit simulations of coupled multiphase
flow and transport, and that is provably monotone in the presence
of buoyancy. In this paper, we build on their approach to construct
an efficient and accurate multidimensional IHU scheme. As in 
\cite{kozdon2011multidimensional}, we neglect capillary forces
and focus on the mixed elliptic-hyperbolic incompressible two-phase
flow and transport problem with viscous and buoyancy forces
discretized on two-dimensional Cartesian grids. We propose a scheme
that reduces the sensitivity of the results to the orientation of
the computational grid while preserving a robust behavior of the
Newton solver in the presence of competing viscous and buoyancy
forces. 

This is achieved by adapting the methodology previously used to
construct the multidimensional PPU scheme. Specifically, to compute
the viscous numerical flux, we introduce a local coupling between
the phase fluxes belonging to the same interaction region of the
dual grid. The buoyancy flux is computed separately and is upwinded
based on fixed density differences. Importantly, this approach
yields a numerical flux that is a monotone function of the
saturations in the presence of viscous and buoyancy forces. We also
prove that the saturation solution to the proposed fully implicit
scheme remains between physical bounds, zero and one. Using
numerical examples, we demonstrate that the proposed scheme is
significantly less sensitive to the orientation of the grid than
the two-point PPU and IHU approaches. We also illustrate the
robustness of the multidimensional IHU by applying the scheme to
challenging nonlinear test cases with large time steps.

We first describe the mathematical model describing coupled
multiphase flow and transport with viscous and buoyancy forces
in Section~\ref{section_governing_equations}. Then, in
Section~\ref{section_numerical_scheme}, we proceed with the
description of the fully implicit finite-volume scheme and we
define the dual grid as the union of interaction regions. These
interaction regions are used to introduce the truly multidimensional
IHU in
Section~\ref{section_multidimensional_implicit_hybrid_upwinding},
and to study its mathematical properties in
Section~\ref{section_mathematical_properties}. Finally, using the
nonlinear solver described in Section~\ref{section_nonlinear_solver},
we demonstrate the accuracy and robustness of the multidimensional
scheme using numerical examples with buoyancy in
Section~\ref{section_numerical_examples}.

\section{\label{section_governing_equations}Governing equations}

In this work, we consider two immiscible and incompressible fluid
phases flowing in an incompressible porous medium. Mass conservation
for phase $\ell$ is expressed as
\begin{equation}
\phi \frac{\partial S_{\ell}}{\partial t} + \nabla \cdot \mathbf{u}_{\ell}  = q_{\ell}
\qquad \forall \, \mathbf{x} \in \Omega \subset \mathbb{R}^2, \quad
\forall \, \ell \in \{ \textit{nw}, w\},
\label{mass_conservation}
\end{equation}
where $\phi(\mathbf{x})$ is the porosity of the medium and $t$
is time. The source/sink term is denoted by $q_{\ell}$, with
the convention that $q_{\ell} > 0$ for injection, and $q_{\ell} < 0$
for production. The saturation $S_{\ell}(\mathbf{x},t)$ represents
the fraction of the pore volume occupied by phase $\ell$, with
the constraint that the sum of saturations is equal to one,
\begin{equation}
\sum_{\ell} S_{\ell} = 1. \label{saturation_constraint}
\end{equation}
Using the multiphase extension of Darcy's law, the phase velocity
$\mathbf{u}_{\ell}$ is written as a function of the phase potential
gradient $\nabla \Phi_{\ell}$ as
\begin{equation}
\mathbf{u}_{\ell} := -k \lambda_{\ell} \nabla \Phi_{\ell} = -k \lambda_{\ell} \big( \nabla p - \mathbf{g}_{\ell} \big) 
\qquad \text{with } \, \mathbf{g}_{\ell} = \rho_{\ell} g \nabla z, 
\quad  \forall \, \ell \in \{ \textit{nw}, w\},
\label{Darcy_s_law}
\end{equation}
where capillary forces are neglected. In (\ref{Darcy_s_law}),
$p(\mathbf{x},t)$ is the pressure, $k(\mathbf{x})$ is the scalar
rock permeability,  
$\lambda_{\ell}(S_{\textit{nw}}, S_w) = k_{r \ell}(S_{\textit{nw}}, S_w) / \mu_{\ell}$
is the phase mobility -- defined as the phase relative permeability 
divided by the phase viscosity --, $\rho_{\ell}$ is the phase
density, $g$ is the gravitational acceleration, and $z$ is the depth. 
We require that the mobilities be monotone functions of the
saturations, which is realistic and holds for two-phase relative
permeability models such as the Corey model
\citep{corey1954interrelation}.
\begin{myass}{(Phase mobilities).} \label{assumption_mobilities}
For $\ell \in \{\textit{nw}, w \}$, the mobility of phase $\ell$ is
positive, and a differentiable function of the saturations.
Furthermore, the mobility of phase $\ell$ is increasing with
respect to its saturation and decreasing with respect to the
saturation of the other phases: 
\begin{equation} 
\frac{\partial \lambda_{\ell}}{\partial S_{\ell}} \geq 0 \qquad \text{and} \qquad \frac{\partial \lambda_{\ell}}{\partial S_{m}} \leq 0 
\quad \forall \,  m \neq \ell.
\end{equation}
\end{myass}

The system of mixed elliptic-hyperbolic governing PDEs can be
written in two equivalent forms, where the linearly independent
primary variables are the pressure, $p$, and the wetting-phase
saturation, denoted by $S = S_w$ for simplicity. In the remainder
of this paper, we use the saturation constraint
(\ref{saturation_constraint}) to write the saturation-dependent
properties as a function of $S$ only. The first form of the
governing PDEs is obtained by inserting the expression of the
phase velocities given by Darcy's law (\ref{Darcy_s_law}) 
into the mass conservation equations (\ref{mass_conservation}), 

\begin{equation}
\phi \frac{\partial S_{\ell}}{\partial t} - \nabla \cdot \big( k \lambda_{\ell} ( \nabla p - \mathbf{g}_{\ell} ) \big) = q_{\ell}
\qquad \forall \, \ell \in \{ \textit{nw}, w \}. 
\label{governing_pdes_batstoiplus}
\end{equation}
The second form of the governing PDEs is split into a flow problem
and a transport problem. The flow problem is obtained by summing
(\ref{mass_conservation}) over all phases and using the saturation 
constraint (\ref{saturation_constraint}) to obtain:
\begin{equation}
\nabla \cdot \boldsymbol{u}_T (\boldsymbol{x}, p, S) = \sum_{\ell} q_{\ell} = q_T,  \label{continuous_pressure_equation}
\end{equation}
where the total velocity, $\mathbf{u}_T$, is defined as the sum of
the phase velocities:
\begin{equation}
\mathbf{u}_T(\mathbf{x}, p, S) := \sum_{\ell} \mathbf{u}_{\ell}(\mathbf{x}, p, S) 
= -k \lambda_{T} \nabla p + k \sum_{\ell} \lambda_{\ell} \mathbf{g}_{\ell}. \label{total_velocity}
\end{equation}
We assume that the following assumption on the total mobility holds: 
\begin{myass}{(Total mobility).} \label{ass_total_mobility}
The total mobility $\lambda_T = \sum_{\ell} \lambda_{\ell}$ is
bounded away from zero:
\begin{equation}
\qquad \qquad \qquad
0 < \epsilon \leq \lambda_T(S) \leq \chi 
\qquad \forall \, S \in [0,1]. 
\label{bounds_on_total_mobility}
\end{equation}
\end{myass}
Equation (\ref{continuous_pressure_equation}), referred to as the
pressure equation, is parabolic when compressibility is taken
into account \citep{trangenstein1989mathematicalb,trangenstein1989mathematical}, and elliptic in the incompressible case that we
will exclusively consider. The flow problem governs the evolution
of the pressure variables as a function of space and time.
It is coupled to the highly nonlinear transport of species,
derived next. Equation (\ref{total_velocity}) is used
to express the pressure gradient as a function of $\mathbf{u}_T$
and the weights $\mathbf{g}_{\ell}$ $(\ell \in \{ \textit{nw}, w \})$,
in order to eliminate the pressure variable from (\ref{Darcy_s_law}).
We obtain
\begin{equation}
\mathbf{u}_{\ell}(\mathbf{x}, p, S) = \frac{\lambda_{\ell}}{\lambda_T}  \mathbf{u}_T(\mathbf{x}, p, S) 
+ k \sum_m \frac{\lambda_m \lambda_{\ell}}{\lambda_T} 
(\mathbf{g}_{\ell} - \mathbf{g}_m). \label{pout_pout}
\end{equation}
This allows us to rewrite the system of PDEs in the following
fractional flow formulation which governs the hyperbolic transport 
of species as
\begin{equation}
\phi \frac{\partial S_{\ell}}{\partial t} + \nabla \cdot \bigg( \frac{\lambda_{\ell}}{\lambda_T}  \mathbf{u}_T(\mathbf{x}, p, S) 
+ k \sum_m \frac{\lambda_m \lambda_{\ell}}{\lambda_T} 
(\mathbf{g}_{\ell} - \mathbf{g}_m) \bigg) = q_{\ell}
\qquad \forall \, \ell \in \{ \textit{nw}, w \}. 
\label{governing_pdes_plusbatstoiplus}
\end{equation}
The system composed of (\ref{continuous_pressure_equation}) and
(\ref{governing_pdes_plusbatstoiplus}) contains a redundant equation. 
Although it is possible to solve both equations in
(\ref{governing_pdes_plusbatstoiplus}) by relaxing the saturation 
constraint, we enforce the saturation constraint and solve
(\ref{continuous_pressure_equation}) along with
(\ref{governing_pdes_plusbatstoiplus}) 
only for $\ell = w$. We highlight that in the two equivalent
systems of governing PDEs of (\ref{governing_pdes_batstoiplus}), or 
(\ref{continuous_pressure_equation}) and
(\ref{governing_pdes_plusbatstoiplus}), the flow is tightly coupled
to the highly nonlinear transport of species through the mobility
terms. Next, we employ a fully implicit finite-volume scheme to
discretize the system of governing PDEs, with an emphasis on the
truly multidimensional computation of the numerical flux aiming at
reducing grid orientation effects. As explained below, the proposed
methodology is based on the fractional flow formulation and builds
on the Implicit Hybrid Upwinding approach \citep{lee2015hybrid} to
construct the flux approximation.

\section{\label{section_numerical_scheme}Numerical scheme}

\subsection{\label{subsection_interaction_regions}Integration region framework}

We consider a two-dimensional Cartesian grid consisting 
of $M$ control volumes discretizing the domain $\Omega$. 
In the multidimensional scheme proposed in this paper, 
the flux is computed with an adaptive stencil containing 
at least two upwind control volumes in
the neighborhood of the interface under consideration.
The weights given to the upwind
volumes are chosen to improve accuracy by accounting
for characteristic information while also honoring monotonicity.
To define this stencil, we adapt the approach previously used in 
\cite{kozdon2011multidimensional,keilegavlen2012multidimensional} 
and use a dual grid made of the union of interaction regions 
as illustrated in Fig.~\ref{fig:interaction_regions}. 
Control-volume finite element methods
\citep{schneider1986skewed} and multi-point
flux finite volume methods
\citep{aavatsmark1998control,edwards1998finite,aavatsmark2002introduction}
utilize similar definitions of interaction regions and dual grids.

Each control volume in the primal grid is part of four interaction 
regions. As in \cite{kozdon2011multidimensional}, these interaction 
regions are locally labeled counterclockwise using the superscript
$(m)$, 
with $m \in \{1, \dots, 4\}$, starting from the bottom left corner 
of the control volume. In the dual grid, the four vertices of an 
interior interaction region correspond to the centers of the four 
control volumes intersected by this interaction region. These four 
vertices are locally labeled counterclockwise with the 
subscript $k \in \{1, \dots, 4\}$, starting from the bottom 
left vertex of the interaction region. Because we are developing 
first-order schemes, the primary variables are located at the
centers of the control volumes in the primal grid, or equivalently,
at the vertices of the interaction regions in the dual grid.
We denote by $p^{(m)}_k$ (respectively, $S^{(m)}_k$) 
the pressure (respectively, the wetting-phase saturation) located 
at the $k^{\text{th}}$ vertex of interaction region $(m)$.
We point out that $p^{(1)}_3$, $p^{(2)}_4$, $p^{(3)}_1$, and
$p^{(4)}_2$ refer to the same degree of freedom viewed from
different interaction regions.

The flux computation is performed at the four control volume 
segments contained in each interaction region and referred 
to as half interfaces. We write the quantities 
computed at these half interfaces with an overline to distinguish
them from the control-volume centered quantities. The four 
half interfaces are also locally numbered counterclockwise
with the subscript $k + 1/2 \in \{1+1/2, \dots, 4+1/2\}$.
In interaction region $(m)$, we denote by
$\overline{F}^{(m)}_{\ell,k+1/2}$ 
the approximation of the integral of the velocity of phase $\ell$
at the half interface $\Gamma^{(m)}_{k+1/2}$ between vertex $k$ and 
$k+1$, defined as
\begin{eqnarray}
\overline{F}^{\, (m)}_{\ell,k+1/2}
& \approx & \int_{\Gamma^{(m)}_{k+1/2}} \boldsymbol{u}_{\ell}(\boldsymbol{x}, p, S) \cdot \boldsymbol{n}^{(m)}_{k+1/2} \, \text{d} \Gamma^{(m)}_{k+1/2}, \label{discrete_flux} 
\end{eqnarray}
where $\boldsymbol{n}^{(m)}_{k+1/2}$ is the outward normal to
$\Gamma^{(m)}_{k+1/2}$. In our convention,
$\overline{F}^{(m)}_{\ell,k+1/2} \geq 0$ means 
that the positive flux for phase $\ell$ at half interface
$\Gamma^{(m)}_{k+1/2}$ is from vertex $k$ to vertex $k+1$ in
interaction region $(m)$. Similarly, considering the total
velocity at half interface $\Gamma^{(m)}_{k+1/2}$, which satisfies 
\begin{equation}
\overline{u}^{(m)}_{T,k+1/2} = \sum_{\ell} \overline{F}^{(m)}_{\ell,k+1/2},
\label{discrete_total_velocity}  
\end{equation}
the inequality $\overline{u}^{(m)}_{T,k+1/2} \geq 0$ means that 
the positive total flux at half interface $\Gamma^{(m)}_{k+1/2}$
is from vertex $k$ to $k+1$.  Note that the subscript $k$ is
defined cyclically on $\{1, \dots, 4\}$, and that, with a slight
abuse of notation, we write $k+1$ instead of $mod(k,4)+1$ in the
previous discussion. With this convention, $k+1$ indicates a
counterclockwise step around an interaction region and $k-1$
indicates a clockwise step.

\begin{figure}[ht]
\centering
\scalebox{0.95}{
\tikzstyle{int}=[draw, minimum size=2em]
\tikzstyle{init} = [pin edge={to-,thick,black}]

\begin{tikzpicture}[node distance=3cm,auto,>=latex']

  \path (3,3) node (a) {};
  \path (6,6) node (b) {};
  \path [draw=black,fill=berkeley_lab_blue!15!white,thick] (a) rectangle (b);

  \path (0,0) node (a) {};
  \path (3,6) node (b) {};
  \path [draw=berkeley_lab_blue,very thick] (a) rectangle (b); 
  \path (3,0) node (a) {};
  \path (6,3) node (b) {};
  \path [draw=berkeley_lab_blue,very thick] (a) rectangle (b); 
  \path (6,0) node (a) {};
  \path (9,3) node (b) {};
  \path [draw=berkeley_lab_blue,very thick] (a) rectangle (b); 

  \path (0,3) node (a) {};
  \path (3,6) node (b) {};
  \path [draw=berkeley_lab_blue,very thick] (a) rectangle (b); 
  \path (3,3) node (a) {};
  \path (6,6) node (b) {};
  \path [draw=berkeley_lab_blue,very thick] (a) rectangle (b); 
  \path (6,3) node (a) {};
  \path (9,6) node (b) {};
  \path [draw=berkeley_lab_blue,very thick] (a) rectangle (b); 

  \path (0,6) node (a) {};
  \path (3,9) node (b) {};
  \path [draw=berkeley_lab_blue,very thick] (a) rectangle (b); 
  \path (3,6) node (a) {};
  \path (6,9) node (b) {};
  \path [draw=berkeley_lab_blue,very thick] (a) rectangle (b); 
  \path (6,6) node (a) {};
  \path (9,9) node (b) {};
  \path [draw=berkeley_lab_blue,very thick] (a) rectangle (b); 

  \node (c) at (1.5,4.5) {\Large $\bullet$}; 
  \node (c) at (7.5,1.5) {\Large $\bullet$}; 
  \path (1.5,1.5) node (a) {\Large $\bullet$};
  \path (4.5,4.5) node (b) {\Large $\bullet$};
  \path [draw=black,dashed,very thick] (a) rectangle (b); 
  \path (4.5,1.5) node (a) {\Large $\bullet$};
  \path (7.5,4.5) node (b) {\Large $\bullet$};
  \path [draw=black,dashed,very thick] (a) rectangle (b); 

  \node (c) at (1.5,1.5) {\Large $\bullet$};
  \node (c) at (1.5,4.5) {\Large $\bullet$};
  \node (c) at (1.5,7.5) {\Large $\bullet$};
  \node (c) at (4.5,7.5) {\Large $\bullet$};
  \node (c) at (7.5,7.5) {\Large $\bullet$};
  \path (4.5,4.5) node (a) {};
  \path (4.5,7.7) node (b) {};
  \path [draw=black,dashed,very thick] (a) -- (b);
  \path (1.5,4.5) node (a) {\Large $\bullet$};
  \path (1.5,7.7) node (b) {};
  \path [draw=black,dashed,very thick] (a) -- (b);
  \path (7.5,4.5) node (a) {\Large $\bullet$};
  \path (7.5,7.7) node (b) {};
  \path [draw=black,dashed,very thick] (a) -- (b); 
  \path (1.5,7.5) node (a) {};
  \path (7.7,7.5) node (b) {};
  \path [draw=black,dashed,very thick] (a) -- (b);

  \node (a) at (3.75,2.8725) {}; 
  \node (b) at (3.75,3.75) {}; 
  \node (c) at (3.75,3.9) {$\boldsymbol{n}^{(1)}_{2 + 1/2}$};
  \path[->,draw=black,very thick] (a) edge node {} (b);  

  \node (a) at (3.75,5.8725) {}; 
  \node (b) at (3.75,6.75) {}; 
  \node (c) at (3.75,6.9) {$\boldsymbol{n}^{(4)}_{2 + 1/2}$};
  \path[->,draw=black,very thick] (a) edge node {} (b);  

  \node (a) at (5.25,3.1175) {}; 
  \node (b) at (5.25,2.25) {}; 
  \node (c) at (5.085,2.4) {$\boldsymbol{n}^{(2)}_{4 + 1/2}$};
  \path[->,draw=black,very thick] (a) edge node {} (b);  

  \node (a) at (5.25,6.1175) {}; 
  \node (b) at (5.25,5.25) {}; 
  \node (c) at (5.085,5.4) {$\boldsymbol{n}^{(3)}_{4 + 1/2}$};
  \path[->,draw=black,very thick] (a) edge node {} (b);

  \node (a) at (3.155,3.75) {}; 
  \node (b) at (2.25,3.75) {}; 
  \node (c) at (2.3,3.65) {$\boldsymbol{n}^{(1)}_{3 + 1/2}$};
  \path[->,draw=black,very thick] (a) edge node {} (b);  

  \node (a) at (2.8725,5.25) {}; 
  \node (b) at (3.75,5.25) {}; 
  \node (c) at (3.675,5.565) {$\boldsymbol{n}^{(4)}_{1 + 1/2}$};
  \path[->,draw=black,very thick] (a) edge node {} (b);  

  \node (a) at (6.155,3.75) {}; 
  \node (b) at (5.25,3.75) {}; 
  \node (c) at (5.3,3.65) {$\boldsymbol{n}^{(2)}_{3 + 1/2}$};
  \path[->,draw=black,very thick] (a) edge node {} (b);  

  \node (a) at (5.8725,5.25) {}; 
  \node (b) at (6.75,5.25) {}; 
  \node (c) at (6.675,5.565) {$\boldsymbol{n}^{(3)}_{1 + 1/2}$};
  \path[->,draw=black,very thick] (a) edge node {} (b);  

  \node (a) at (2,7) {\Large $(4)$};
  \node (a) at (2,2) {\Large $(1)$};
  \node (a) at (7,2) {\Large $(2)$};
  \node (a) at (7,7) {\Large $(3)$}; 

  \path (4.5,4.5) node (a) {\textcolor{berkeley_lab_blue}{\Large $\bullet$}};


  \node (a) at (13.5,5) {}; 
  \node (b) at (11,2.5) {}; 
  \path [fill=berkeley_lab_blue!15!white,very thick] (a) rectangle (b); 

  \node (c) at (16,0) {\Large $\bullet$}; 
  \node (c) at (11,5) {\Large \textcolor{berkeley_lab_blue}{$\bullet$}}; 
  \path (11,0) node (a) {\Large $\bullet$};
  \path (16,5) node (b) {\Large $\bullet$};
  \path [draw=black,dashed,very thick] (a) rectangle (b); 

  \node (a) at (13.5,-0.15) {}; 
  \node (b) at (13.5,5.15) {}; 
  \path[draw=berkeley_lab_blue,very thick] (a) edge node {} (b);

  \node (a) at (10.85,2.5) {}; 
  \node (b) at (16.15,2.5) {}; 
  \path[draw=berkeley_lab_blue,very thick] (a) edge node {} (b);

  \node (c) at (11.4,-0.5) {\Large $p^{}_{1}, \, S^{}_{1}$};
  \node (c) at (11.4,5.5) {\Large $p^{}_{4}, \, S^{}_{4}$};
  \node (c) at (15.9,5.5) {\Large $p^{}_{3}, \, S^{}_{3}$};
  \node (c) at (15.9,-0.5) {\Large $p^{}_{2}, \, S^{}_{2}$};

  \node (a) at (13.625,3.75) {}; 
  \node (b) at (12.5 ,3.75) {}; 
  \node (c) at (12.65,4.1) {\Large $\boldsymbol{n}^{}_{3 + 1/2}$};
  \path[->,draw=black,very thick] (a) edge node {} (b);  

  \node (a) at (12.25,2.625) {}; 
  \node (b) at (12.25,1.5) {}; 
  \node (c) at (12.25,1.525) {\Large $\boldsymbol{n}^{}_{4 + 1/2}$};
  \path[->,draw=black,very thick] (a) edge node {} (b);  

  \node (a) at (13.375,1.25) {}; 
  \node (b) at (14.5 ,1.25) {}; 
  \node (c) at (14.35,1.6) {\Large $\boldsymbol{n}^{}_{1 + 1/2}$};
  \path[->,draw=black,very thick] (a) edge node {} (b);  

  \node (a) at (14.75,2.375) {}; 
  \node (b) at (14.75,3.5) {}; 
  \node (c) at (14.75,3.65) {\Large $\boldsymbol{n}^{}_{2 + 1/2}$};
  \path[->,draw=black,very thick] (a) edge node {} (b);  

  \node (a) at (13.5,7) {\Large Local indexing in};
  \node (b) at (13.6,6.45) {\Large interaction region (2)};

\end{tikzpicture}
}
\caption{\label{fig:interaction_regions}Primal grid (solid blue
  line) and dual grid (dashed black line). The schematic on the
  left shows the four interaction regions intersecting the
  control volume in light
  blue. The schematic on the right focuses on 
  interaction region (2) located at the bottom right of the control 
  volume in light blue. In this schematic, we omitted the superscript 
  (2) denoting the interaction region index. The dots show the
  location of the control volume centers, or,
  equivalently, the vertices of the interaction regions. The arrows
  show the outward normals to the half interfaces. There are four
  half interfaces associated with each interaction region and, as
  a result, there are eight half interfaces associated with a primal
  control volume.}
\end{figure}
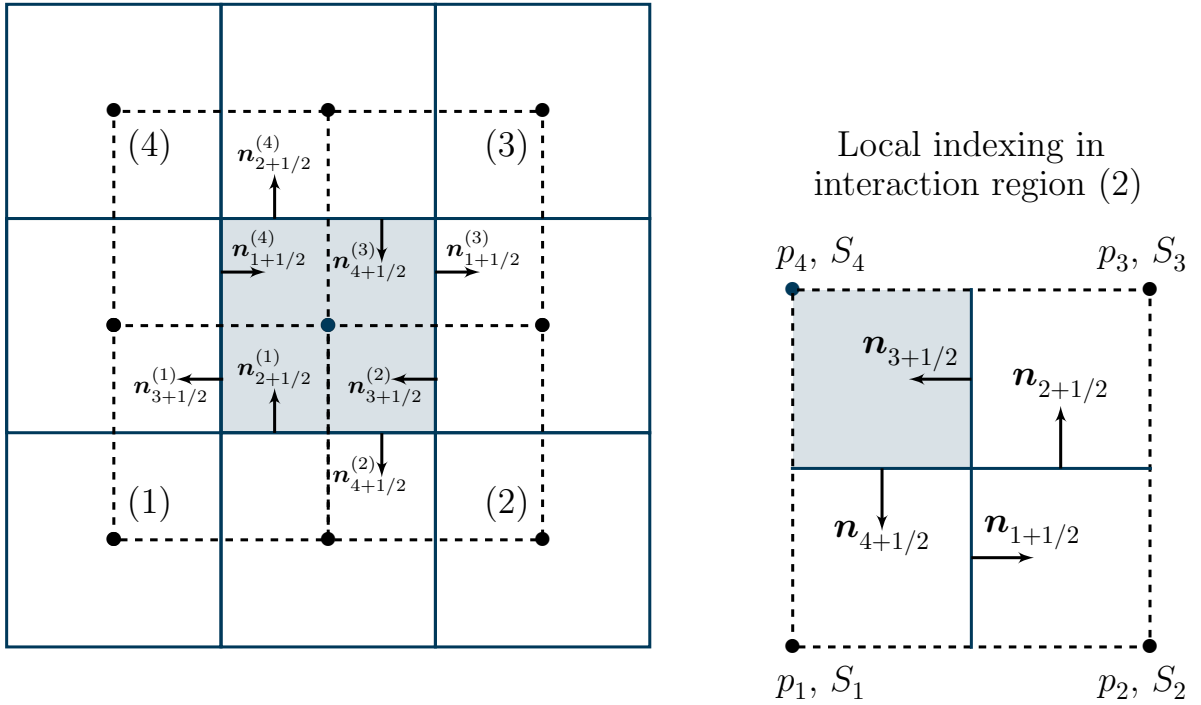

\subsection{\label{subsection_implicit_finite_volume_scheme}Implicit finite-volume scheme}

We are ready to write the discrete version of the system of
governing equations 
(\ref{continuous_pressure_equation})-(\ref{governing_pdes_plusbatstoiplus}). 
We consider the control volume in blue in
Fig.~\ref{fig:interaction_regions}. Viewed from interaction region
(2), the pressure and saturation of phase $\ell$ at the center
of this control
volume are on the fourth vertex and are denoted by $p^{(2)}_{\ell,4}$
and $S^{(2)}_{\ell,4}$. The
statement of discrete mass conservation for phase
$\ell \in \{1,2\}$ in this control volume reads
\begin{equation}
V \phi \displaystyle \frac{S^{(2)}_{\ell,4}-(S^{(2)}_{\ell,4})^n}{\Delta t^n} 
     + \sum_{k} \bigg( \overline{F}^{\, (k-2)}_{\ell,k+1/2} - \overline{F}^{\, (k-2)}_{\ell,k-1/2} \bigg) 
     = V q_{\ell},
\label{discrete_transport_equation}
\end{equation}
where $\Delta t^n = t^{n+1} - t^n$ is the time step and $V$ is
the volume of the control volume. We assume fixed positive
injection terms balanced by variable (i.e., pressure-controlled)
negative production terms. The flux term can be 
split to isolate the contribution from each interaction region as
\begin{eqnarray}
\sum_{k} \bigg( \overline{F}^{\, (k-2)}_{\ell,k+1/2} - \overline{F}^{\, (k-2)}_{\ell,k-1/2} \bigg) 
& = & 
\underbrace{\overline{F}^{\, (1)}_{\ell,3+1/2} - \overline{F}^{\, (1)}_{\ell,2+1/2}}_{\substack{\text{interaction} \\ \text{region (1)}}}
\, + \, 
\underbrace{\overline{F}^{\, (2)}_{\ell,4+1/2} - \overline{F}^{\, (2)}_{\ell,3+1/2}}_{\substack{\text{interaction} \\ \text{region (2)}}} \nonumber \\& + &
\underbrace{\overline{F}^{\, (3)}_{\ell,1+1/2} - \overline{F}^{\, (3)}_{\ell,4+1/2}}_{\substack{\text{interaction} \\ \text{region (3)}}}
\, + \,
\underbrace{\overline{F}^{\, (4)}_{\ell,2+1/2} - \overline{F}^{\, (4)}_{\ell,1+1/2}}_{\substack{\text{interaction} \\ \text{region (4)}}}.
\label{flux_term_decomposition}
\end{eqnarray}
Summing (\ref{discrete_transport_equation}) over all phases
and using (\ref{discrete_total_velocity}) to form the total
velocity results in the discrete pressure equation:
\begin{equation}
\sum_{k} \bigg( \overline{u}^{\, (k-2)}_{T,k+1/2} - \overline{u}^{\, (k-2)}_{T,k-1/2} \bigg) 
     = V q_{T},
\label{discrete_pressure_equation}
\end{equation}
where we have used the saturation constraint
(\ref{saturation_constraint}) to cancel the derivatives of
saturation on the left-hand side. Finally, we underline the
fact that we use a fully implicit discretization, with all
the quantities evaluated at the most recent time and iteration
level, except the saturation at the previous time in the accumulation
term.

\subsection{\label{subsection_flux_discretization}Flux approximation}

The flux defined in (\ref{discrete_flux}) can be decomposed
into two terms. The first term is a static transmissibility
coefficient depending on the rock properties and the grid geometry
in a stencil involving multiple control volumes in the
neighborhood of the interface. This term is independent 
of the saturations and can be computed in a preprocessing step. 
In this work, we limit the scope of the numerical study to
Cartesian grids with scalar permeability and therefore exclusively
consider a Two-Point
Flux Approximation (TPFA) to compute this transmissibility term.
Multi-Point Flux Approximations
\citep[MPFA,][]{aavatsmark1998control,edwards1998finite,aavatsmark2002introduction,zhou2011automatic}
are required for more general cases as previously shown by
\cite{keilegavlen2012multidimensional,lamine2013higher,souza2018higher}. Considering half interface
$\Gamma^{(m)}_{k+1/2}$ between vertices $k$ and $k+1$, we define
the rock- and geometric transmissibility as
\begin{equation}
\overline{T}^{(m)}_{k+1/2} :=
\frac{2 l_{k+1/2}}{\displaystyle d_{k+1/2}
    \bigg( \frac{1}{k_{k}} + \frac{1}{k_{k+1}} \bigg)}.
\label{rock_and_geometric_transmissibility}
\end{equation}
In (\ref{rock_and_geometric_transmissibility}), $l_{k+1/2}$
is the length of half interface $\Gamma^{(m)}_{k+1/2}$,
$d_{k+1/2}$ is the distance between vertices $k$ and $k+1$, and
$k_k$ (respectively, $k_{k+1}$) denotes the absolute permeability
in the control volume centered at vertex $k$ (respectively, vertex
$k+1$).
This choice implies that in two consecutive interaction regions
-- e.g., (1) and (2) in Fig.~\ref{fig:interaction_regions} --
the two half interfaces splitting a full control-volume interface
-- $\Gamma^{(1)}_{2+1/2}$
and $\Gamma^{(2)}_{4+1/2}$ -- have the same transmissibility
-- that is,
$\overline{T}^{(1)}_{2+1/2} = \overline{T}^{(2)}_{4+1/2}$. In
the ordering of Fig.~\ref{fig:interaction_regions}, we have
\begin{equation}
\overline{T}^{(k+1)}_{k-3/2} = \overline{T}^{(k+2)}_{k+1/2}.
\end{equation}

The second term is a highly nonlinear, non-convex function of the
saturations whose computation has a significant impact on the
robustness and the accuracy of the numerical scheme.
Even though high-resolution schemes have
been developed for multiphase flow in porous media
\citep{blunt1992implicit,mallison2005high,lamine2010higher,mykkeltvedt2017fully}, we
focus on commonly used monotone first-order upwinding
schemes and will consider higher-resolution extensions
in future work. Robust two-point upwinding schemes such
as Phase-Potential Upwinding (PPU) and Implicit Hybrid
Upwinding (IHU) are relatively simple and inexpensive but
suffer from a strong sensitivity to the orientation of the
computational grid for flows involving adverse mobility ratios.
This leads to inaccurate predictions when the flow is not aligned
with the grid, which is often the case in practical simulations. 

In this work, we address this issue with a multidimensional
evaluation of the saturation-dependent coefficients based on
a larger stencil that adapts to the flow direction.
Our approach extends the original IHU approach to improve the
accuracy of the fully implicit numerical scheme while retaining 
the robustness of the IHU scheme in the presence of strong buoyancy. 
The scheme is constructed to reduce the sensitivity of the results 
to the orientation of the computational grid by introducing a local
coupling between the fluxes in each interaction region. A half 
flux can therefore be a function of all the saturations in 
the interaction region, leading to a stencil with a maximum of 
nine points in the discretization of the mobilities. To guarantee
that 
the resulting scheme is well behaved, we require that for a
fixed total velocity field, the fluxes satisfy the following
monotonicity constraint in interaction region $(m)$: 
\begin{equation}
\frac{\partial (\overline{F}^{\, (m)}_{\ell,k+1/2} - \overline{F}^{\, (m)}_{\ell,k-1/2}) }{\partial S^{(m)}_{\ell,j \neq k}} \leq 0 \qquad
\forall \, k \in \{1, \dots, 4\}, \quad \forall \, \ell \in
\{\textit{nw}, w\}.
\label{monotonicity_conditiona}
\end{equation}
It follows from a mass conservation statement written at two
consecutive half interfaces $\Gamma^{(m)}_{k-1/2}$ and
$\Gamma^{(m)}_{k+1/2}$
that whenever the monotonicity constraint holds we also have
\begin{equation}
\frac{\partial (\overline{F}^{\, (m)}_{\ell,k+1/2} - \overline{F}^{\, (m)}_{\ell,k-1/2}) }{\partial S^{(m)}_{\ell,k}} \geq 0 \quad \forall \, \ell \in \{\textit{nw}, w\}.
\label{corollary_monotonicity_conditiona}
\end{equation}
A similar monotonicity condition is used in \cite{kozdon2011multidimensional,keilegavlen2012multidimensional}. 
In Section \ref{section_mathematical_properties}, we use this property to show that the saturation solution of the fully 
implicit scheme remains between physical 
bounds, 0 and 1.

\section{\label{section_multidimensional_implicit_hybrid_upwinding}Multidimensional Implicit Hybrid Upwinding}

In the following sections, we consider an interaction region $(m)$
and we detail the discretization of the interfacial quantities
in the four half fluxes. To construct the approximation, we
follow the Implicit Hybrid Upwinding (IHU) approach 
\citep{eymard1989hybrid,lee2015hybrid,lee2016c,lee2018hybrid,hamon2016analysis,hamon2016implicit,hamon2016capillary,moncorge2018consistent,xie2019unstructured}
and split the numerical flux at half interface $\Gamma_{k+1/2}$
between vertices $k$ and $k+1$ into a viscous part,
$\overline{V}_{\ell,k+1/2}$, and a buoyancy part, $\overline{G}_{\ell,k+1/2}$,
as follows 
\begin{eqnarray}
\overline{F}_{\ell,k+1/2} &=& \overline{V}_{\ell,k+1/2} + \overline{G}_{\ell,k+1/2}  \\
& \approx & \int_{\Gamma^{}_{k+1/2}} \frac{\lambda_{\ell}}{\lambda_T}  \mathbf{u}_T(\mathbf{x}, p, S) \cdot \boldsymbol{n}^{}_{k+1/2} \, \text{d} \Gamma^{}_{k+1/2} \nonumber \\
& + & \sum_m \int_{\Gamma^{}_{k+1/2}} k  \frac{\lambda_m \lambda_{\ell}}{\lambda_T} (\mathbf{g}_{\ell} - \mathbf{g}_m) \cdot \boldsymbol{n}^{}_{k+1/2} \, \text{d} \Gamma^{}_{k+1/2},
\label{viscous_buoyancy_split}
\end{eqnarray}
where we omitted the superscript denoting the interaction region.
The viscous and buoyancy parts are then evaluated separately,
based on physical considerations, to achieve a differentiable flux
when the total velocity field is fixed. The proposed discretization of
$\overline{V}_{\ell,k+1/2}$ and $\overline{G}_{\ell,k+1/2}$, explained next, 
attenuates grid orientation effects when the flow is not aligned
with the computational grid, but reduces to the original IHU when
the flow is aligned with the computational grid. We show in the
next sections that it satisfies the monotonicity property outlined
above, and that it preserves a robust nonlinear convergence behavior
even in the presence of strong buoyancy forces.

Multiple strategies have been used to reflect the local flow
pattern in the construction of multidimensional upwind 
schemes for porous media flow. In the IMPES 
and fully implicit multidimensional PPU schemes of  
\cite{kozdon2011multidimensional,kozdon2011batstoicmame,keilegavlen2012multidimensional}, 
the authors design a flux approximation based on
(\ref{governing_pdes_batstoiplus}) by directly computing
weighted averages of mobilities to account for the flow 
orientation in each interaction region. Instead, in the 
IMPES fractional-flow based schemes of
\cite{lamine2010higher,lamine2013higher,lamine2015multidimensional,edwards2011multi}, 
the flow orientation is exploited in the computation of
weighted averages of saturations -- which are then used to 
evaluate mobility ratios -- and of weighted averages of fluxes. 
Our methodology is closer to the latter. Given the 
structure of the Implicit Hybrid Upwinding flux, we employ
weighted averages of mobility ratios to obtain a fully implicit
multidimensional IHU satisfying the monotonicity property. 

\subsection{\label{subsection_viscous_term}Viscous term}

In the viscous term $\overline{V}_{\ell,k+1/2}$, the mobility ratio
at half interface $\Gamma_{k+1/2}$ is approximated based on the
sign of the discrete total velocities in the interaction region.
In this section, we assume that these total
velocities have been computed and are available. Their fully implicit
discretization will be reviewed in Section 
\ref{subsection_total_velocity_discretization}. 
We write the viscous flux at half interface $\Gamma_{k+1/2}$ as
\begin{equation}
\overline{V}_{\ell,k+1/2} := 
\overline{\chi}_{\ell,k+1/2} \overline{u}_{T,k+1/2}.
\label{viscous_half_numerical_flux}
\end{equation}
The interfacial quantity $\overline{\chi}_{\ell,k+1/2}$ approximates
the mobility ratio of the viscous term at half interface
$\Gamma_{k+1/2}$ (see Fig.~\ref{fig:viscous_term}) using a
weighted average of the mobility ratio
evaluated in the control volume and at the previous interface to
reflect the local flow pattern in the interaction region, that is,
\begin{equation}
\overline{\chi}_{\ell,k+1/2} := \left\{
\begin{array}{l l}
  (1-\overline{\omega}^V_{k+1/2}) \, \displaystyle \chi_{\ell,k}
  \, + \,
  \overline{\omega}^V_{k+1/2} \, \overline{\chi}_{\ell,k-1/2}
  & \text{if } \, \overline{u}_{T,k+1/2} > 0 \\[10pt]
  (1-\overline{\omega}^V_{k+1/2}) \, \displaystyle \chi_{\ell,k+1}
  \, + \,
  \overline{\omega}^V_{k+1/2} \, \overline{\chi}_{\ell,k+3/2}
  & \text{otherwise,} \\
\end{array} \right. 
\label{definition_chi}
\end{equation}
where $\chi_{\ell,k}$ denotes the mobility ratio evaluated
at vertex $k$:
\begin{equation}
\chi_{\ell,k} := \frac{\lambda_{\ell,k}}{\lambda_{T,k}} = \frac{\lambda_{\ell}(S_{\ell,k})}{\lambda_{T}(S_{\ell,k})}.
\label{vertex_based_mobility_ratio}
\end{equation}
The weighting coefficient $\overline{\omega}^V_{k+1/2}$ is a
function of the primary variables -- pressure and saturation --
in the interaction region and is constructed as the ratio of
the total velocities at two consecutive interfaces:
\begin{equation}
\overline{\omega}^V_{k+1/2} := \left\{
\begin{array}{l l}
\varphi \bigg( \max \big( 0, \displaystyle \frac{\overline{u}_{T,k-1/2}}{\overline{u}_{T,k+1/2}}\big) \bigg)  & \text{if} \, \, \overline{u}_{T,k+1/2} > 0 \\[10pt]
\varphi \bigg( \max \big( 0, \displaystyle \frac{\overline{u}_{T,k+3/2}}{\overline{u}_{T,k+1/2}}\big) \bigg)  & \text{if} \, \, \overline{u}_{T,k+1/2} < 0 \\[10pt]
0                 & \text{otherwise,}
\end{array} \right. 
\label{definition_omega_v}
\end{equation}
where $\varphi: \mathbb{R}^+ \mapsto [0,1]$ is a limiter function
used to guarantee that $\overline{\omega}^V_{k+1/2} \in [0,1]$
and enforce monotonicity with respect to saturation. With this
definition, it follows that the multidimensional viscous term
(\ref{viscous_half_numerical_flux}) reduces 
to the viscous term in the original one-dimensional IHU scheme
whenever $\varphi \equiv 0$. An illustration of
the viscous upwinding is given in Fig.~\ref{fig:viscous_term}
where $\overline{u}_{T,1+1/2} > 0$, $\overline{u}_{T,2+1/2} > 0$,
$\overline{u}_{T,3+1/2} < 0$, and $\overline{u}_{T,4+1/2}$.

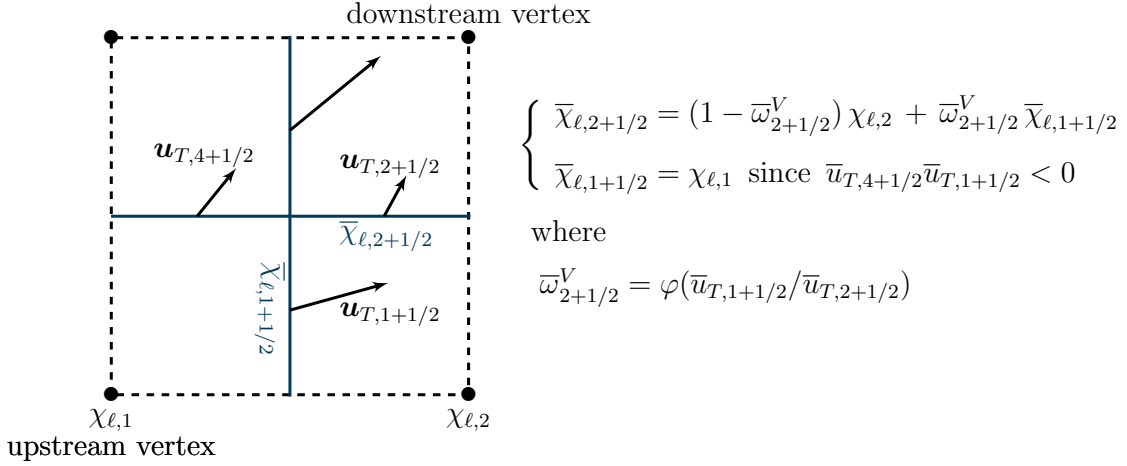
\begin{figure}[ht]
\centering
\scalebox{0.95}{
\tikzstyle{int}=[draw, minimum size=2em]
\tikzstyle{init} = [pin edge={to-,thick,black}]

\begin{tikzpicture}[node distance=3cm,auto,>=latex']


  \node (c) at (5,0) {\Large $\bullet$}; 
  \node (c) at (0,5) {\Large $\bullet$}; 
  \path (0,0) node (a) {\Large $\bullet$};
  \path (5,5) node (b) {\Large $\bullet$};
  \path [draw=black,dashed,very thick] (a) rectangle (b); 

  \node (a) at (2.5,-0.15) {}; 
  \node (b) at (2.5,5.15) {}; 
  \path[draw=berkeley_lab_blue,very thick] (a) edge node {} (b);

  \node (a) at (-0.1275,2.5) {}; 
  \node (b) at (5.15,2.5) {}; 
  \path[draw=berkeley_lab_blue,very thick] (a) edge node {} (b);

  \node (c) at (0,-0.35) {\large $\chi_{\ell,1}$};
  \node (c) at (0,-0.75) {\large upstream vertex};
  \node (c) at (5,-0.35) {\large $\chi_{\ell,2}$};
  \node (c) at (0,-0.75) {\large upstream vertex};
  \node (c) at (5,5.35) {\large downstream vertex};

  \node (c) at (3.85,2.225) {\large \textcolor{berkeley_lab_blue}{$\overline{\chi}_{\ell,2+1/2}$}};
  \node[rotate=-90] (c) at (2.225,1.2) {\large \textcolor{berkeley_lab_blue}{$\overline{\chi}_{\ell,1+1/2}$}};

  \node (a) at (2.375,3.6) {}; 
  \node (b) at (3.9 ,4.85) {}; 
  \path[->,draw=black,very thick] (a) edge node {} (b);  

  \node (a) at (1.1,2.375) {}; 
  \node (b) at (1.85,3.3) {}; 
  \node (c) at (1.3,3.35) {\large $\boldsymbol{u}_{T,4 + 1/2}$};
  \path[->,draw=black,very thick] (a) edge node {} (b);  

  \node (a) at (2.375,1.15) {}; 
  \node (b) at (4 ,1.6) {};
  \node (c) at (3.9,1.125) {\large $\boldsymbol{u}_{T,1 + 1/2}$};
  \path[->,draw=black,very thick] (a) edge node {} (b);  

  \node (a) at (3.75,2.375) {}; 
  \node (b) at (4.2,3.2) {};
  \node (c) at (3.9,3.2) {\large $\boldsymbol{u}_{T,2 + 1/2}$};
  \path[->,draw=black,very thick] (a) edge node {} (b);  

  \node (a) at (10,3.5) {\large $\left\{ \begin{array}{l l} \overline{\chi}_{\ell,2+1/2} = (1-\overline{\omega}^V_{2+1/2}) \, \displaystyle \chi_{\ell,2} \, + \, \overline{\omega}^V_{2+1/2} \, \overline{\chi}_{\ell,1+1/2} \\[10pt] \overline{\chi}_{\ell,1+1/2} =  \chi_{\ell,1} \, \, \, \text{since} \, \, \, \overline{u}_{T,4+1/2} \overline{u}_{T,1+1/2} < 0 \end{array} \right.$};
  \node (b) at (6.4,2.3) {\large where};
    \node (c) at (8.6,1.5) {\large $\overline{\omega}^V_{2+1/2} = \varphi( \overline{u}_{T,1+1/2} / \overline{u}_{T,2+1/2} )$};

\end{tikzpicture}
}
\caption{\label{fig:viscous_term}Viscous upwinding 
for half interfaces $\Gamma_{1+1/2}$ and $\Gamma_{2+1/2}$ in an 
interaction region. The arrows show the orientation of the total 
velocity at the half interfaces. The interfacial mobility, 
$\overline{\chi}_{\ell,k+1/2}$, and the vertex-based mobility ratio, 
$\chi_{\ell,k}$, are defined in (\ref{definition_chi}) and (\ref{vertex_based_mobility_ratio}), respectively.
}
\end{figure}

Different limiters have been
proposed in previous work, such as the Tight Multi-D Upstream 
(TMU) limiter \citep{schneider1986skewed,velasco2007quadrilateral}
defined by $\varphi^{\textit{TMU}}(r) = \min(1,r)$
and the Smooth Multi-D Upstream (SMU) limiter
\citep{velasco2007quadrilateral,kozdon2009robust} computed with
$\varphi^{\textit{SMU}}(r) = r / (1+r)$. For a constant (uniform)
flow field, the SMU limiter corresponds to the upwind scheme 
of \cite{koren1991low} and the TMU limiter corresponds to the
narrow scheme of \cite{roe1992optimum}. In the absence of buoyancy,
SMU aligns the numerical diffusion tensor with the flow direction
and TMU minimizes the transverse numerical diffusion
\citep{kozdon2009robust}. 
In numerical tests not included here for brevity we observed that
SMU does not sufficiently reduce transverse diffusion when combined
with MultiD-IHU. But, we also noted that TMU introduces a
discontinuity in the derivatives of the weight
$\overline{\omega}^V_{k+1/2}$ that undermines the nonlinear
convergence behavior of the MultiD-IHU scheme. Therefore, we
propose a limiter that reduces transverse diffusion almost as
well as TMU while preserving a robust nonlinear behavior.
This variant of the SMU limiter is defined as
\begin{equation}
\varphi^{\textit{SMU4}}(r) := 
\frac{r^4 + r^3 + r^2 + r}{r^4 + r^3 + r^2 + r + 1}.
\label{limiter}
\end{equation}
This limiter, referred to as fourth-order SMU (SMU4), is shown
in Fig. \ref{fig:limiters}. In the MultiD-IHU scheme, we use
the SMU4 limiter, that is, we set
$\varphi \equiv \varphi^{\textit{SMU4}}$.
We note that the SMU4 limiter satisfies the symmetry property
\begin{equation}
\varphi\big(\frac{1}{r}\big) = \frac{\varphi(r)}{r}.
\label{symmetry_property}
\end{equation}
This property, also satisfied by the SMU and TMU limiters,
ensures that backward-facing and forward-facing 
fluxes are treated the same way. It is also used to prove
the monotonicity of the buoyancy term introduced in
Section~\ref{subsection_buoyancy_term}.

\begin{figure}[ht]
\centering
\begin{tikzpicture}
\node[anchor=south west,inner sep=0] at (0,0) {\includegraphics[scale=0.35]{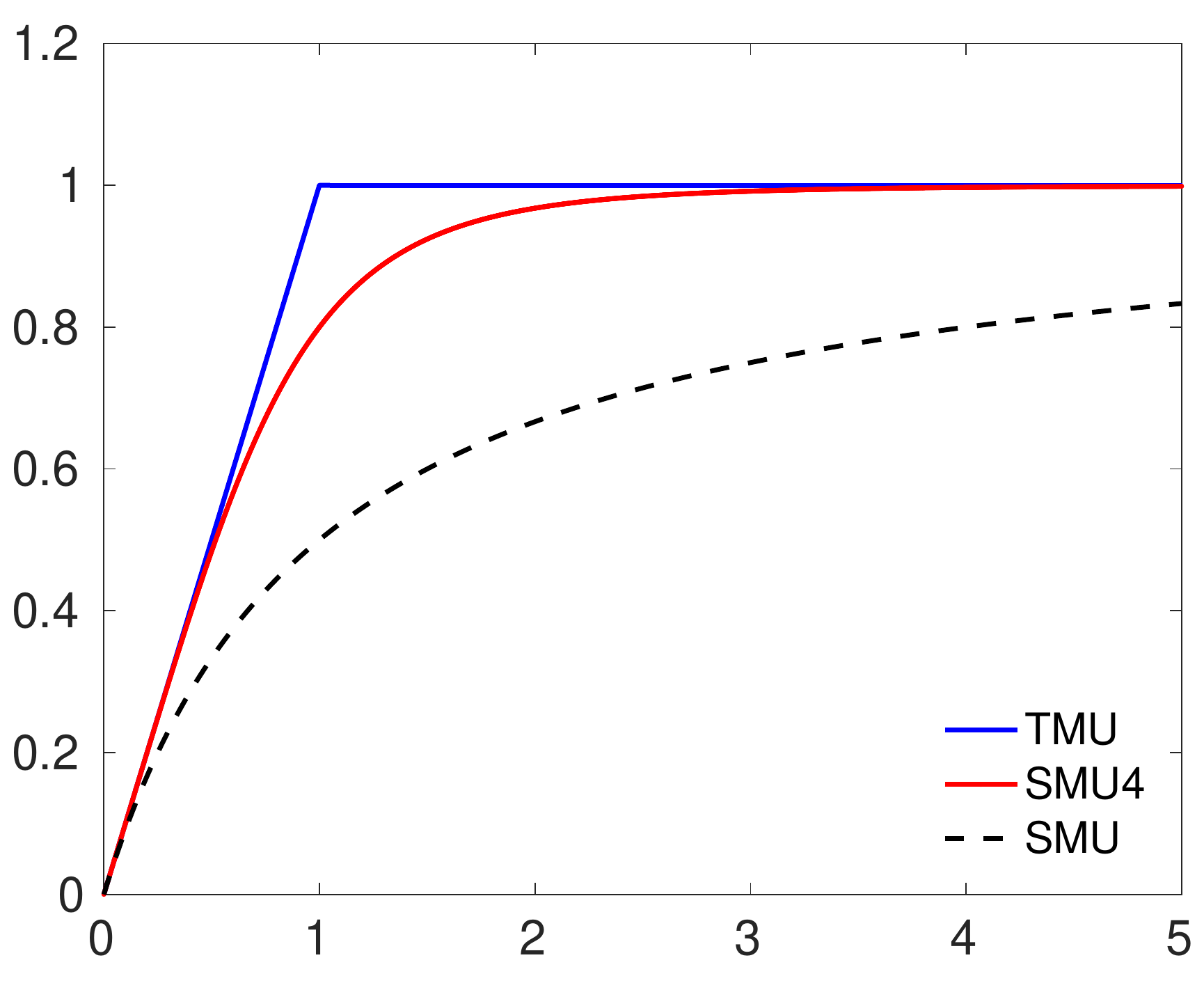}};
\node (ib_4) at (-0.32,2.75) {\small $\varphi(r)$};
\node (ib_4) at (3.3,0.1) {\small $r$};
\end{tikzpicture}
\caption{\label{fig:limiters}
Multidimensional limiters discussed in this work, including the Tight Multi-D Upstream (TMU) 
limiter of \cite{schneider1986skewed}, the Smooth Multi-D Upstream (SMU) limiter of 
\cite{velasco2007quadrilateral}, and the fourth-order SMU limiter of (\ref{limiter}).
}
\end{figure}

This methodology leads to a viscous flux that is a differentiable
function of the saturations in the interaction region for a fixed
total velocity field. We will show with numerical examples in 
Section~\ref{section_numerical_examples} that the scheme remains
well-behaved when the total velocity is a space- and time-dependent
function of the primary variables. Equations (\ref{definition_chi})
and (\ref{definition_omega_v}) introduce a local coupling of the
viscous terms in the interaction region that can be written in a
compact matrix form as 
\begin{equation}
\boldsymbol{A} \overline{\boldsymbol{\chi}}_{\ell} = \boldsymbol{B} \boldsymbol{\chi}_{\ell},
\label{definition_linear_system}
\end{equation}
where $\boldsymbol{A} = \{a_{ij}\} \in \mathbb{R}^{4 \times 4}$
has unit diagonal elements $(a_{kk} = 1)$ and one non-zero
off-diagonal term per row defined as
\begin{equation}
a_{k(k-1)} := \left\{
\begin{array}{l l}
- \overline{\omega}^V_{k+1/2}  & \text{if } \, \overline {u}_{T,k+1/2} \geq 0 \\[10pt]
0             & \text{otherwise.} \\
\end{array} \right. 
\qquad
a_{k(k+1)} := \left\{
\begin{array}{l l}
0  & \text{if } \, \overline {u}_{T,k+1/2} \geq 0 \\[10pt]
- \overline{\omega}^V_{k+1/2}  & \text{otherwise,} \\
\end{array} \right.
\label{definition_A}
\end{equation}
and $\boldsymbol{B} = \{b_{ij}\} \in \mathbb{R}^{4 \times 4}$ is
non-negative, such that
\begin{equation}
b_{kk} := \left\{
\begin{array}{l l}
1 - \overline{\omega}^V_{k+1/2}  & \text{if } \, \overline{u}_{T,k+1/2} \geq 0 \\[10pt]
0  & \text{otherwise.} \\
\end{array} \right. 
\qquad
b_{k(k+1)} := \left\{
\begin{array}{l l}
0  & \text{if } \, \overline {u}_{T,k+1/2} \geq 0 \\[10pt]
1 - \overline{\omega}^V_{k+1/2}  & \text{otherwise,} \\
\end{array} \right. 
\label{definition_B}
\end{equation}
In \ref{app_properties_of_matrix_A}, we use the definitions given
in (\ref{limiter}) to (\ref{definition_B}) to prove that
$\boldsymbol{A}$  is a non-singular M-matrix.
We also show
that the matrix
$\boldsymbol{C} := \boldsymbol{A}^{-1} \boldsymbol{B}$ is
non-negative and satisfies
\begin{equation}
  \sum_j c_{ij} = 1 \qquad i \, \in \, \{1, \dots, 4\}.
  \label{sum_c_coefs_equal_to_one}
\end{equation}
The interfacial mobility ratios are then determined by solving
the following local ($4 \times 4$) linear system
\begin{equation}
\overline{\boldsymbol{\chi}}_{\ell} = \boldsymbol{A}^{-1} \boldsymbol{B} \boldsymbol{\chi}_{\ell}.
\label{linear_solution_to_define_overline}  
\end{equation}
With this approach, each interfacial mobility ratio can depend on
more than two saturations in the interaction region. This leads to
an improved accuracy when the flow is not perpendicular to the
interface while preserving the monotonicity
of the viscous flux in the sense of (\ref{monotonicity_conditiona}),
as proven in \ref{app_viscous_flux_monotonicity}.
We mention
here that the matrices $\boldsymbol{A}$ and $\boldsymbol{B}$
are also used to obtain the derivatives of the mobility ratios
with respect to pressure and saturation as shown below,
\begin{equation}
\frac{\partial \overline{\boldsymbol{\chi}}_{\ell} }{\partial \tau_k} 
= \boldsymbol{A}^{-1} \big( 
- \frac{\partial \boldsymbol{A}}{\partial \tau_k} \overline{\boldsymbol{\chi}}_{\ell}
+ \frac{\partial \boldsymbol{B}}{\partial \tau_k} \boldsymbol{\chi}_{\ell} 
+ \boldsymbol{B} \frac{\partial \boldsymbol{\chi}_{\ell}}{\partial \tau_k} \big),
\end{equation}
where $\tau_k$ represents the pressure at vertex $k$,
$p_k$, or the wetting-phase saturation at vertex
$k$, $S_k$. The matrices $\partial \boldsymbol{A} / \partial \tau_k$
and $\partial \boldsymbol{B} / \partial \tau_k$ contain the
derivatives of the coefficients of $\boldsymbol{A}$ and
$\boldsymbol{B}$, respectively.
Finally, we highlight that at each half interface, we still
recover the discrete total velocity when we sum
of the viscous terms over the two phases. That is, for interface
$\Gamma_{k+1/2}$, we have
\begin{equation}
\sum_{\ell} \overline{V}_{\ell,k+1/2}
= \big( \sum_{\ell} \overline{\chi}_{\ell,k+1/2} \big)
\overline{u}_{T,k+1/2}
= \big( \sum_j c_{ij}  ( \sum_{\ell} \chi_{\ell,j} )
\big)
\overline{u}_{T,k+1/2}
= \overline{u}_{T,k+1/2},
\label{consistency_viscous_term}
\end{equation}
where the rightmost equality is obtained using
(\ref{vertex_based_mobility_ratio})
and (\ref{sum_c_coefs_equal_to_one}).

\subsection{\label{subsection_buoyancy_term}Buoyancy term}

In the buoyancy term $\overline{G}_{\ell,k+1/2}$, the discretization
of the mobility ratio at half interface $\Gamma_{k+1/2}$ is fully
based on (fixed) density differences as in the one-dimensional IHU
scheme. Our approach is based on the fact that the heaviest
(respectively, lightest) phase propagates downwards (respectively,
upwards). We write the buoyancy term as
\begin{equation}
\overline{G}_{\ell,k+1/2} := 
\overline{T}_{k+1/2} \sum_m \overline{\psi}_{\ell,m,k+1/2} (\rho_{\ell} - \rho_m) g \overline{\Delta z}_{k+1/2}, 
\label{buoyancy_half_numerical_flux}
\end{equation}
where $\overline{\Delta z}_{k+1/2} = z_{k+1} - z_{k}$, and where
$\overline{T}_{k+1/2}$
is defined in (\ref{rock_and_geometric_transmissibility}).
We omitted again the 
superscript denoting the interaction region in the right-hand side.
As in the viscous term, the discrete mobility ratio at half 
interface $\Gamma_{k+1/2}$ between vertices $k$ and $k+1$ is
obtained with a weighted averaging procedure that accounts for
the orientation of the buoyancy force with respect to the grid.
This procedure is described below for the case
$(\rho_{\ell} - \rho_m) \overline{\Delta z}_{k+1/2} > 0$:
\begin{equation}
\overline{\psi}_{\ell,m,k+1/2} 
:= \left\{
\begin{array}{l l} 
(1-\overline{\omega}^G_{k+1/2}) \displaystyle \frac{ \lambda_{\ell,k} \lambda_{m,k+1} }{ \lambda_{\ell,k} + \lambda_{m,k+1} } 
+ \overline{\omega}^G_{k+1/2}   \displaystyle \frac{ \lambda_{\ell,k} \lambda_{m,k+2} }{ \lambda_{\ell,k} + \lambda_{m,k+2} } &  \text{if} \, \overline{\Delta z}_{k+1/2}\overline{\Delta z}_{k+3/2} > 0 \\[10pt]
(1-\overline{\omega}^G_{k+1/2}) \displaystyle \frac{ \lambda_{\ell,k} \lambda_{m,k+1} }{ \lambda_{\ell,k} + \lambda_{m,k+1} } 
+ \overline{\omega}^G_{k+1/2}   \displaystyle \frac{ \lambda_{\ell,k-1} \lambda_{m,k+1} }{ \lambda_{\ell,k-1} + \lambda_{m,k+1} } & \text{if} \, \overline{\Delta z}_{k+1/2}\overline{\Delta z}_{k-1/2} > 0 \\[10pt]
\displaystyle \frac{ \lambda_{\ell,k} \lambda_{m,k+1} }{ \lambda_{\ell,k} + \lambda_{m,k+1} }  & \text{otherwise.}
\end{array} \right. \label{buoyancy_mobility_ratio}
\end{equation}
The following cases are illustrated in Fig.~\ref{fig:buoyancy_term}
for a light non-wetting phase and a heavy wetting phase.
First, we consider
updip half interface $\Gamma_{2+1/2}$
($\overline{\Delta z}_{2+1/2} < 0$). Since $\Gamma_{1+1/2}$ is
downdip and $\Gamma_{3+1/2}$ is updip, the non-wetting phase
mobility, $\lambda_{\textit{nw}}$,  is evaluated at the bottom
vertex 2 whereas the wetting-phase mobility $\lambda_w$ is evaluated
at vertex 3 and at the top vertex 4. Second, we consider updip
half interface $\Gamma_{3+1/2}$
($\overline{\Delta z}_{3+1/2} < 0$). Since $\Gamma_{4+1/2}$ is downdip
and $\Gamma_{2+1/2}$ is updip, $\lambda_{\textit{nw}}$ is evaluated
at vertex 3 and at the bottom vertex 2 while $\lambda_w$ is
evaluated at the top vertex 1.  

\begin{figure}[ht]
\centering
\scalebox{0.95}{
\tikzstyle{int}=[draw, minimum size=2em]
\tikzstyle{init} = [pin edge={to-,thick,black}]

\begin{tikzpicture}[node distance=3cm,auto,>=latex']


  \node (c) at (5,0) {\Large $\bullet$}; 
  \node (c) at (0,5) {\Large $\bullet$}; 
  \path (0,0) node (a) {\Large $\bullet$};
  \path (5,5) node (b) {\Large $\bullet$};
  \path [draw=black,dashed,very thick] (a) rectangle (b); 

  \node (a) at (2.5,-0.15) {}; 
  \node (b) at (2.5,5.15) {}; 
  \path[draw=berkeley_lab_blue,very thick] (a) edge node {} (b);

  \node (a) at (-0.1275,2.5) {}; 
  \node (b) at (5.15,2.5) {}; 
  \path[draw=berkeley_lab_blue,very thick] (a) edge node {} (b);

  \node (c) at (0,5.35) {\large $\lambda_{w,4}$};
  \node (c) at (5,5.35) {\large $\lambda_{w,3}$, $\lambda_{\textit{nw},3}$};
  \node (c) at (5,-0.35) {\large $\lambda_{\textit{nw},2}$};

  \node (c) at (0,5.75) {\large top vertex};
  \node (c) at (5,-0.75) {\large bottom vertex};

  \node (c) at (3.85,2.8) {\large \textcolor{berkeley_lab_blue}{$\overline{\psi}_{\textit{nw},w,2+1/2}$}};
  \node[rotate=90] (c) at (2.2,3.75) {\large \textcolor{berkeley_lab_blue}{$\overline{\psi}_{\textit{nw},w,3+1/2}$}};

  \node (a) at (2.375,3.95) {}; 
  \node (b) at (3.35 ,3.40) {}; 
  \node (c) at (3.425,3.9) {\large $\boldsymbol{g}_{3 + 1/2}$};
  \path[->,draw=black,very thick] (a) edge node {} (b);  

  \node (a) at (1.1,2.55) {}; 
  \node (b) at (2.075,2.0) {}; 
  \node (c) at (1.1,1.925) {\large $\boldsymbol{g}_{4 + 1/2}$};
  \path[->,draw=black,very thick] (a) edge node {} (b);  

  \node (a) at (2.375,1.4) {}; 
  \node (b) at (3.35 ,0.85) {};
  \node (c) at (3.425,0.65) {\large $\boldsymbol{g}_{1 + 1/2}$};
  \path[->,draw=black,very thick] (a) edge node {} (b);  

  \node (a) at (3.6,2.55) {}; 
  \node (b) at (4.575,2.0) {};
  \node (c) at (4.2,1.9) {\large $\boldsymbol{g}_{2 + 1/2}$};
  \path[->,draw=black,very thick] (a) edge node {} (b);  

  \node (x) at (10.5,4.55) {\large $\text{Since} \, \, (\rho_{\textit{nw}} - \rho_w) \overline{\Delta z}_{2+1/2} > 0 \, \, \text{and} \, \, (\rho_{\textit{nw}} - \rho_w) \overline{\Delta z}_{3+1/2} > 0$};
  \node (y) at (10.625,3.75) {\large $\text{and} \, \, \overline{\Delta z}_{1+1/2} > 0, \, \, \overline{\Delta z}_{2+1/2} < 0, \, \, \overline{\Delta z}_{3+1/2} < 0, \, \, \overline{\Delta z}_{4+1/2} > 0$};
  \node (a) at (10.75,2.45) {\large $\overline{\psi}_{\textit{nw},w,2+1/2} = (1-\overline{\omega}^G_{2+1/2}) \displaystyle \frac{ \lambda_{\textit{nw},2} \lambda_{w,3} }{ \lambda_{\textit{nw},2} + \lambda_{w,3} } + \overline{\omega}^G_{2+1/2}   \displaystyle \frac{ \lambda_{\textit{nw},2} \lambda_{w,4} }{ \lambda_{\textit{nw},2} + \lambda_{w,4} }$};

  \node (a) at (10.75,0.75) {\large $\overline{\psi}_{\textit{nw},w,3+1/2} = (1-\overline{\omega}^G_{3+1/2}) \displaystyle \frac{ \lambda_{\textit{nw},3} \lambda_{w,4} }{ \lambda_{\textit{nw},3} + \lambda_{w,4} } + \overline{\omega}^G_{3+1/2}   \displaystyle \frac{ \lambda_{\textit{nw},2} \lambda_{w,4} }{ \lambda_{\textit{nw},2} + \lambda_{w,4} }$};


\end{tikzpicture}
}
\caption{\label{fig:buoyancy_term}
Buoyancy upwinding for half interfaces
$\Gamma_{2+1/2}$ and $\Gamma_{3+1/2}$ in an interaction region. The arrows indicate
the gravity vector at each half interface. For this configuration,
vertex 4 is at the top and vertex 2 is at the bottom. Therefore,
we have
$\overline{\Delta z}_{1+1/2} = z_2 - z_1 > 0$,
$\overline{\Delta z}_{2+1/2} = z_3 - z_2 < 0$,
$\overline{\Delta z}_{3+1/2} = z_4 - z_3 < 0$, and
$\overline{\Delta z}_{4+1/2} = z_1 - z_4 > 0$.
We consider a light non-wetting phase denoted by the
subscript $\textit{nw}$ and a heavy wetting phase denoted 
by the subscript $w$, such that
$\rho_{\textit{nw}} - \rho_{w} < 0$. The interfacial mobility ratio, 
$\overline{\psi}_{\textit{nw},w,k+1/2}$, is defined in
(\ref{buoyancy_mobility_ratio}).
}
\end{figure}
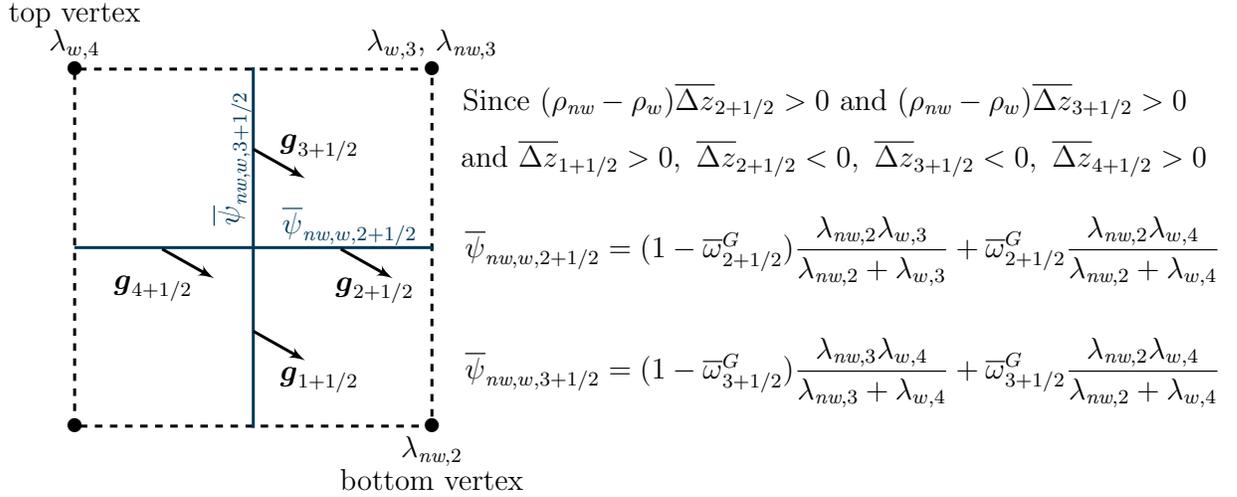

The case $(\rho_{\ell} - \rho_m) \overline{\Delta z}_{k+1/2} \leq 0$ is
treated analogously and is obtained by switching the phase
indices $\ell$ and $m$ in (\ref{buoyancy_mobility_ratio}).
Finally, the weighting coefficient $\overline{\omega}^G_{k+1/2}$
is written as
\begin{equation}
\overline{\omega}^G_{k+1/2} := \left\{
\begin{array}{l l}
\varphi \bigg( \displaystyle \frac{\overline{T}_{k+3/2} \overline{\Delta z}_{k+3/2} }{\overline{T}_{k+1/2} \overline{\Delta z}_{k+1/2} }  \bigg)  & \text{if} \, \overline{\Delta z}_{k+1/2}\overline{\Delta z}_{k+3/2} > 0 \\[10pt]
\varphi \bigg( \displaystyle \frac{\overline{T}_{k-1/2} \overline{\Delta z}_{k-1/2} }{\overline{T}_{k+1/2} \overline{\Delta z}_{k+1/2}}  \bigg)  & \text{if} \,  \overline{\Delta z}_{k+1/2}\overline{\Delta z}_{k-1/2} > 0 \\[10pt]
0                 & \text{otherwise,}
\end{array} \right. \label{definition_omega_g}
\end{equation}
where $\varphi$ is defined in (\ref{limiter}) and guarantees that
$\overline{\omega}^G_{k+1/2} \in [0,1]$. The multidimensional
buoyancy term reduces to the one-dimensional IHU buoyancy term
whenever $\varphi \equiv 0$. In the compressible case, the ratio of 
(\ref{definition_omega_g}) would also contain phase densities.
These densities cancel in the incompressible case considered in 
this work and are omitted for clarity.

Unlike the viscous term, the buoyancy term does not require a
local linear solve to determine the mobility ratio since
$\overline{\psi}_{\ell,m,k+1/2}$ is decoupled from
$\overline{\psi}_{\ell,m,k-1/2}$ and
$\overline{\psi}_{\ell,m,k+3/2}$. By construction,
$\overline{G}_{\ell,k+1/2}$ is a differentiable function of
the saturations, and its derivatives can be computed in a
straightforward manner. We show in
\ref{app_buoyancy_flux_monotonicity} that since the SMU4
limiter satisfies the symmetry property (\ref{symmetry_property}),
the buoyancy term constitutes a monotone approximation of
the buoyancy flux in the sense of (\ref{monotonicity_conditiona}).
Since the same property holds for the
viscous term, (\ref{viscous_buoyancy_split}) implies that
$\overline{F}_{\ell,k+1/2}$ satisfies the monotonicity condition
of (\ref{monotonicity_conditiona}). Finally, we note that by
construction, we have
$\overline{\psi}_{\ell,m,k+1/2} = \overline{\psi}_{m,\ell,k+1/2}$.
It follows that
\begin{equation}
  \overline{G}_{\ell,k+1/2} = -\overline{G}_{m,k+1/2}.
  \label{consistency_buoyancy_term}
\end{equation}

\subsection{\label{subsection_total_velocity_discretization}Total velocity discretization}

In the previous paragraphs, the presentation has focused on the
discretization of the mobilities present in the transport equation.
To complete the formulation of the multidimensional IHU scheme, we
now describe the computation of the total velocity.
Before this,
we note that (\ref{discrete_pressure_equation}) does not provide
a practical way to compute $\overline{u}_{T,k+1/2}$, as this
quantity is needed to evaluate the phase fluxes present in the
right-hand side. Instead, equation (\ref{discrete_pressure_equation})
imposes a consistency requirement on the transport scheme. This
consistency is satisfied by the methodology presented in
Sections~\ref{subsection_viscous_term}-\ref{subsection_buoyancy_term},
as shown by (\ref{consistency_viscous_term}) and
(\ref{consistency_buoyancy_term}).

Considering a half interface $\Gamma_{k+1/2}$
between vertices $k$ and $k+1$, we compute the total velocity, 
$\overline{u}_{T,k+1/2}$, as
\begin{equation}
\overline{u}_{T,k+1/2}( \overline{\Delta p}_{k+1/2},S_k,S_{k+1}) 
:=  
\overline{T}_{k+1/2} \overline{\lambda}^{\textit{WA}}_{T,k+1/2}
\overline{\Delta p}_{k+1/2} 
+
\overline{T}_{k+1/2}
\sum_{\ell} \overline{\lambda}^{\textit{WA}}_{\ell,k+1/2}
\rho_{\ell} g \overline{\Delta z}_{k+1/2},
 \label{discrete_total_velocity_v_g_pc}
\end{equation}
where $\overline{\lambda}_{\ell,k+1/2}^{\textit{WA}}$
denotes the discrete mobility of phase $\ell$ at half
interface $\Gamma_{k+1/2}$. To compute
this quantity, we apply a differentiable discretization
derived in \cite{hamon2016implicit} that aims at
attenuating the flip-flopping in the total velocity
to improve the nonlinear convergence of the scheme.
This methodology is based on an averaging procedure in
which the mobility at half interface $\Gamma_{k+1/2}$
in (\ref{discrete_total_velocity_v_g_pc}) is defined as 
\begin{equation}
\overline{\lambda}^{\textit{WA}}_{\ell,k+1/2} 
:= 
\overline{\beta}_{k+1/2}( \overline{\Delta p}_{k+1/2})
\lambda_\ell(S_k) 
\, \, +  \, \, 
\big(1- \overline{\beta}_{k+1/2}( \overline{\Delta p}_{k+1/2}) \big)
\lambda_\ell(S_{k+1}).
\label{phase_mobility_weighted_average}
\end{equation}
The weighting coefficient $\overline{\beta}_{k+1/2} \in [0,1]$
is designed such that the resulting phase mobilities are
differentiable and increasing functions of the pressure
difference that adapt to the balance between viscous and
buoyancy forces at each interface in the computational
domain. It reads
\begin{equation}
\overline{\beta}_{k+1/2}( \overline{\Delta p}_{k+1/2})
:= \frac{1}{2} + \frac{1}{\pi} 
\arctan( \overline{\gamma}_{k+1/2} \overline{ \Delta \Phi}_{k+1/2}).
\label{appendix_phase_mobility_weighted_average}
\end{equation}
The coefficient $\overline{\gamma}_{k+1/2} \in \mathbb{R}^+$  relates
the magnitude of viscous forces to the magnitude of buoyancy
forces and is given in \cite{hamon2016implicit}.
In (\ref{appendix_phase_mobility_weighted_average}),
$\overline{\Delta \Phi}_{k+1/2}$ denotes the average
potential gradient at the interface, defined by:
\begin{equation}
\overline{\Delta \Phi}_{k+1/2} :=
\frac{1}{n_p} \sum_{\ell} \overline{\Delta \Phi}_{\ell,k+1/2} 
= 
\overline{\Delta p}_{k+1/2}
+
\frac{1}{n_p} \sum_{\ell} \rho_{\ell} g \overline{\Delta z}_{k+1/2}. 
\label{tranche_de_pain}
\end{equation}
We emphasize the fact that the only difference
between the standard PPU discretization of the total velocity used,
for instance, in \cite{lee2018hybrid}, and our methodology is
the discretization of the mobilities described in
(\ref{phase_mobility_weighted_average}) to (\ref{tranche_de_pain}).
In previous work \citep{hamon2016implicit}, we found that
the nonlinear behavior improves significantly when the
mobilities are discretized with
(\ref{phase_mobility_weighted_average}) to
(\ref{tranche_de_pain}) and we proved that the
resulting discrete pressure equation remains well behaved with
this approach.

\section{\label{section_mathematical_properties}Mathematical properties of the scheme}

Here, we analyze the fully implicit scheme for coupled flow and
transport constructed above. We consider the system of governing
equations given in Section~\ref{section_numerical_scheme}, with a
discrete transport equation
(\ref{discrete_transport_equation}) for one of the phases (here,
the wetting phase $\ell = w$) and a discrete pressure equation
(\ref{discrete_pressure_equation}).
The total velocity is discretized as explained in
Section~\ref{subsection_total_velocity_discretization}. Next, we use
the monotonicity condition satisfied by the multidimensional IHU
numerical flux to show that the saturation solution remains between
0 and 1.

\begin{myprop}{(Saturation estimate)}
Consider the pressure-saturation solution to the fully implicit
scheme defined in (\ref{discrete_transport_equation})-(\ref{discrete_pressure_equation}) with $q_{\ell} = 0$. Provided that the initial
saturation is between physical bounds, we have the following
saturation estimate, valid for all interaction regions in the
domain and all time levels $n \in \mathbb{N}^+$:
\begin{equation}
0 \leq \big( S^{(m)}_{k} \big)^{\, n} \leq 1,
\label{nonlinear_elliptic}
\end{equation}
where $k$ denotes the local index of the vertices in interaction
region $(m)$.
\end{myprop}

\begin{proof}
The proof is done by induction on the time level $n$ and relies
on the monotonicity of the flux in the sense of 
(\ref{monotonicity_conditiona}). We adapt a proof from
\cite{brenner2013finite} also used in \cite{hamon2016capillary}.

We assume that (\ref{nonlinear_elliptic}) holds for all
interaction regions for time levels $n \leq n_0$. To prove
that the solution of the scheme also satisfies this inequality
at time level $n_0 + 1$, we first consider the blue control
volume of Fig.~\ref{fig:interaction_regions}. Unless stated
otherwise, all the quantities in the proof are evaluated fully
implicitly at time level $n_0 + 1$. In interaction
region (2), the net wetting-phase flux at the two half interfaces
$\Gamma^{(2)}_{3+1/2}$ and $\Gamma^{(2)}_{4+1/2}$ is a function
of the total velocity field in interaction region (2), denoted by
$\overline{u}^{(2)}_{T,k+1/2}$, $k \in \{1, \dots, 4\}$, and the
saturations used in the evaluation of the viscous and buoyancy
mobility ratios of (\ref{definition_chi}) and
(\ref{buoyancy_mobility_ratio}), denoted by $S^{(2)}_{k}$.
Isolating the saturation at the center of the blue control volume
-- or, equivalently, the saturation at vertex 4 in interaction
region (2) -- we write the net flux at $\Gamma^{(2)}_{3+1/2}$ and
$\Gamma^{(2)}_{4+1/2}$, denoted by $\overline{F}^{\, (2)}$, as
\begin{equation}
\overline{F}^{\, (2)}_w :=
\bigg(
\overline{F}^{\, (2)}_{w,4+1/2}
-
\overline{F}^{\, (2)}_{w,3+1/2}
\bigg)
\big(
\overline{u}^{(2)}_{T,l+1/2,l \in \{1,\dots,4\}},
S^{(2)}_{4},
S^{(2)}_{l \neq 4}
\big).
\label{net_flux_definition}
\end{equation}
Now, we take the maximum between $S^{(2)}_{l}$ and 1 at vertices
$l \neq 4$ in (\ref{net_flux_definition}). Using the
monotonicity of the net flux given in (\ref{monotonicity_conditiona}),
we obtain
\begin{equation}
\bigg(
\overline{F}^{\, (2)}_{w,4+1/2}
-
\overline{F}^{\, (2)}_{w,3+1/2}
\bigg)
\big(
\overline{u}^{(2)}_{T,l+1/2,l \in \{1,\dots,4\}},
S^{(2)}_{4},
\max( S^{(2)}_{l}, 1)_{l\neq4}
\big)
\leq
\overline{F}^{\, (2)}_w.
\label{transport_inequality_flux}
\end{equation}
We note that an analogous result holds for interaction regions
$(1)$, $(3)$, and $(4)$. Now, we consider the accumulation term
in the blue control volume of Fig.~\ref{fig:interaction_regions},
still viewed from interaction region (2). Using the induction
hypothesis, one can write
\begin{equation}
  V \phi \displaystyle \frac{S^{(2)}_{4}-1}{\Delta t^n}
  \leq
  V \phi \displaystyle \frac{S^{(2)}_{4}-(S^{(2)}_{4})^{n_0} }{\Delta t^n}.
\label{transport_inequality_accum}
\end{equation}
Since we are considering the solution of the system at time $n_0+1$,
it satisfies the discrete transport equation
(\ref{discrete_transport_equation}). Using this remark, we combine
(\ref{transport_inequality_flux}) written in the four interaction
regions and (\ref{transport_inequality_accum}) to obtain
\begin{equation}
V \phi \displaystyle \frac{S^{(2)}_{4}-1}{\Delta t^n}
+
\sum_k
\bigg(
\overline{F}^{\, (k-2)}_{w,k+1/2}
-
\overline{F}^{\, (k-2)}_{w,k-1/2}
\bigg)
\big(
\overline{u}^{(k-2)}_{T,l+1/2,l \in \{1,\dots,4\}},
S^{(k-2)}_{k},
\max( S^{(k-2)}_{l}, 1)_{l\neq k}
\big)
\leq 0.
\label{transport_equation_after_maximum}
\end{equation}
Before proceeding to the next step, we remind the reader that
in the notation detailed in Fig.~\ref{fig:interaction_regions},
$S^{(1)}_{3}$, $S^{(2)}_{4}$, $S^{(3)}_{1}$, and
$S^{(2)}_{4}$ denote the same degree of freedom, $S^{(k-2)}_{k}$,
which refers to the wetting-phase saturation at the
center of the control volume in blue.

Next, we consider the discrete pressure equation in the blue control
volume of Fig.~\ref{fig:interaction_regions}. First, assume that for
the total velocity field satisfying the discrete pressure equation
(\ref{discrete_pressure_equation}), we evaluate the viscous mobility
ratios of (\ref{definition_chi}) and the buoyancy mobility ratio of
(\ref{buoyancy_mobility_ratio}) with unit wetting-phase saturations
in the interaction region, i.e., $S^{(2)}_{l} = 1$ for
$l \in \{1, \dots, 4\}$. In this configuration, we have
$\overline{\chi}^{(2)}_{w,l+1/2} = 1$
and
$\overline{\psi}^{(2)}_{w,l+1/2} = 0$ for $l \in \{1, \dots, 4\}$.
This implies that
\begin{align}
\bigg(
\overline{F}^{\, (2)}_{w,4+1/2}
-
\overline{F}^{\, (2)}_{w,3+1/2}
\bigg)
\big(
\overline{u}^{(2)}_{T,l+1/2,l \in \{1,\dots,4\}},
1,
(1)_{l\neq4}
\big)
&=
\overline{\chi}^{(2)}_{w,4+1/2} \overline{u}^{(2)}_{T,4+1/2}
-
\overline{\chi}^{(2)}_{w,3+1/2} \overline{u}^{(2)}_{T,3+1/2}
\nonumber \\
&+
\overline{\psi}^{(2)}_{w,\textit{nw},4+1/2} 
-
\overline{\psi}^{(2)}_{w,\textit{nw},3+1/2} 
\nonumber \\
&= \overline{u}^{(2)}_{T,4+1/2} - \overline{u}^{(2)}_{T,3+1/2}.
\label{pressure_inequality_total_velocity}
\end{align}
Next, we resort to the procedure employed to obtain
(\ref{transport_inequality_flux}). That is, in
(\ref{pressure_inequality_total_velocity}), we take the maximum
between $S^{(2)}_{l}$ and 1 at vertices $l \neq 4$ and use
the monotonicity of the numerical flux in the sense of
(\ref{monotonicity_conditiona}) to obtain
\begin{equation}
\bigg(
\overline{F}^{\, (2)}_{w,4+1/2}
-
\overline{F}^{\, (2)}_{w,3+1/2}
\bigg)
\big(
\overline{u}^{(2)}_{T,l+1/2,l \in \{1,\dots,4\}},
1,
\max(S^{(2)}_{l},1)_{l\neq4}
\big)
\leq
\overline{u}^{(2)}_{T,4+1/2} - \overline{u}^{(2)}_{T,3+1/2}.
\end{equation}
An analogous inequality holds for interaction regions $(1)$, $(3)$,
and $(4)$, respectively. Since the total velocity field satisfies
the discrete pressure equation (\ref{discrete_pressure_equation}),
we obtain
\begin{equation}
\sum_{k}
\bigg(
\overline{F}^{\, (k-2)}_{w,k+1/2}
-
\overline{F}^{\, (k-2)}_{w,k-1/2}
\bigg)
\big(
\overline{u}^{(k-2)}_{T,l+1/2,l \in \{1,\dots,4\}},
1,
\max(S^{(k-2)}_{l},1)_{l\neq k}
\big)
\leq 0.
\label{pressure_equation_alt_form}
\end{equation}

We now consider (\ref{transport_equation_after_maximum}) and
(\ref{pressure_equation_alt_form}). Since $S^{(k-2)}_{2}$ for
different $k \in \{1, \dots, 4\}$ refers to the same degree of
freedom,
that is, the saturation at the center of the blue control volume
in Fig.~\ref{fig:interaction_regions}, we see that
(\ref{transport_equation_after_maximum}) and
(\ref{pressure_equation_alt_form}) only differ by one argument in
the flux term -- in the second slot -- and in the accumulation term
-- in the saturation at the most recent time level. This argument
is equal to $S^{(2)}_{4}$ in
(\ref{transport_equation_after_maximum}) and to 1 in
(\ref{pressure_equation_alt_form}). In the next equation, we combine
(\ref{transport_equation_after_maximum}) and
(\ref{pressure_equation_alt_form}) by taking the maximum of
$S^{(2)}_{4}$ and 1 which leads to
\begin{align}
  & V \phi 
  \displaystyle \frac{\max( S^{(2)}_{4},1) -1 }{\Delta t^n}
  \qquad \qquad \nonumber \\
  & \quad + \quad \sum_{k}
  \bigg(
  \overline{F}^{\, (k-2)}_{w,k}
  -
  \overline{F}^{\, (k-2)}_{w,k-1}
  \bigg)
  \big(
  \overline{u}^{(k-2)}_{T,l+1/2,l \in \{1,\dots,4\}},
  \max( S^{(k-2)}_{k},1),
  \max(S^{(k-2)}_{l},1)_{l\neq k}
  \big)  
  \leq 0. \label{pressure_equation_transport_equation_after_maximum}
\end{align}
We see that we recover
(\ref{transport_equation_after_maximum}) whenever $S^{(2)}_{4} > 1$,
and we recover (\ref{pressure_equation_alt_form})
whenever $S^{(2)}_{4} \leq 1$. To obtain
(\ref{pressure_equation_transport_equation_after_maximum}), we have
considered the blue control volume in
Fig.~\ref{fig:interaction_regions}, but this inequality is valid for
all the control volumes in the domain. Therefore, we can sum this
inequality over all control volumes in the domain. In the result
inequality, we slightly abuse the notation and now denote by $S_{i}$
the saturation at the center of control volume $i$. Since the
scheme is mass conservative and the fluxes are evaluated with the
same arguments, the flux terms cancel and we are left with
\begin{equation}
\sum_{i} V_{i} \phi_{i} \displaystyle \frac{\max( S_{i},1) -1}{\Delta t^n} \leq 0.
\end{equation}
which gives $S_{i} \leq 1$ for all control volumes. Therefore, we
have proven the rightmost inequality in (\ref{nonlinear_elliptic}).
The proof for $S_{i} \geq 0$ is analogous.
\end{proof}

\section{\label{section_nonlinear_solver}Newton's method with damping}

The systems of nonlinear algebraic equations of
(\ref{discrete_transport_equation})-(\ref{discrete_pressure_equation})
are solved with Newton's method with damping
\citep{deuflhard2011newton}. Successive linearizations and updates
are performed until convergence is reached:
\begin{eqnarray} 
  \text{solve} & & \boldsymbol{J} \, \delta \boldsymbol{U}^{\nu+1} = -  \boldsymbol{R}( \boldsymbol{U}^{\, n+1,\nu}) \text{ for } \delta \boldsymbol{U}^{\nu+1}, 
  \label{damped_newton_linear_solve} \\ 
  \text{then} & & \boldsymbol{U}^{\, n+1,\nu+1} \leftarrow \boldsymbol{U}^{\, n+1,\nu} + \boldsymbol{\tau}^{\nu+1} \delta \boldsymbol{U}^{\nu+1}, 
  \label{damped_newton_solution_update}
\end{eqnarray}
where $n$ denotes the time level, $\nu$ denotes the nonlinear
iteration number,
$\boldsymbol{\tau}^{\nu+1} \in \mathbb{R}^{Nn_p \times Nn_p}$ is
a diagonal matrix of damping parameters, and
$\boldsymbol{J} \in \mathbb{R}^{Nn_p \times Nn_p}$ denotes the
Jacobian matrix of $\boldsymbol{R}$ with respect to the primary
variables.
In (\ref{damped_newton_solution_update}), the diagonal term of a
pressure row of $\boldsymbol{\tau}^{\nu+1} \in ]0,1]$ is equal
to one. The diagonal term of a saturation row is equal to one if
the local absolute saturation update, $|\delta S^{\nu+1}|$, is
smaller than $(\delta S)_{\textit{max}} := 0.2$. Otherwise, this
diagonal term is set to
$(\delta S)_{\textit{max}}/|\delta S^{\nu+1}|$.
We consider the control volume in blue in
Fig.~\ref{fig:interaction_regions}, in which 
the saturation of phase $\ell$ viewed from interaction
region (2) is denoted by $S^{(2)}_{\ell,4}$.
The convergence criterion in this control volume is satisfied when
the value of the normalized residual drops below the tolerance for
the two phases:
\begin{equation}
   \frac{V \phi ( S^{(2)}_{\ell,4}-(S^{(2)}_{\ell,4})^n )
  + \Delta t^n \big( \sum_{k} \big( \overline{F}^{\, (k-2)}_{\ell,k+1/2} - \overline{F}^{\, (k-2)}_{\ell,k-1/2} \big) - V q_{\ell} \big)}{V \phi}  \leq 10^{-8}, \qquad \ell \in \{ \textit{nw}, w\}.
\label{convergence_criterion}
\end{equation}
Convergence is reached when
(\ref{convergence_criterion}) is satisfied in all control volumes.
If the solver fails to converge after 50 iterations, Newton's method
is restarted from the previous time step, with a time step size
reduced by half. If necessary, this time step chopping is repeated.
After a converged solution is obtained, the time step size is reset
to its original size.

\section{\label{section_numerical_examples}Numerical examples}

In this section, we compare the accuracy and performance of four fully implicit schemes. The first scheme,
referred to as 1D-PPU, is based on the standard two-point phase-per-phase upwinding commonly used
in industrial simulators. 1D-IHU is the scheme proposed in \cite{hamon2016implicit}. It uses a two-point hybrid upwinding
combined with the procedure based on a weighted averaging of the mobilities in the total velocity. 
MultiD-PPU is the multidimensional scheme described in \cite{kozdon2011multidimensional}
constructed with the SMU limiter. Finally, MultiD-IHU is the multidimensional scheme based on 
hybrid upwinding described in Section \ref{section_numerical_scheme}. It relies on the SMU4 limiter for 
the viscous and buoyancy fluxes.

\subsection{\label{subsection_three_well_problem_with_buoyancy}Three-well problem with buoyancy}

We first assess the reduction of the numerical bias caused by the orientation of the grid
with a three-well problem with buoyancy derived from \cite{kozdon2011multidimensional,keilegavlen2012multidimensional}.
We consider a $x-y$ domain of dimensions
$[-0.5 \, \text{ft}, 0.5 \, \text{ft}] \times [-0.5 \, \text{ft}, 0.5 \, \text{ft}]$, discretized
with a $51 \times 51$ uniform Cartesian grid so that $\Delta x = \Delta y = 1/51 \, \text{ft}$. The scalar
absolute permeability is equal to
\begin{equation}
k(x,y) = \left\{
\begin{array}{l l}
50 \, \text{mD}  & \text{if } r = \sqrt{x^2 + y^2} < r_0 \\
5 \times 10^{-5} \, \text{mD} & \text{otherwise,} 
\end{array} \right. 
\end{equation}     %
where $r_0 = (1 - \Delta x)/2$. The well locations are computed using a secondary coordinate system defined as
$(x',y') = (x \cos\theta + y \sin\theta, -x \sin\theta + y \cos\theta)$. The injector is
always in the center of the domain at $(x',y') = (0,0)$ and is operated at a fixed
rate. The two producers are placed at $(x',y') = (\pm 0.3 \sin(\pi/6), - 0.3 \cos(\pi/6) )$
and are both operated using a bottom-hole-pressure control. The symmetry of the flow pattern is preserved by orienting
the buoyancy force in the direction that is perpendicular to the straight line connecting the two producers, going updip 
from the injector to the producers. 
The phase properties are such that the non-wetting phase is twice as light as the wetting phase, with
$\rho_{\textit{nw}} = 32 \, \text{lb}_m . \text{ft}^{-3}$
and $\rho_w = 64 \, \text{lb}_m . \text{ft}^{-3}$.
The constant phase viscosities are chosen to be $\mu_{\textit{nw}} = 100 \, \text{cP}$ and $\mu_w = 1 \, \text{cP}$. Finally,
we use Corey-type relative permeabilities such that $k_{r\textit{nw}}(S) = (1-S)^4$ and $k_{rw}(S) = S^2$. 
The gravity number is such that
\begin{equation}
N_G := \frac{k |g_w - g_{\textit{nw}}|}{\mu_w |u_T|} = 1.3.
\end{equation}
The saturation maps for different angles $\theta$ at $T = 0.092 \, \text{PVI}$ are in Fig.~\ref{fig:saturation_maps_three_wells}. 
The constant time step size is $\Delta t \approx 0.0013 \, \text{PVI}$ and corresponds to a CFL number of approximately 3.4. 
The saturation maps show that the 1D-PPU and 1D-IHU schemes are very sensitive to the orientation of the grid as they prefer 
flow along the grid directions. The proposed MultiD-IHU scheme does not exhibit this severe numerical bias and significantly reduces 
the grid orientation effect. We note that the MultiD-IHU saturation maps are in agreement with those generated with MultiD-PPU. 
These findings are confirmed by the water cuts at the two producers in Fig.~\ref{fig:water_cut_three_wells}. They show that the truly 
multidimensional schemes yield a consistent value of the breakthrough time at the producers when the grid is rotated, and better 
preserve the symmetry of the problem. This is not the case for 1D-PPU and 1D-IHU, which both predict very different water cuts
for the two producers. 

\begin{figure}[ht]
  \centering
  \begin{tikzpicture}
    \node[anchor=south west,inner sep=0] at (0,0) {\includegraphics[scale=0.162]{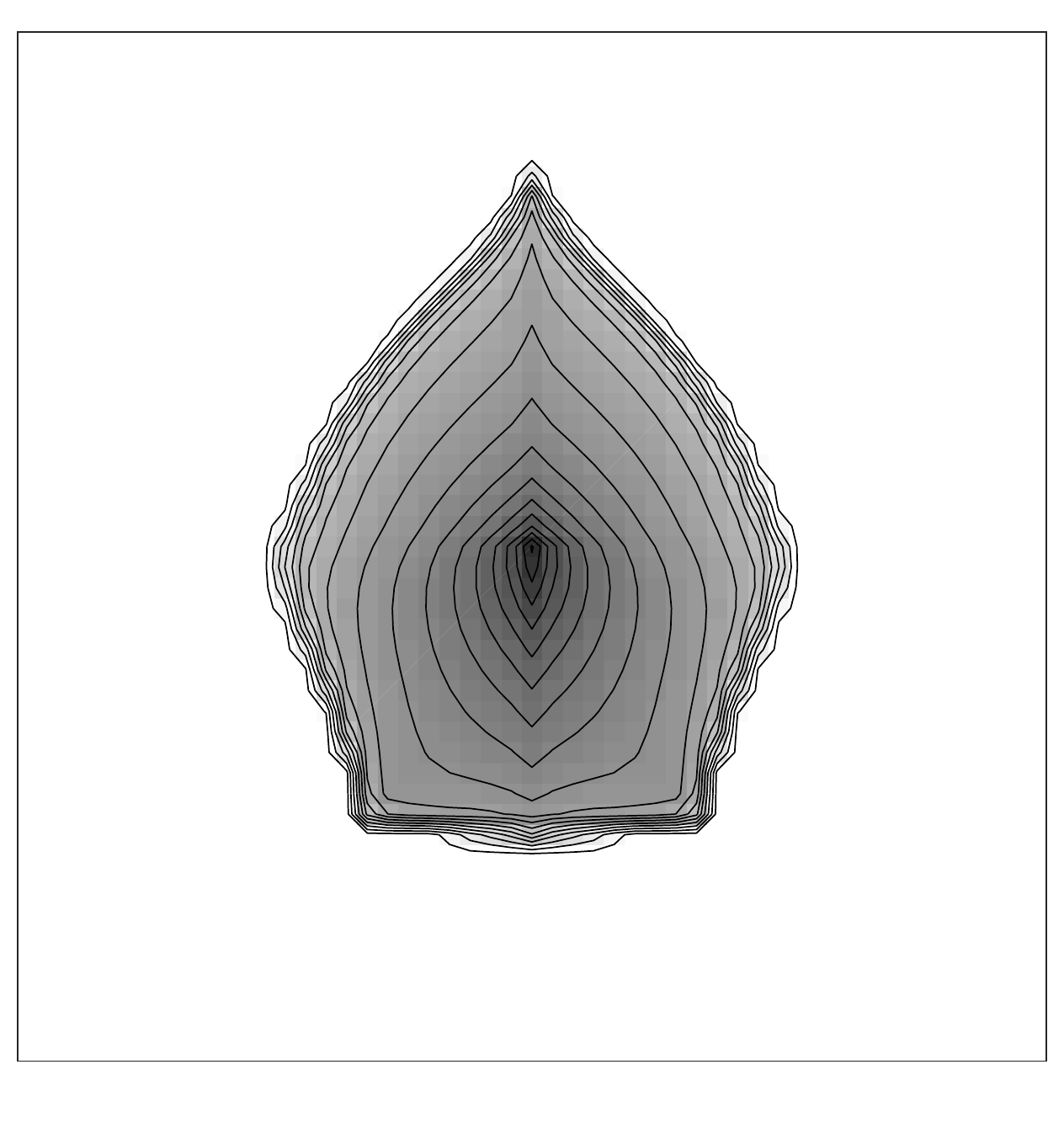}};
    \draw[gray,thick,dashed] (1.05,1.14) circle (1.01cm);
    \node (ib_4) at (1.045,1.135) {\tiny $\textcolor{white}{\star}$};
    \node (ib_4) at (0.7362,0.6252) {\tiny $\textcolor{white}{\star}$};
    \node (ib_4) at (1.3588,0.6252) {\tiny $\textcolor{white}{\star}$};
    \node[anchor=south west,inner sep=0, rotate=22.5] at (0.52,-2.945) {\includegraphics[scale=0.162]{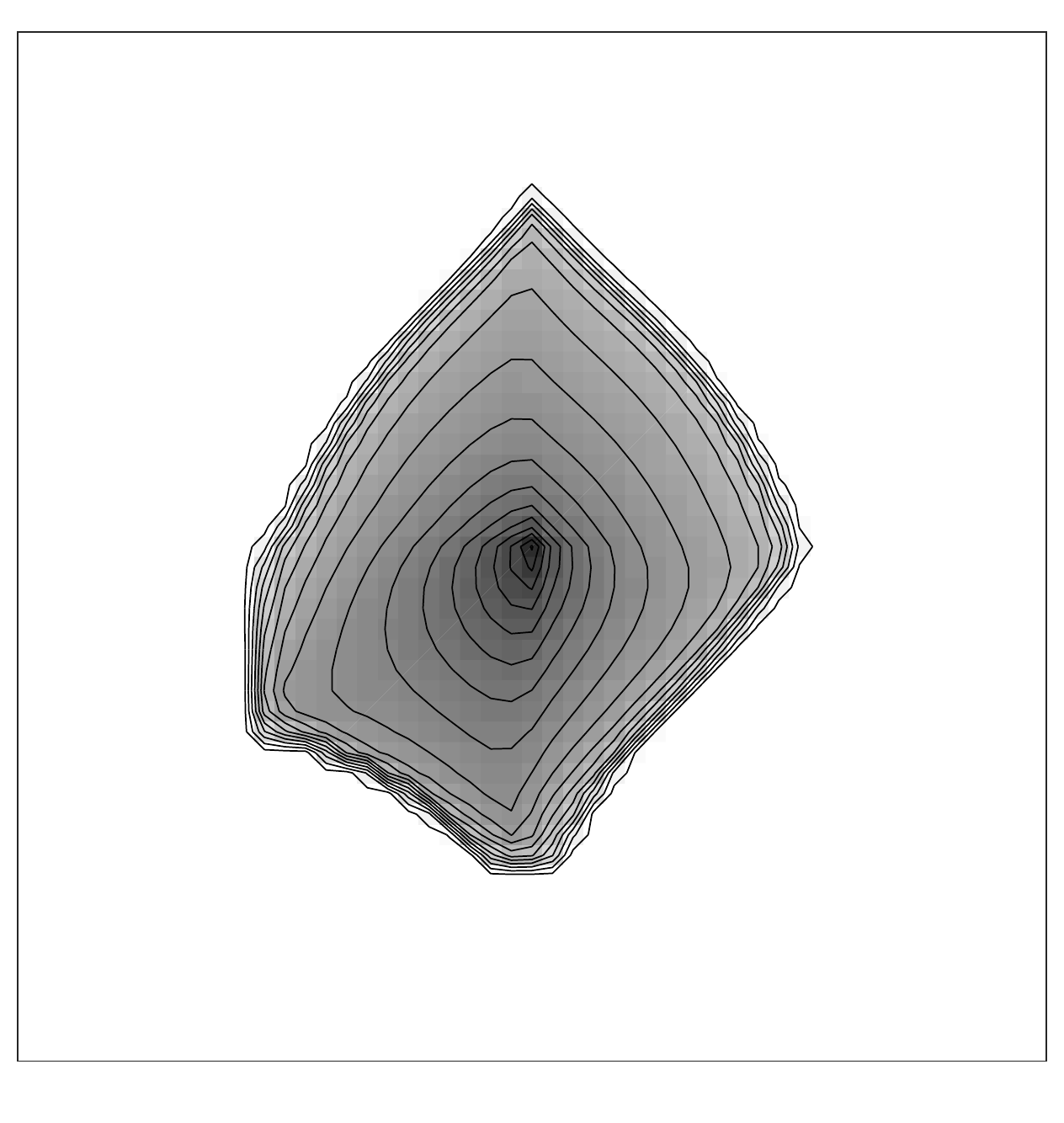}};
    \draw[gray,dashed,thick] (1.05,-1.5) circle (1.01cm);
    \node (ib_4) at (1.045,-1.5) {\tiny $\textcolor{white}{\star}$};
    \node (ib_4) at (0.7362,-2.0098) {\tiny $\textcolor{white}{\star}$};
    \node (ib_4) at (1.3588,-2.0098) {\tiny $\textcolor{white}{\star}$};
    \node[anchor=south west,inner sep=0, rotate=45] at (1.115,-5.68) {\includegraphics[scale=0.162]{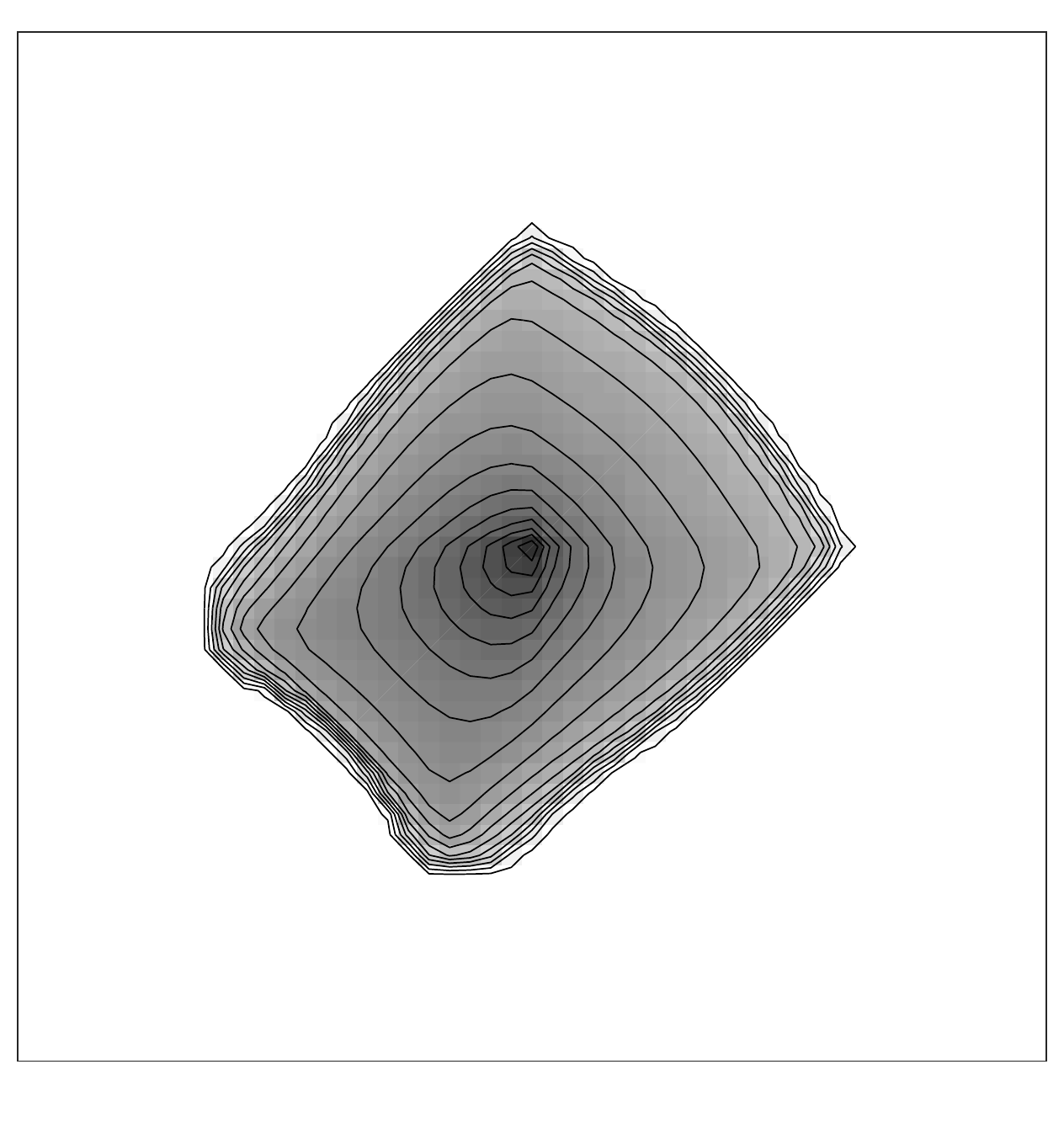}};
    \draw[gray,dashed,thick] (1.05,-4.14) circle (1.01cm);
    \node (ib_4) at (1.045,-4.14) {\tiny $\textcolor{white}{\star}$};
    \node (ib_4) at (0.7362,-4.6498) {\tiny $\textcolor{white}{\star}$};
    \node (ib_4) at (1.3588,-4.6498) {\tiny $\textcolor{white}{\star}$};
    \node (ib_4) at (1.,2.8) {\small 1D-PPU};
    \node (ib_4) at (-1.2,1.16) {\small $\theta = 0$};
    \node (ib_4) at (-1.2,-1.56) {\small $\theta = \pi/8$};
    \node (ib_4) at (-1.2,-4.15) {\small $\theta = \pi/4$};
  \end{tikzpicture}
  \hspace{-0.2cm}
  \begin{tikzpicture}
    \node[anchor=south west,inner sep=0] at (0,0) {\includegraphics[scale=0.162]{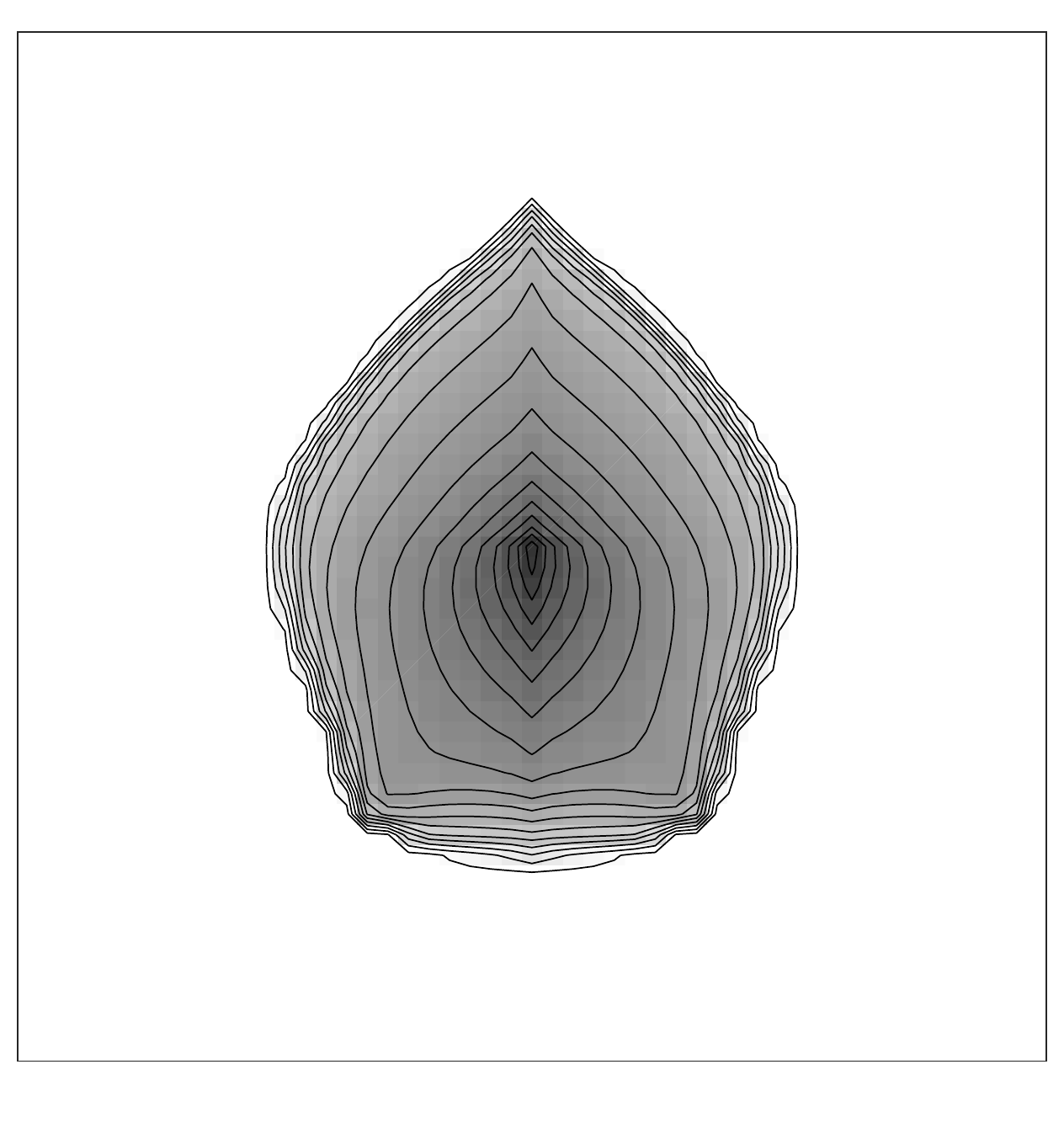}};
    \draw[gray,thick,dashed] (1.05,1.14) circle (1.01cm);
    \node (ib_4) at (1.045,1.135) {\tiny $\textcolor{white}{\star}$};
    \node (ib_4) at (0.7362,0.6252) {\tiny $\textcolor{white}{\star}$};
    \node (ib_4) at (1.3588,0.6252) {\tiny $\textcolor{white}{\star}$};
    \node[anchor=south west,inner sep=0, rotate=22.5] at (0.52,-2.945) {\includegraphics[scale=0.162]{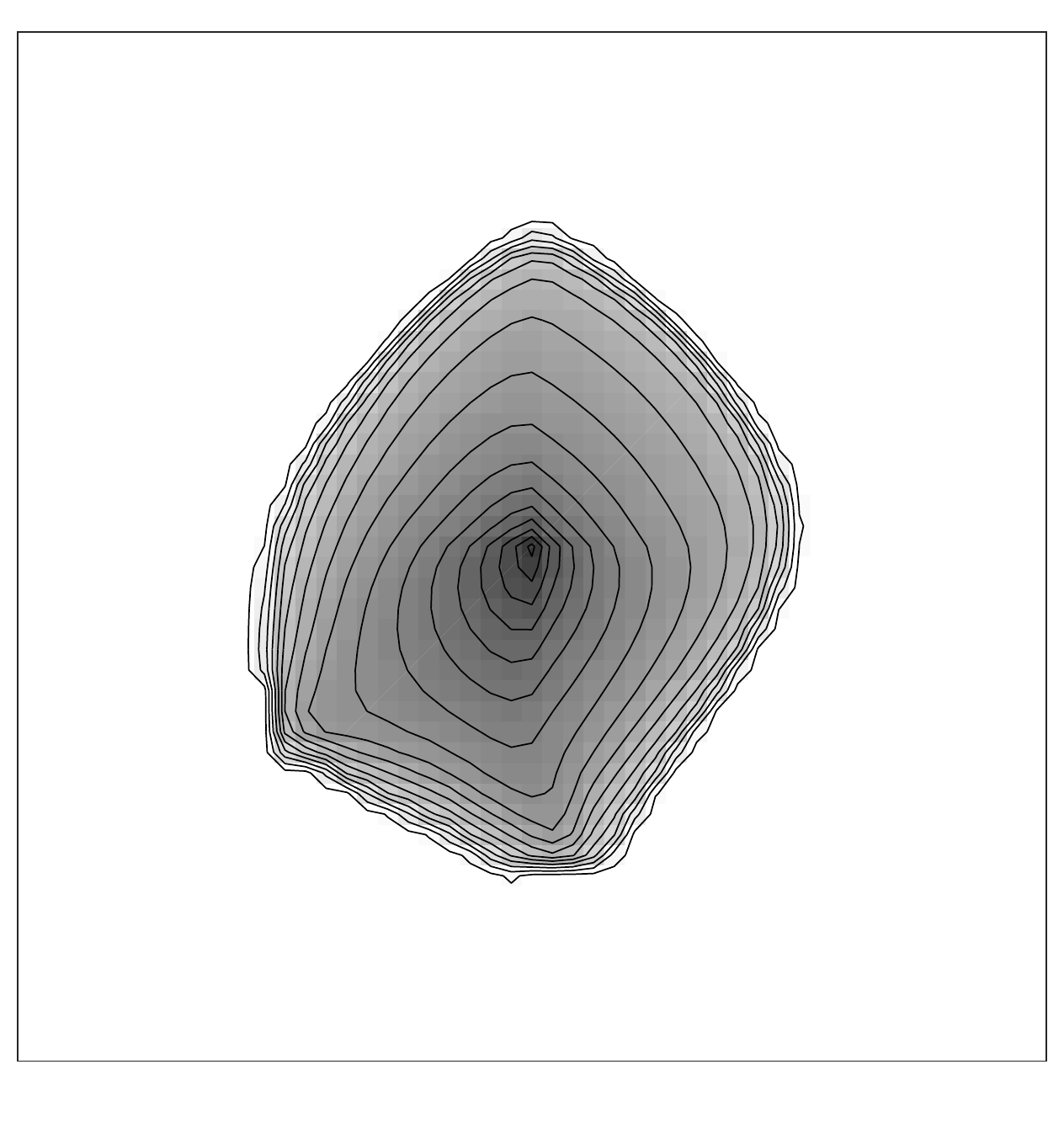}};
    \draw[gray,dashed,thick] (1.05,-1.5) circle (1.01cm);
    \node (ib_4) at (1.045,-1.5) {\tiny $\textcolor{white}{\star}$};
    \node (ib_4) at (0.7362,-2.0098) {\tiny $\textcolor{white}{\star}$};
    \node (ib_4) at (1.3588,-2.0098) {\tiny $\textcolor{white}{\star}$};
    \node[anchor=south west,inner sep=0, rotate=45] at (1.115,-5.68) {\includegraphics[scale=0.162]{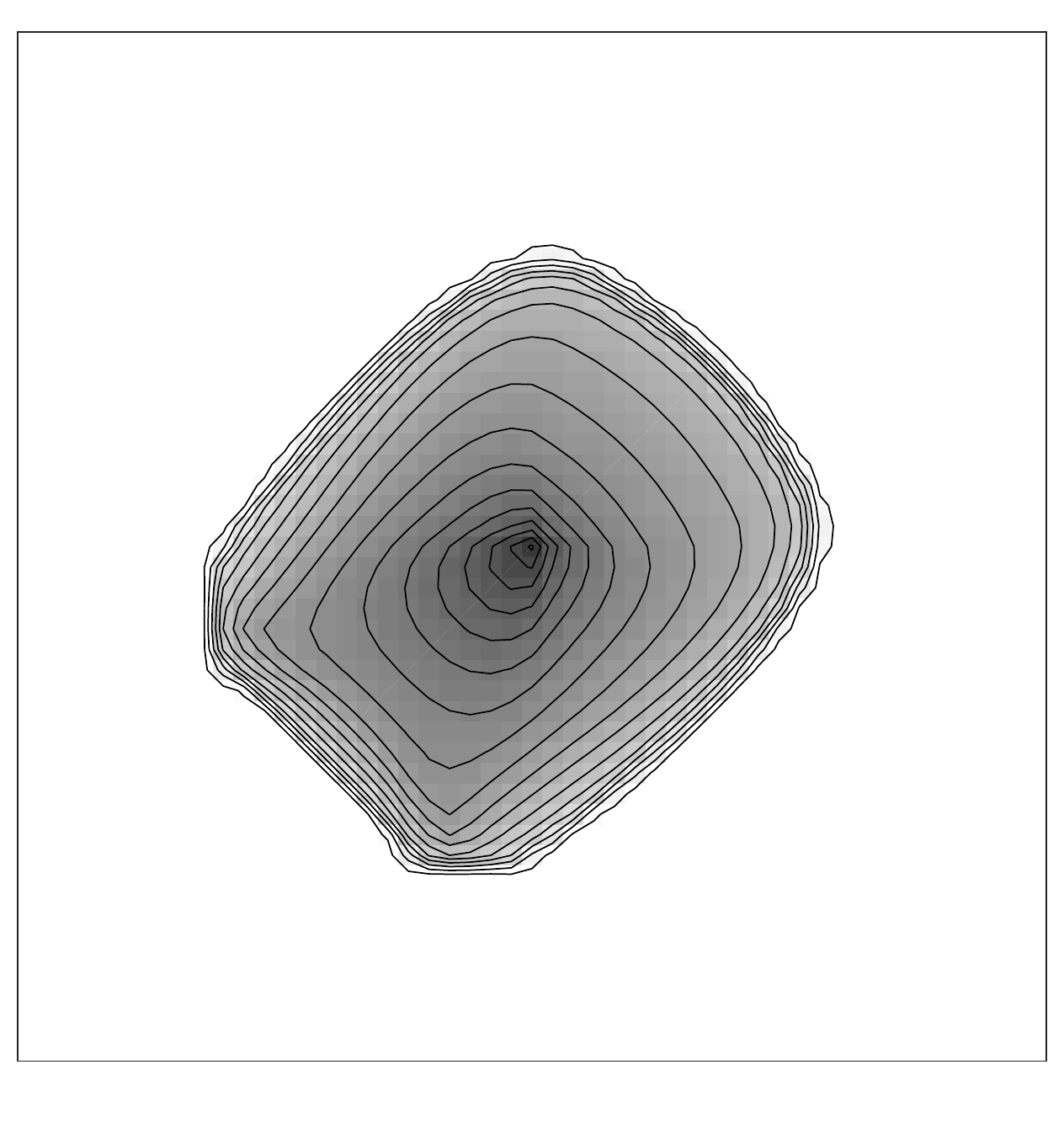}};
    \draw[gray,dashed,thick] (1.05,-4.14) circle (1.01cm);
    \node (ib_4) at (1.045,-4.14) {\tiny $\textcolor{white}{\star}$};
    \node (ib_4) at (0.7362,-4.6498) {\tiny $\textcolor{white}{\star}$};
    \node (ib_4) at (1.3588,-4.6498) {\tiny $\textcolor{white}{\star}$};
    \node (ib_4) at (1.,2.8) {\small 1D-IHU};
  \end{tikzpicture}
  \hspace{-0.2cm}
  \begin{tikzpicture}
    \node[anchor=south west,inner sep=0] at (0,0) {\includegraphics[scale=0.162]{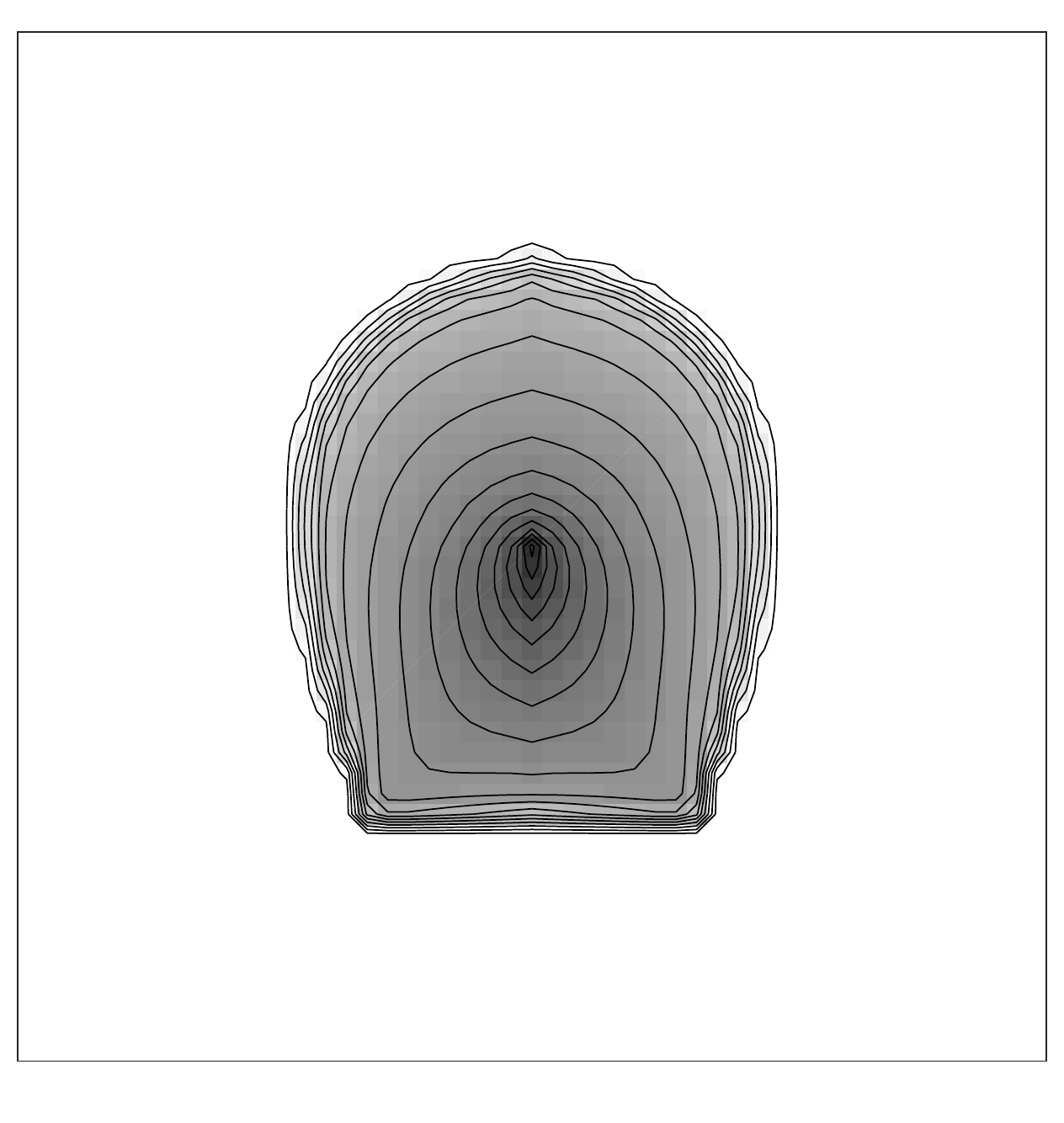}};
    \draw[gray,thick,dashed] (1.05,1.14) circle (1.01cm);
    \node (ib_4) at (1.045,1.135) {\tiny $\textcolor{white}{\star}$};
    \node (ib_4) at (0.7362,0.6252) {\tiny $\textcolor{white}{\star}$};
    \node (ib_4) at (1.3588,0.6252) {\tiny $\textcolor{white}{\star}$};
    \node[anchor=south west,inner sep=0, rotate=22.5] at (0.52,-2.945) {\includegraphics[scale=0.162]{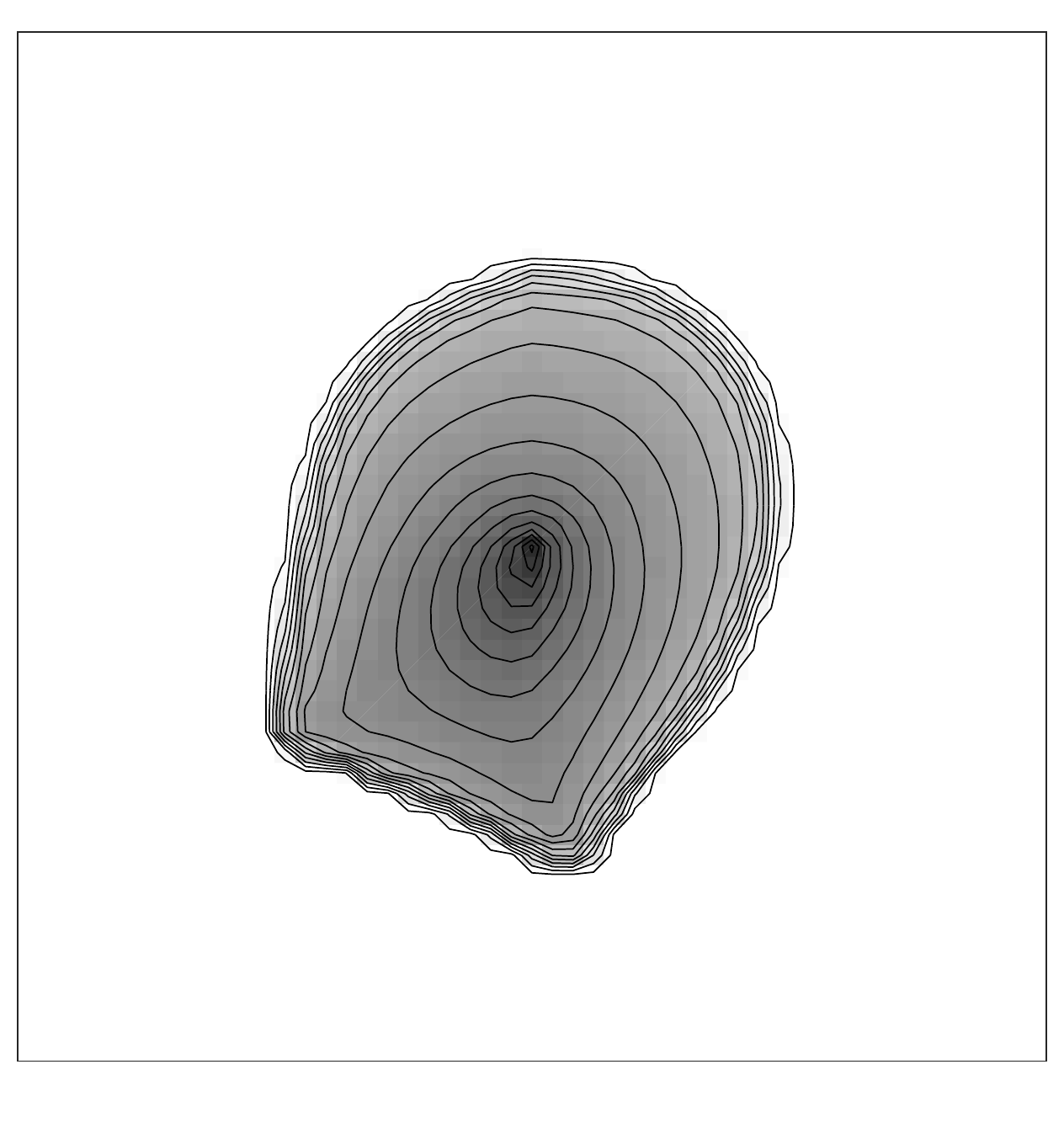}};
    \draw[gray,dashed,thick] (1.05,-1.5) circle (1.01cm);
    \node (ib_4) at (1.045,-1.5) {\tiny $\textcolor{white}{\star}$};
    \node (ib_4) at (0.7362,-2.0098) {\tiny $\textcolor{white}{\star}$};
    \node (ib_4) at (1.3588,-2.0098) {\tiny $\textcolor{white}{\star}$};
    \node[anchor=south west,inner sep=0, rotate=45] at (1.115,-5.68) {\includegraphics[scale=0.162]{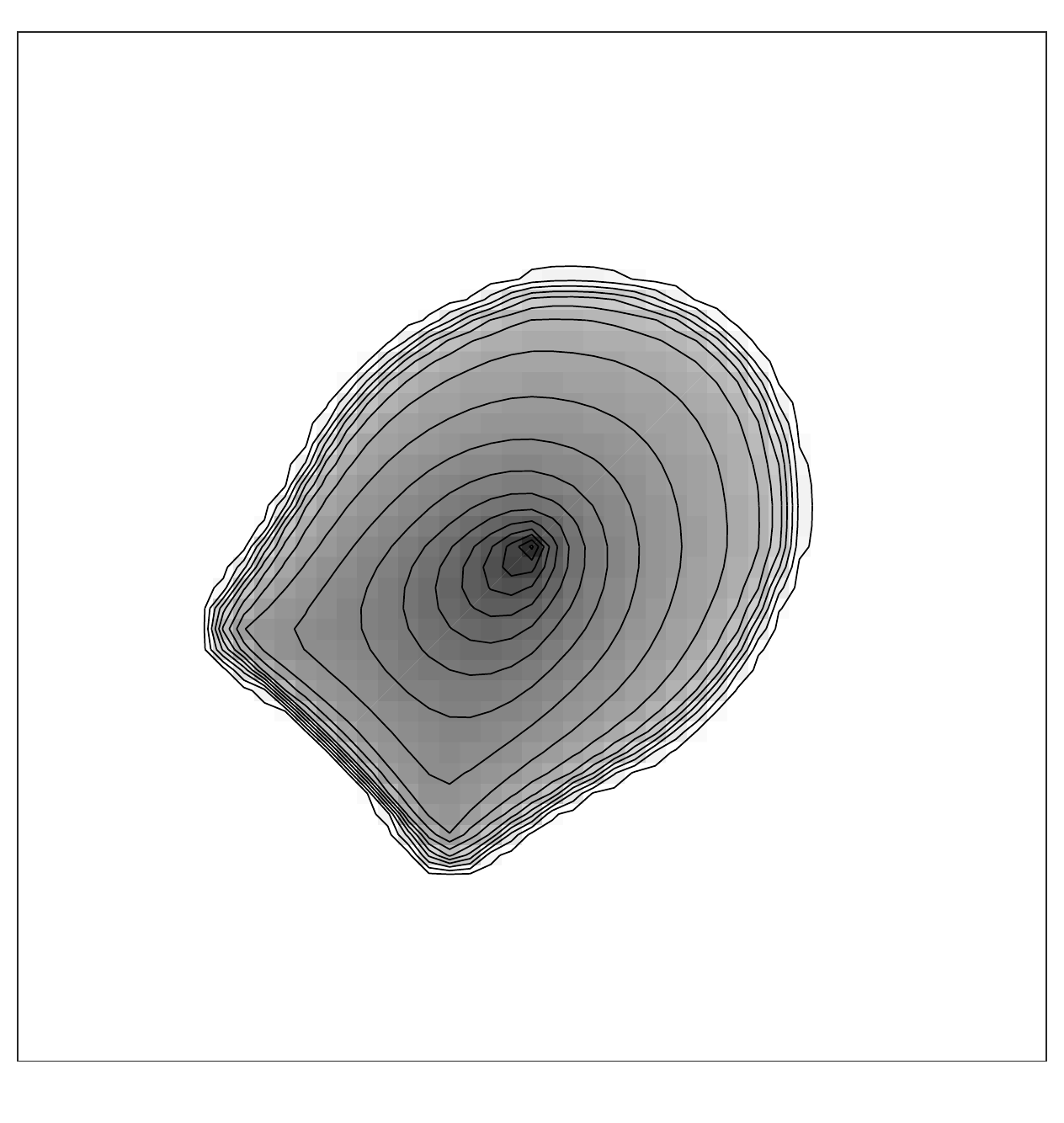}};
    \draw[gray,dashed,thick] (1.05,-4.14) circle (1.01cm);
    \node (ib_4) at (1.045,-4.14) {\tiny $\textcolor{white}{\star}$};
    \node (ib_4) at (0.7362,-4.6498) {\tiny $\textcolor{white}{\star}$};
    \node (ib_4) at (1.3588,-4.6498) {\tiny $\textcolor{white}{\star}$};
    \node (ib_4) at (1.05,2.8) {\small MultiD-PPU};
    \node (ib_4) at (1.05,2.451) {\small (SMU)};
  \end{tikzpicture}
  \hspace{-0.2cm}
  \begin{tikzpicture}
    \node[anchor=south west,inner sep=0] at (0,0) {\includegraphics[scale=0.162]{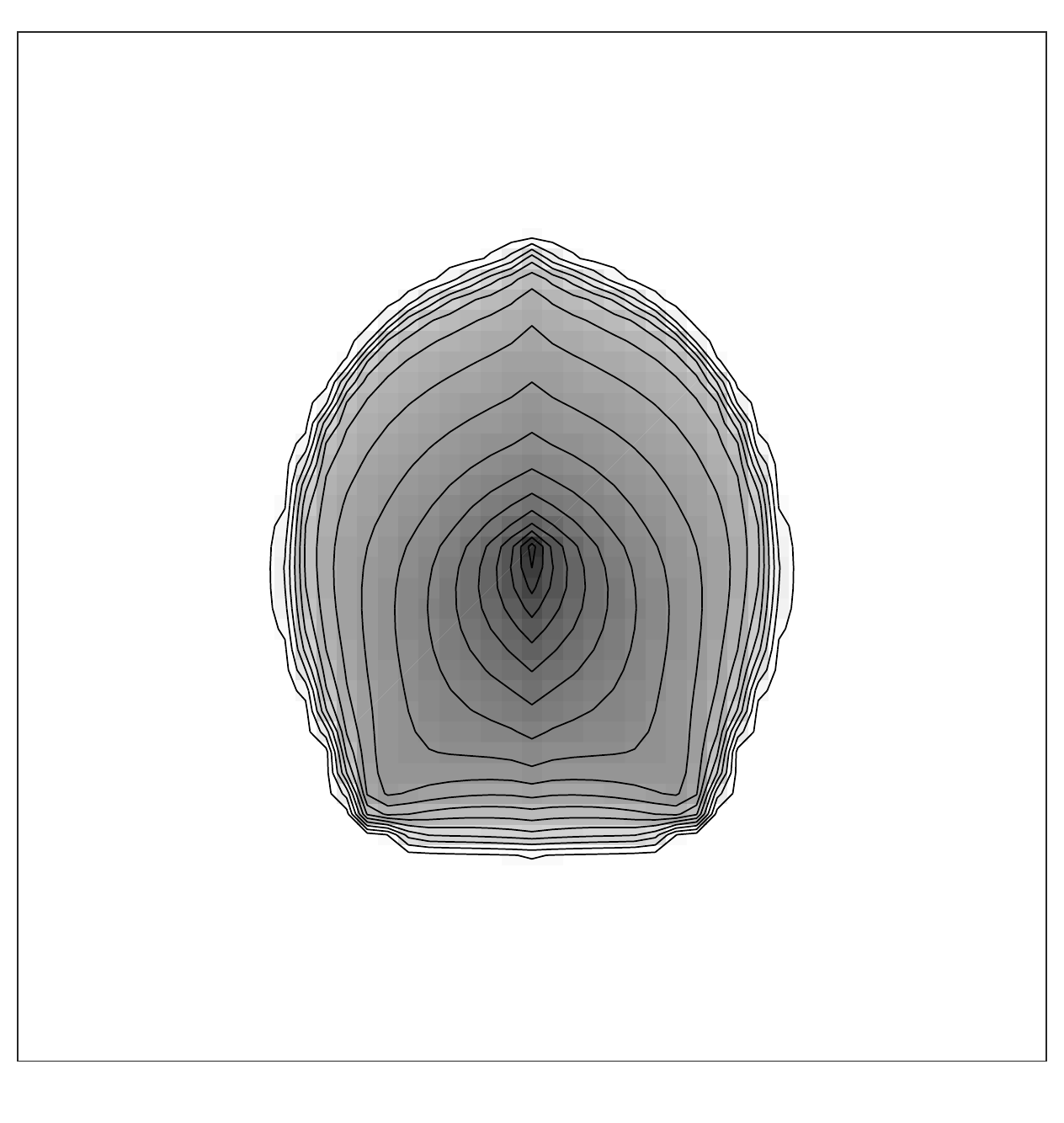}};
    \draw[gray,thick,dashed] (1.05,1.14) circle (1.01cm);
    \node (ib_4) at (1.045,1.135) {\tiny $\textcolor{white}{\star}$};
    \node (ib_4) at (0.7362,0.6252) {\tiny $\textcolor{white}{\star}$};
    \node (ib_4) at (1.3588,0.6252) {\tiny $\textcolor{white}{\star}$};
    \node[anchor=south west,inner sep=0, rotate=22.5] at (0.52,-2.945) {\includegraphics[scale=0.162]{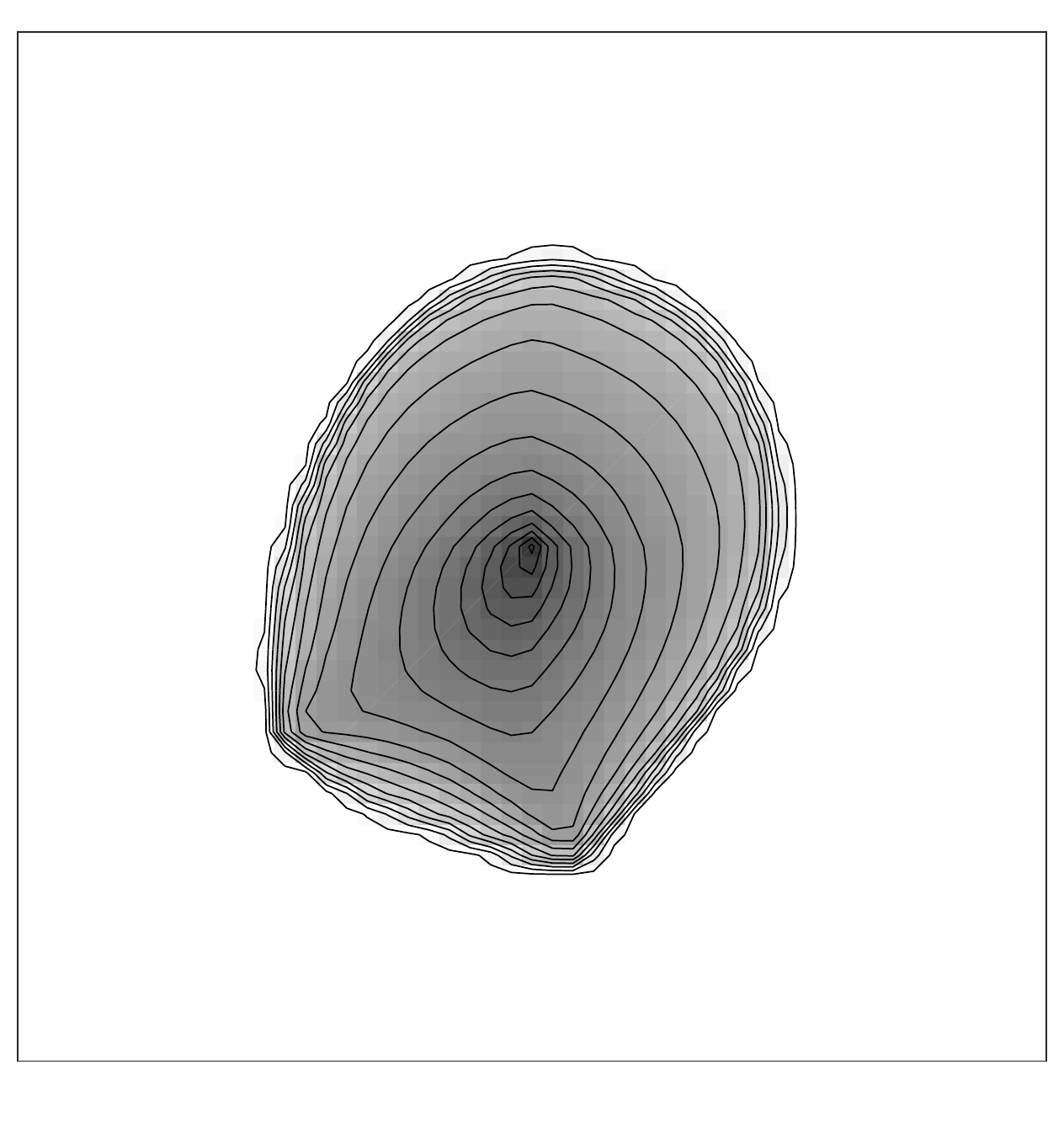}};
    \draw[gray,dashed,thick] (1.05,-1.5) circle (1.01cm);
    \node (ib_4) at (1.045,-1.5) {\tiny $\textcolor{white}{\star}$};
    \node (ib_4) at (0.7362,-2.0098) {\tiny $\textcolor{white}{\star}$};
    \node (ib_4) at (1.3588,-2.0098) {\tiny $\textcolor{white}{\star}$};
    \node[anchor=south west,inner sep=0, rotate=45] at (1.115,-5.68) {\includegraphics[scale=0.162]{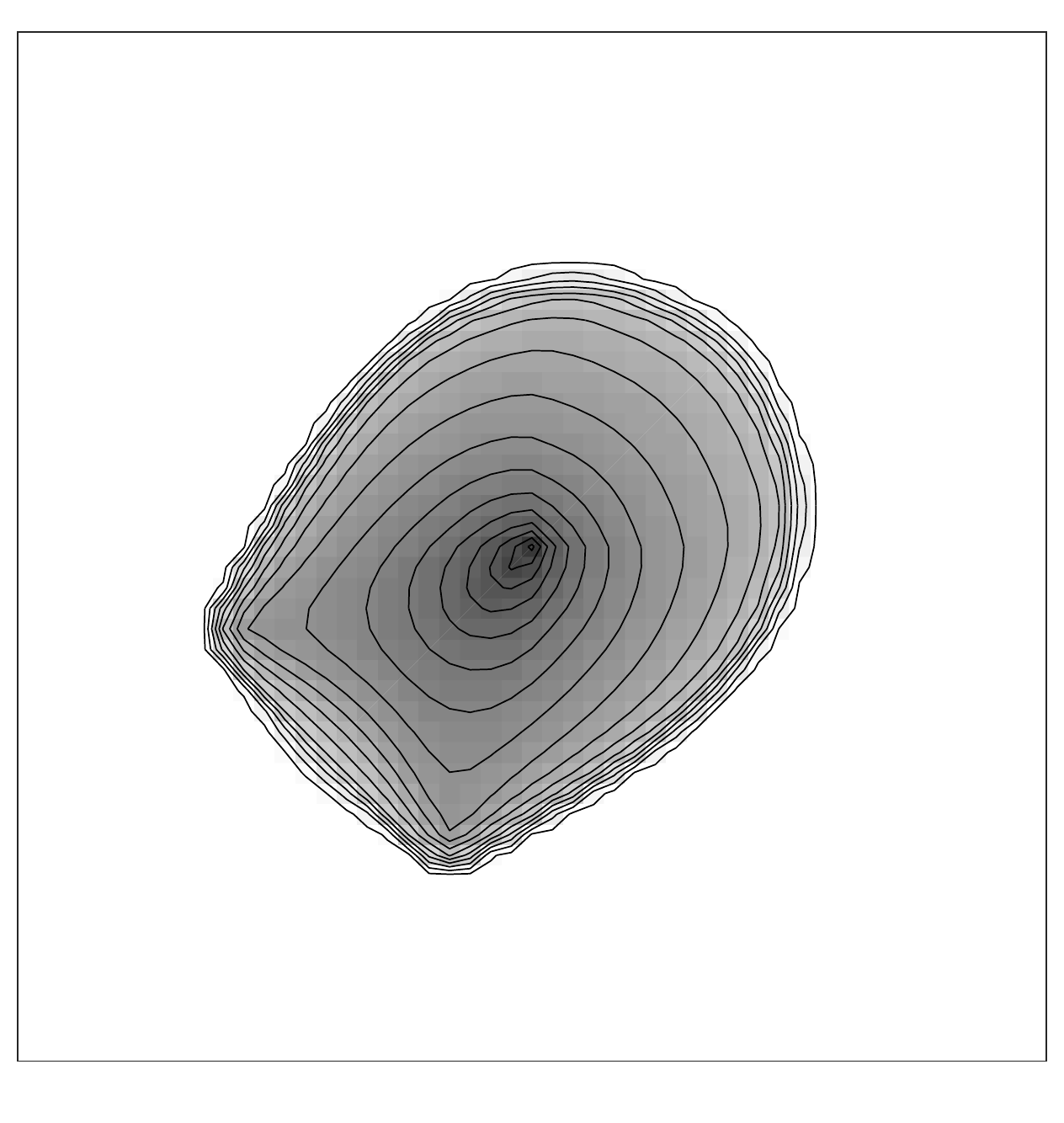}};
    \draw[gray,dashed,thick] (1.05,-4.14) circle (1.01cm);
    \node (ib_4) at (1.045,-4.14) {\tiny $\textcolor{white}{\star}$};
    \node (ib_4) at (0.7362,-4.6498) {\tiny $\textcolor{white}{\star}$};
    \node (ib_4) at (1.3588,-4.6498) {\tiny $\textcolor{white}{\star}$};
    \node (ib_4) at (1.05,2.8) {\small MultiD-IHU};
    \node (ib_4) at (1.05,2.451) {\small (SMU4)};
  \end{tikzpicture}
\vspace{-0.3cm}
\caption{\label{fig:saturation_maps_three_wells}Saturation maps for different angles $\theta$ at $T = 0.092 \, \text{PVI}$
  in the three-well problem with buoyancy. The CFL number of these simulations is 3.4. Outside the dashed circle, the control
  volumes have an absolute permeability set to $5 \times 10^{-5} \, \text{mD}$.  The white stars show the location
  of the wells.}
\end{figure}
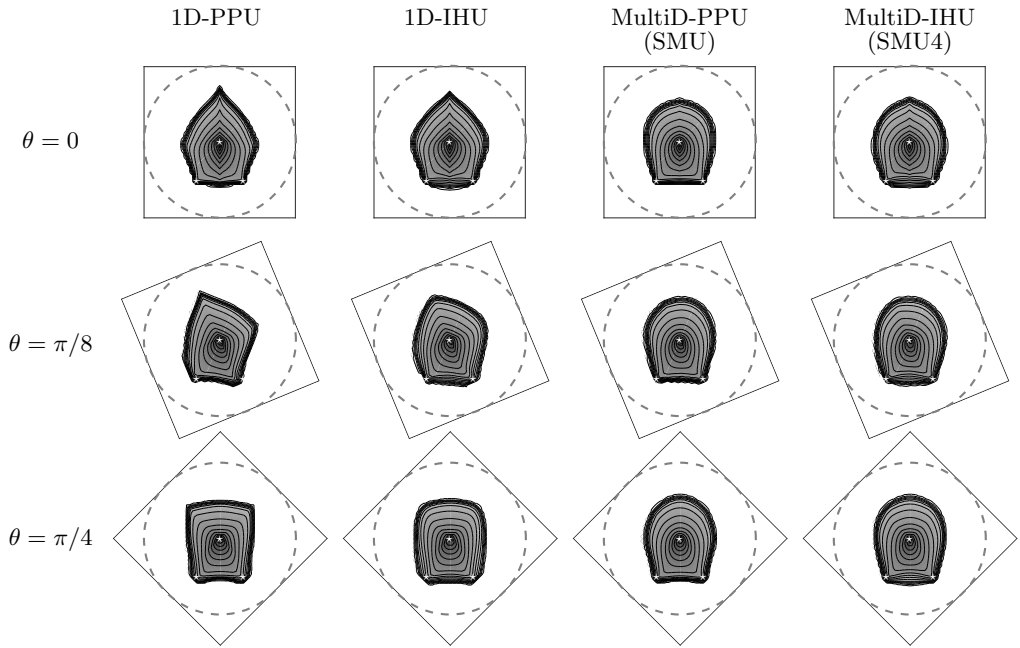

\begin{figure}[ht]
\centering
\subfigure[1D-PPU]{
  \begin{tikzpicture}
    \node[anchor=south west,inner sep=0] at (0,0) {\includegraphics[scale=0.3]{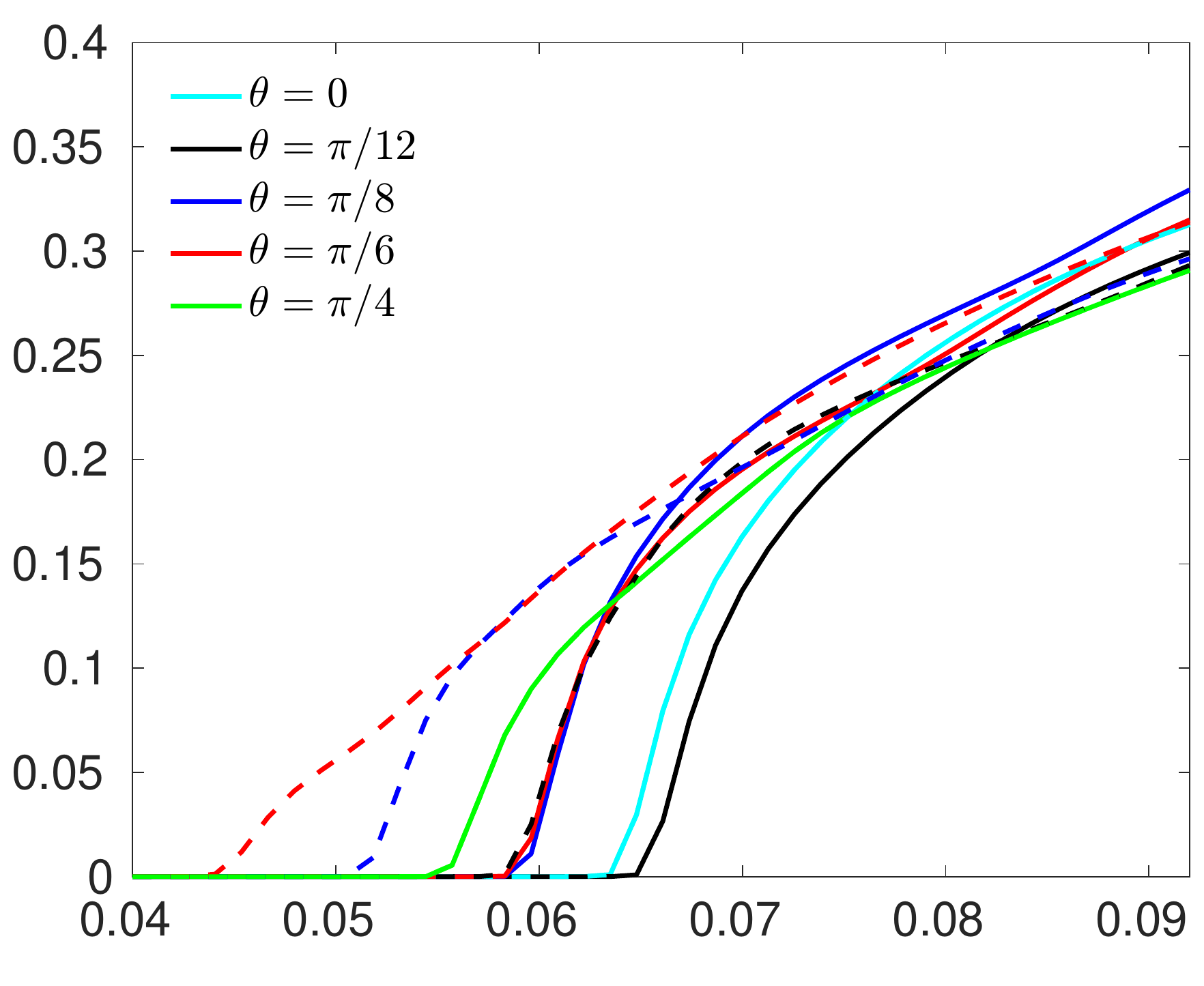}};
    \node (ib_4) at (2.8,-0.) {\small PVI};
    \node[rotate=90] (ib_4) at (-0.25,2.35) {\small Water cut};
  \end{tikzpicture}
\label{fig:water_cut_three_wells_a}
}
\hspace{1cm}
\subfigure[1D-IHU]{
  \begin{tikzpicture}
    \node[anchor=south west,inner sep=0] at (0,0) {\includegraphics[scale=0.3]{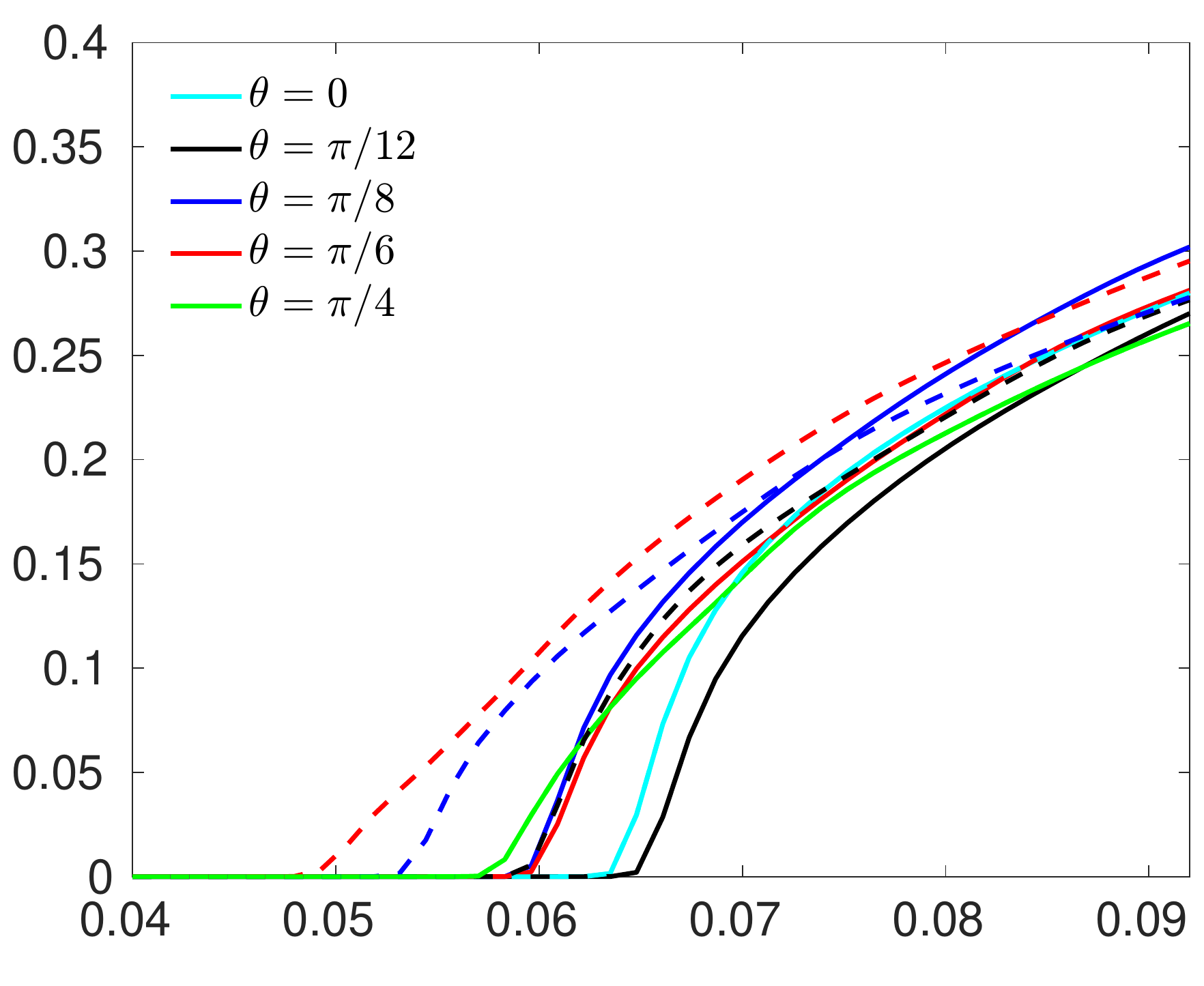}};
    \node (ib_4) at (2.8,-0.) {\small PVI};
  \end{tikzpicture}
\label{fig:water_cut_three_wells_b}
} \\
\subfigure[MultiD-PPU]{
  \begin{tikzpicture}
    \node[anchor=south west,inner sep=0] at (0,0) {\includegraphics[scale=0.3]{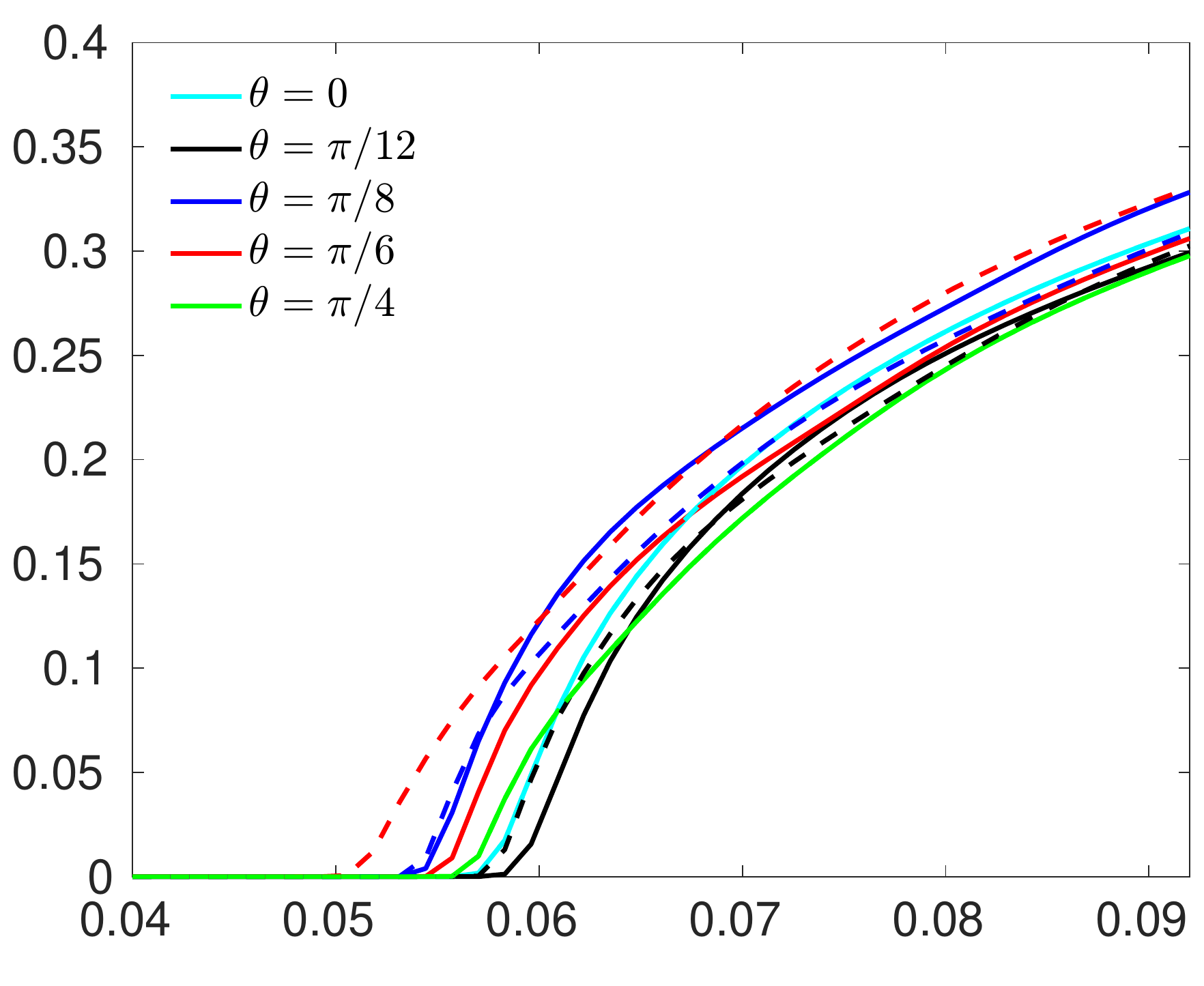}};
    \node (ib_4) at (2.8,-0.) {\small PVI};
    \node[rotate=90] (ib_4) at (-0.25,2.35) {\small Water cut};
  \end{tikzpicture}
\label{fig:water_cut_three_wells_c}
} 
\hspace{1cm}
\subfigure[MultiD-IHU]{
  \begin{tikzpicture}
    \node[anchor=south west,inner sep=0] at (0,0) {\includegraphics[scale=0.3]{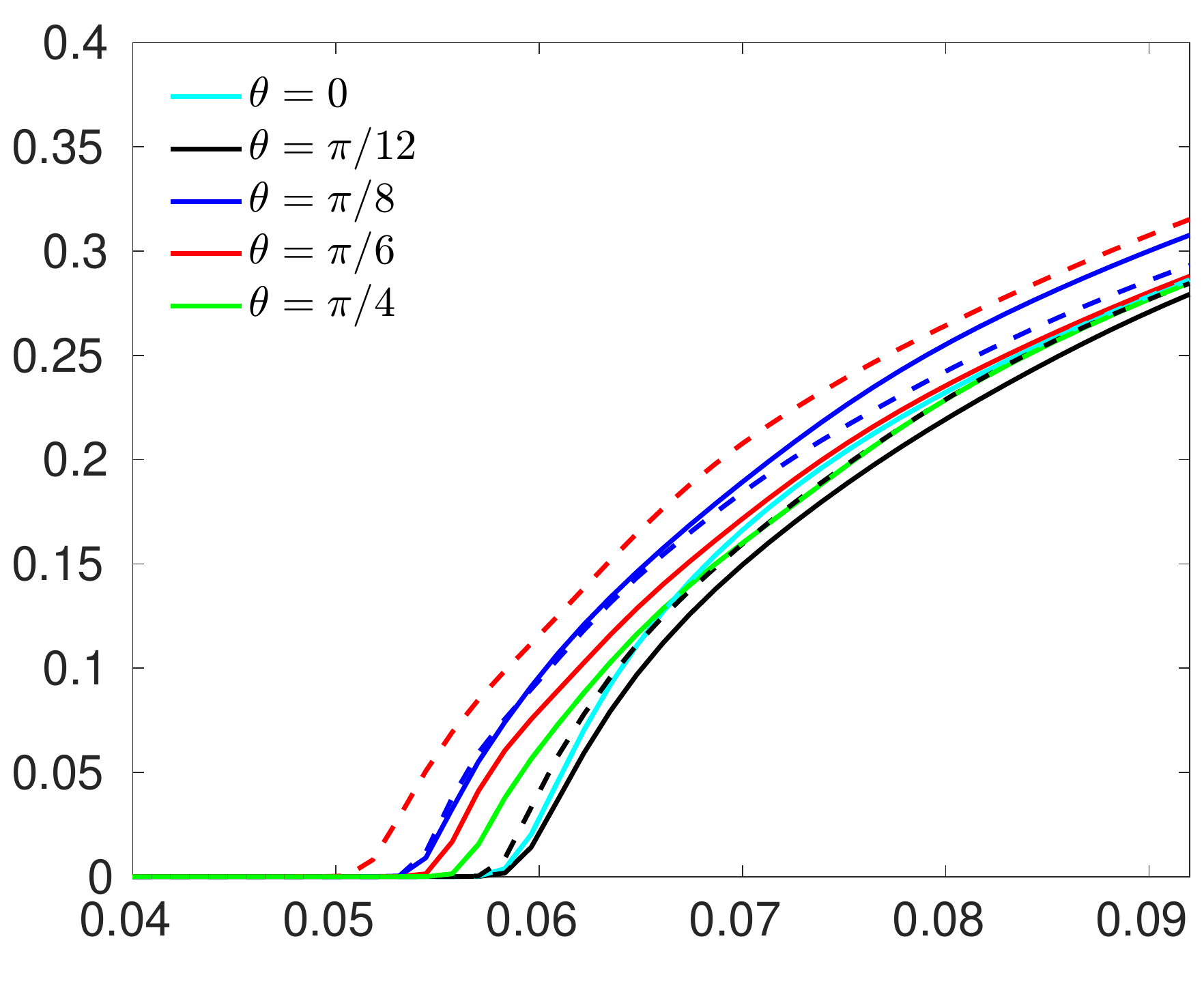}};
    \node (ib_4) at (2.8,-0.) {\small PVI};
  \end{tikzpicture}
\label{fig:water_cut_three_wells_d}
}
\vspace{-0.3cm}
\caption{\label{fig:water_cut_three_wells}Water cuts at the two producers for different angles $\theta$ at
  $T = 0.092 \, \text{PVI}$ in the three-well problem with buoyancy. The CFL number of these simulations is 3.4.
  The solid (respectively, dashed) lines show the water cut for the producer on the left (respectively, on the right).}
\end{figure}

The nonlinear behavior for different angles is summarized in Table \ref{tab:nonlinear_behavior_three_wells}. We run the simulations with a 
small time step size leading to a CFL number of 3.4, and with a larger time step size corresponding to a CFL number of 59.7. 
Given that we have set the maximum number of Newton iterations to a large value -- i.e., 50 --, none of the schemes require time step
cuts. 
As in 
\cite{kozdon2011multidimensional}, we observe that the multidimensional schemes require fewer Newton iterations than the one-dimensional 
schemes, even though the cost per iteration of the multidimensional schemes is significantly larger. With MultiD-IHU, the reduction in 
the number of iterations summed over the five cases compared to 1D-PPU is 12.6\% for a CFL number of 3.4, and reaches 12.3\% for a larger CFL 
number of 59.7. We finally observe that MultiD-IHU consistently converges faster than MultiD-PPU in both time stepping configurations.

\begin{table}[!ht]
\scalebox{0.85}{
\centering
         \begin{tabular}{ccccccccc}
           \\ \toprule
            Angle     & \multicolumn{4}{c}{Small CFL}        & \multicolumn{4}{c}{Large CFL}  \\ 
            $\theta$  & 1D-PPU & 1D-IHU & MultiD-PPU & MultiD-IHU    & 1D-PPU & 1D-IHU & MultiD-PPU & MultiD-IHU \\ \toprule
              0        &   267     &   253     &     241       &    231           &    44    &    44    &       40     &  40 \\
              $\pi/12$ &   270     &   258     &     244       &    234           &    47    &    48    &       43    &   42 \\
              $\pi/8$  &   270     &   259     &     244       &    236           &    48    &    47    &       42     &  40 \\
              $\pi/6$  &   270     &   258     &     248       &    239           &    48    &    47    &       43     &  41 \\
              $\pi/4$  &   267     &   245     &     243       &    234           &    48    &    46    &       43     &  43 \\ \bottomrule 
         \end{tabular}}
\caption{\label{tab:nonlinear_behavior_three_wells} Total number of Newton iterations for different angles at
  $T = 0.092 \, \text{PVI}$ in the three-well problem with buoyancy. There was no time step cut for any of the
  schemes considered here.}
\end{table}

\subsection{\label{subsection_second_numerical_example}Two-phase flow in a heterogeneous medium}

Here, we evaluate the reduction of the grid orientation effect achieved with MultiD-IHU on a heterogeneous porous medium. 
The two-dimensional $x - y$ domain of dimensions $[-75 \, \text{ft}, \, 75 \, \text{ft}] \times [-75 \, \text{ft}, \, 75 \, \text{ft}]$ 
is discretized with 101 $\times$ 101 uniform Cartesian control volumes. Outside a disc of radius $r_0 = (150 - \Delta x)/2 \, \text{ft}$ 
centered in $(x_0,y_0) = (x'_0,y'_0) = (0,0)$, we set a constant absolute permeability of $10^{-10} \, \text{mD}$. Inside the disc, the permeability 
field is given by
\begin{equation}
k(x',y') = 200 \bigg( 1 + \big( \cos(3 x' \pi) \cos(3 y' \pi) \cos(3 \bar{x} \pi) \cos(3 \bar{y} \pi) \big)/2 \bigg)^3,
\label{permeability_field}
\end{equation}
where 
\begin{equation}
\bar{x} = x' \cos\big(\frac{\pi}{4}\big) + y' \sin\big(\frac{\pi}{4}\big)  \qquad \text{and} \qquad \bar{y} = -x' \sin\big(\frac{\pi}{4}\big) + y' \cos\big(\frac{\pi}{4}\big).
\end{equation}
The permeability map obtained from (\ref{permeability_field}) can be seen in Fig.~\ref{fig:permeability_map_second_numerical_example}.
It exhibits eight preferential paths originating from the center of the domain. This map is rotated to test the sensitivity of the 
schemes to the orientation of the grid using the same technique as in Section \ref{subsection_three_well_problem_with_buoyancy}. The height of 
each control volume in the domain is a function of the radius $r = \sqrt{x'^2+y'^2}$, computed with
\begin{equation}
  h(x',y') = h_{\textit{ref}} + h_{\textit{bump}}(x',y') \qquad \text{where} \quad h_{\textit{bump}}(x',y') = 20 \sin\bigg( \big(1 + \min(1,\frac{r}{r_0}) \big)\frac{\pi}{2} \bigg),
  \label{depth_map}
\end{equation}
where $h_{\textit{ref}}$ denotes the height of the domain boundaries. Here, $h(x_a',y_a') > h(x_b',y_b')$ means that $(x_b',y_b')$ is deeper
than $(x_a',y_a')$. With (\ref{depth_map}), $(x'_0,y'_0) = (0,0)$ is at the top of the domain while the domain boundaries are at the bottom 
as shown in Fig.~\ref{fig:depth_map_second_numerical_example}.
The domain is initially saturated with the lighter non-wetting phase.
A well placed in control volume $(x'_0,y'_0) = (0,0)$ injects a heavy and mobile wetting phase into a domain fully saturated with the
light non-wetting phase. We fix the pressure in the control volumes surrounding the disc. The phase properties are the same as
in Section \ref{subsection_three_well_problem_with_buoyancy}. The gravity number is $N_G = 12.9$, with countercurrent flow 
at about 25\% of the interfaces.

\begin{figure}[ht]
\centering
\subfigure[]{
\begin{tikzpicture}
\node[anchor=south west,inner sep=0] at (0,0) {\includegraphics[scale=0.25]{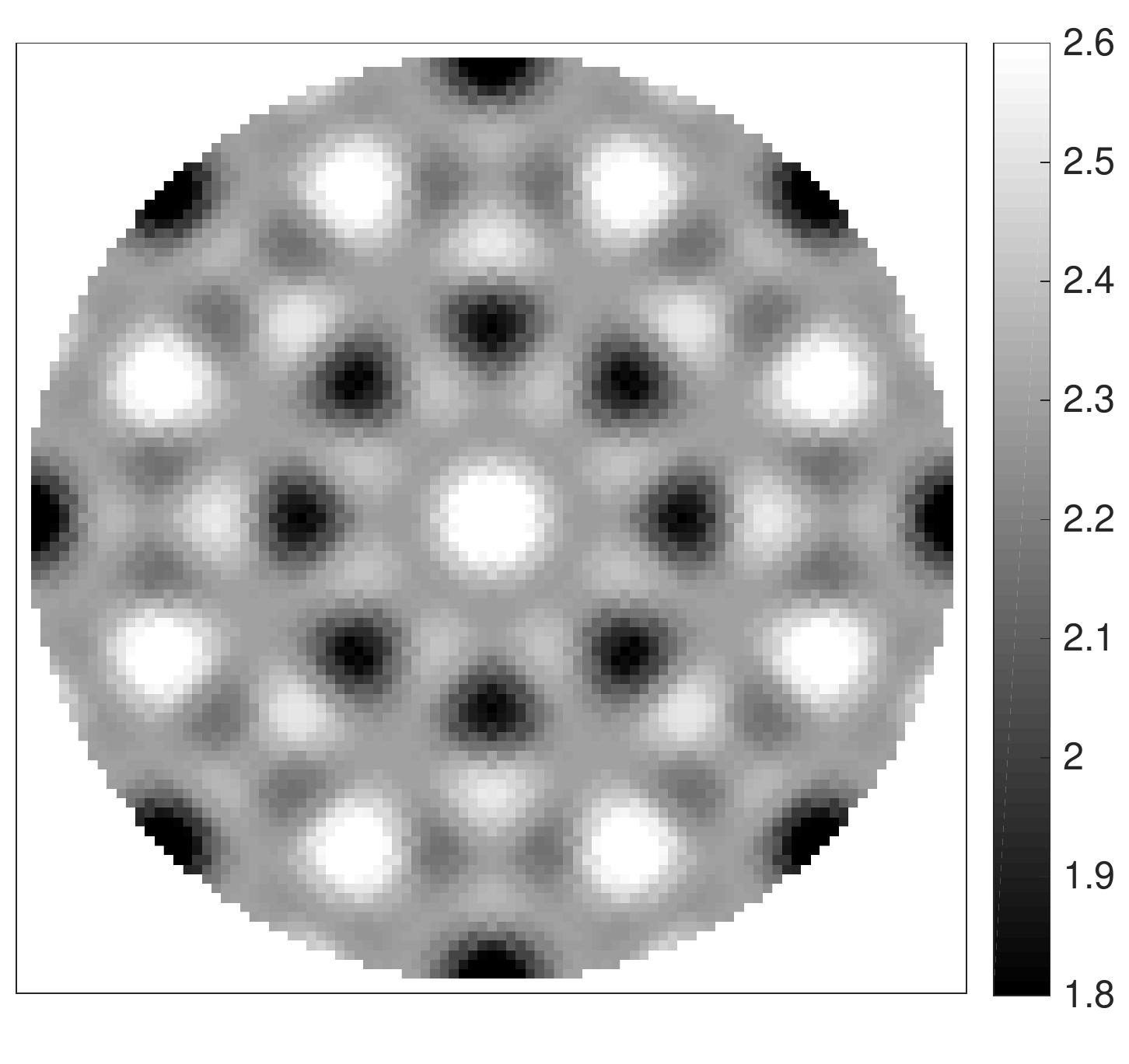}};
\node (ib_4) at (1.6,0) {$x$};
\node (ib_4) at (-0.2,1.75) {$y$};
\end{tikzpicture}
\label{fig:permeability_map_second_numerical_example}
} \hspace{2cm}
\subfigure[]{
\begin{tikzpicture}
\node[anchor=south west,inner sep=0] at (0,0) {\includegraphics[scale=0.25]{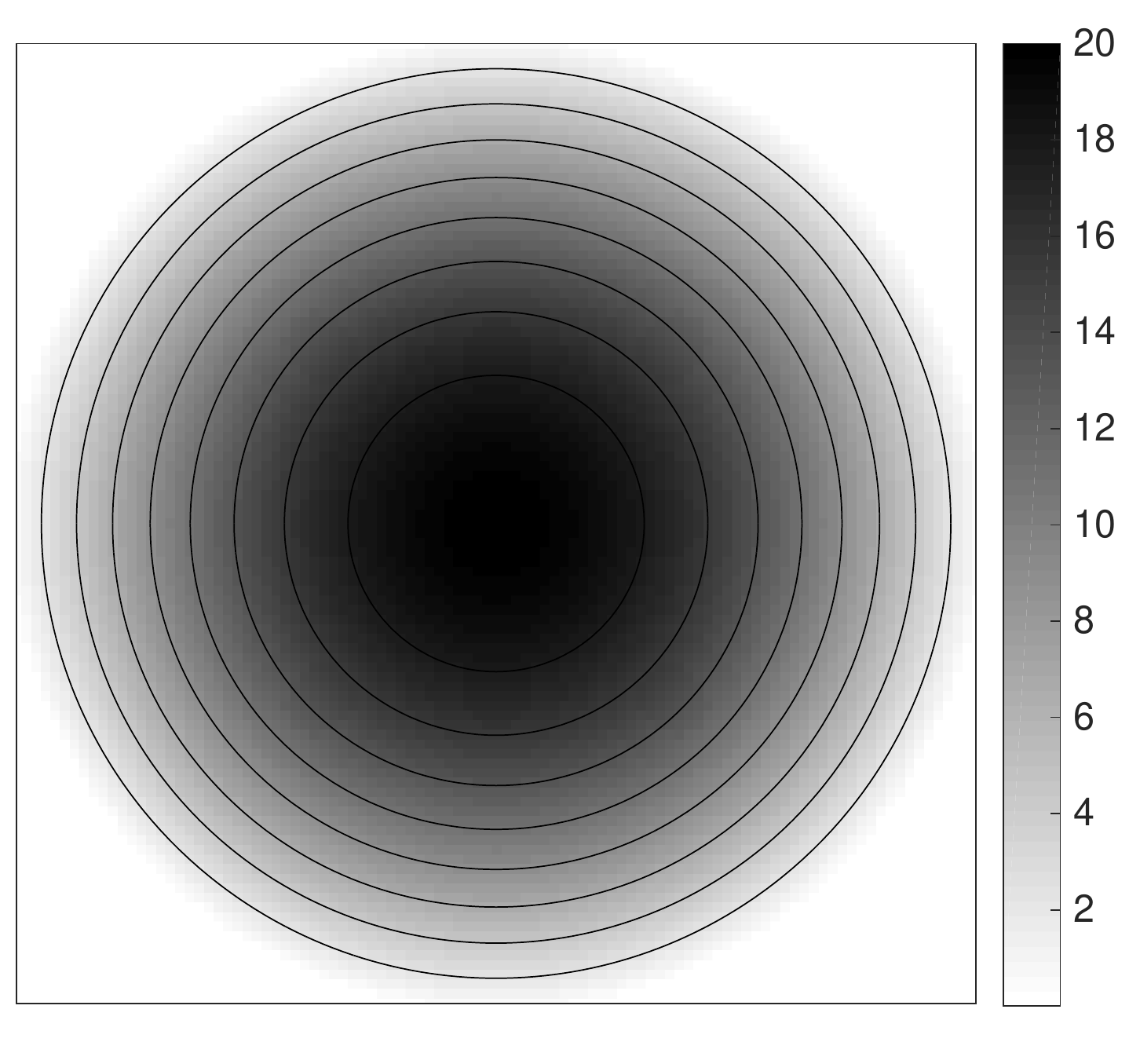}};
\node (ib_4) at (1.6,0) {$x$};
\node (ib_4) at (-0.2,1.75) {$y$};
\end{tikzpicture}
\label{fig:depth_map_second_numerical_example}
}
\vspace{-0.3cm}
\caption{\label{fig:permeability_map_depth_second_numerical_example}Logarithm of the absolute permeability field in 
\subref{fig:permeability_map_second_numerical_example}. It is computed with (\ref{permeability_field}) within 
the disc of radius $r_0$ centered in $(x'_0,y'_0) = (0,0)$. The smallest (respectively, largest) permeability is 
$30.7 \, \text{mD}$ (respectively, $675 \, \text{mD}$). The height $h_{\textit{bump}}$ obtained from (\ref{depth_map}) is shown in 
\subref{fig:depth_map_second_numerical_example}. Outside the disc, the absolute permeability is set 
to $10^{-10} \, \text{mD}$ (not shown in the figure) and the domain is flat.}
\end{figure}

The saturation maps after 0.06 PVI are in Fig.~\ref{fig:saturation_maps_second_numerical_example}. They have been obtained 
with a constant time step of $\Delta t \approx 0.0021 \, \text{PVI}$ corresponding to a CFL number of 21.1. The saturation pattern
is clearly biased by the orientation of the grid with both 1D-PPU and 1D-IHU. This is the case for all angles, but particularly
for $\theta = \pi/12$ and $\theta = \pi/8$, for which the grid orientation effect is aggravated by the fact that the high-permeability
channels are aligned with the main directions of the grid. This nonphysical behavior is significantly attenuated -- though not completely
eliminated -- with the multidimensional schemes. We note that MultiD-PPU and MultiD-IHU yield similar saturation patterns far from the well,
even though MultiD-PPU better reduces the numerical biasing caused by the orientation
of the grid near the well.

\begin{figure}[ht]
\centering
\begin{tikzpicture}
  \node[anchor=south west,inner sep=0] at (0,0) {\includegraphics[scale=0.162]{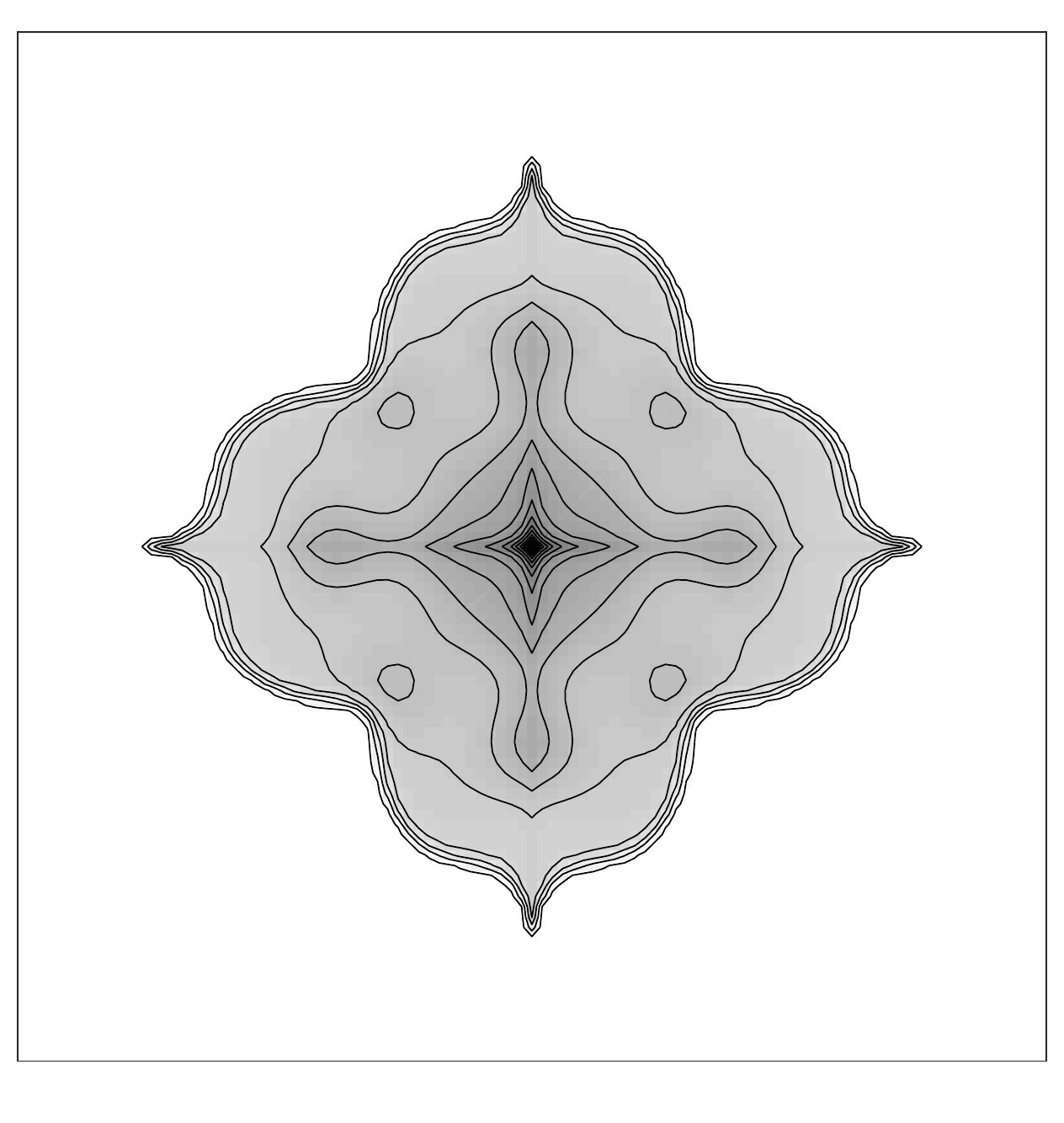}};
  \draw[gray,thick,dashed] (1.05,1.14) circle (1.01cm);
  \node (ib_4) at (1.045,1.135) {\tiny $\textcolor{white}{\star}$};
  \node[anchor=south west,inner sep=0, rotate=15] at (0.32,-2.65) {\includegraphics[scale=0.162]{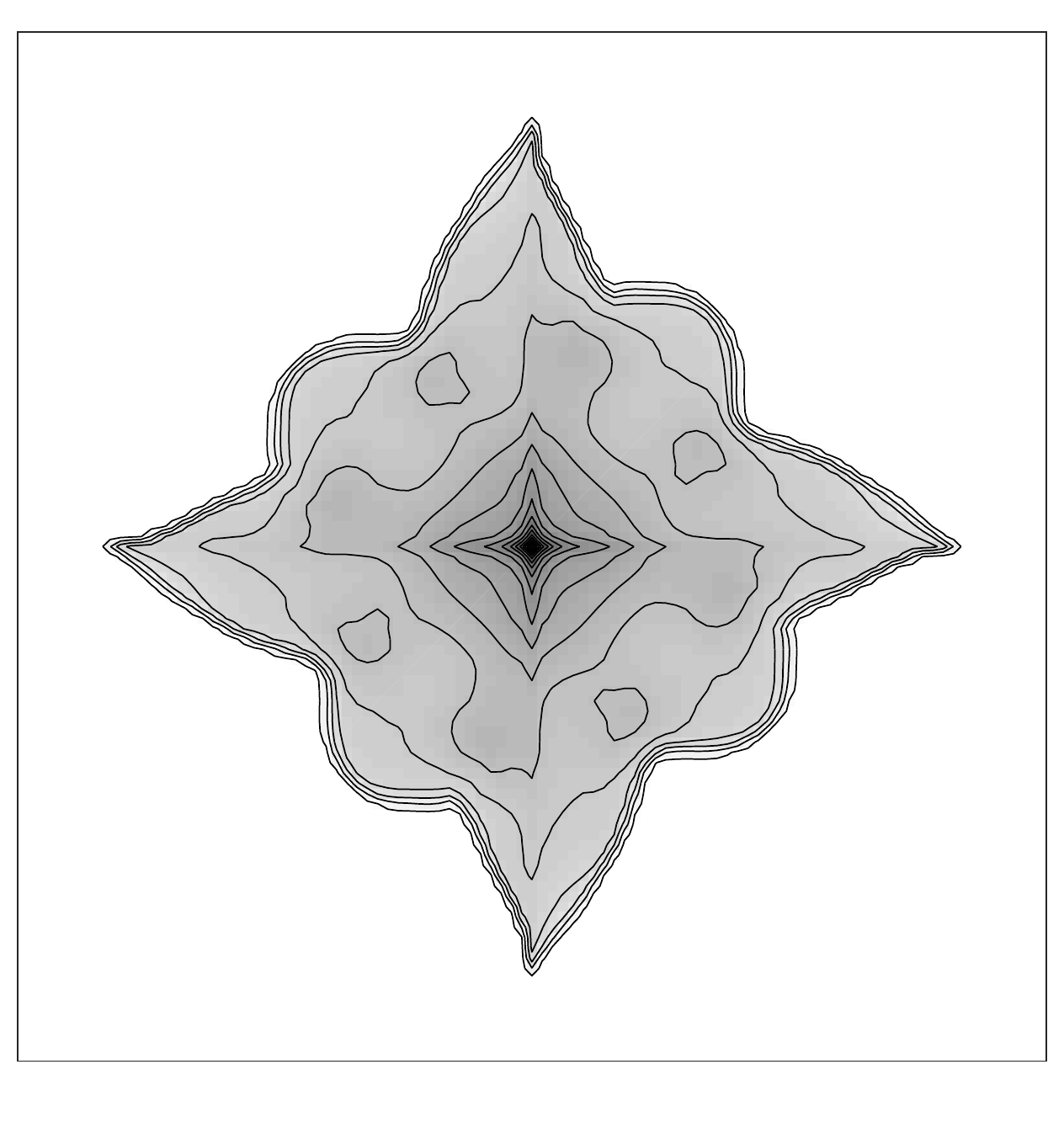}};
  \draw[gray,thick,dashed] (1.05,-1.28) circle (1.01cm);
  \node (ib_4) at (1.04,-1.28) {\tiny $\textcolor{white}{\star}$};
  \node[anchor=south west,inner sep=0, rotate=22.5] at (0.52,-5.245) {\includegraphics[scale=0.162]{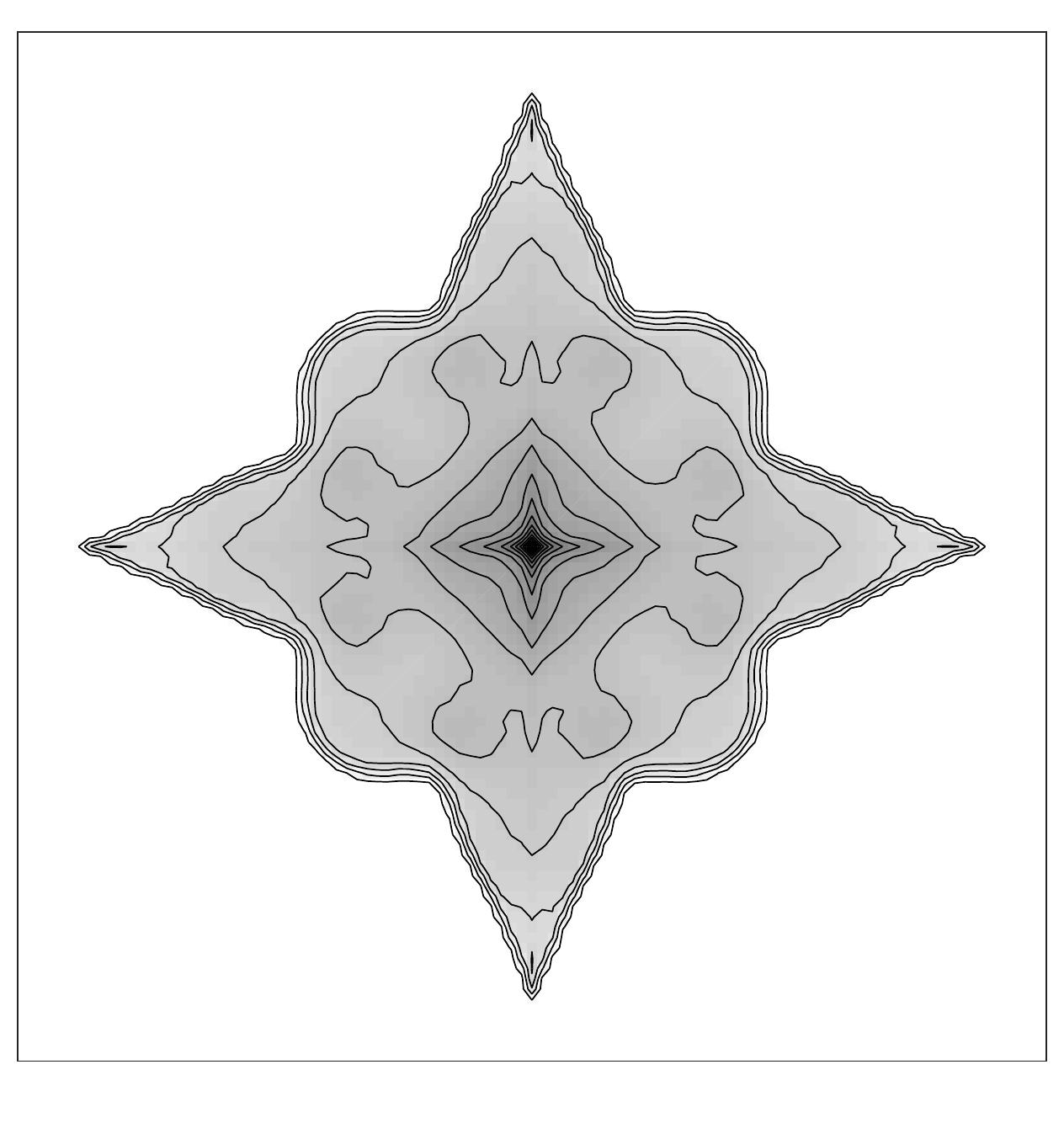}};
  \draw[gray,dashed,thick] (1.05,-3.8) circle (1.01cm);
  \node (ib_4) at (1.05,-3.8) {\tiny $\textcolor{white}{\star}$};
  \node[anchor=south west,inner sep=0, rotate=45] at (1.115,-8.08) {\includegraphics[scale=0.162]{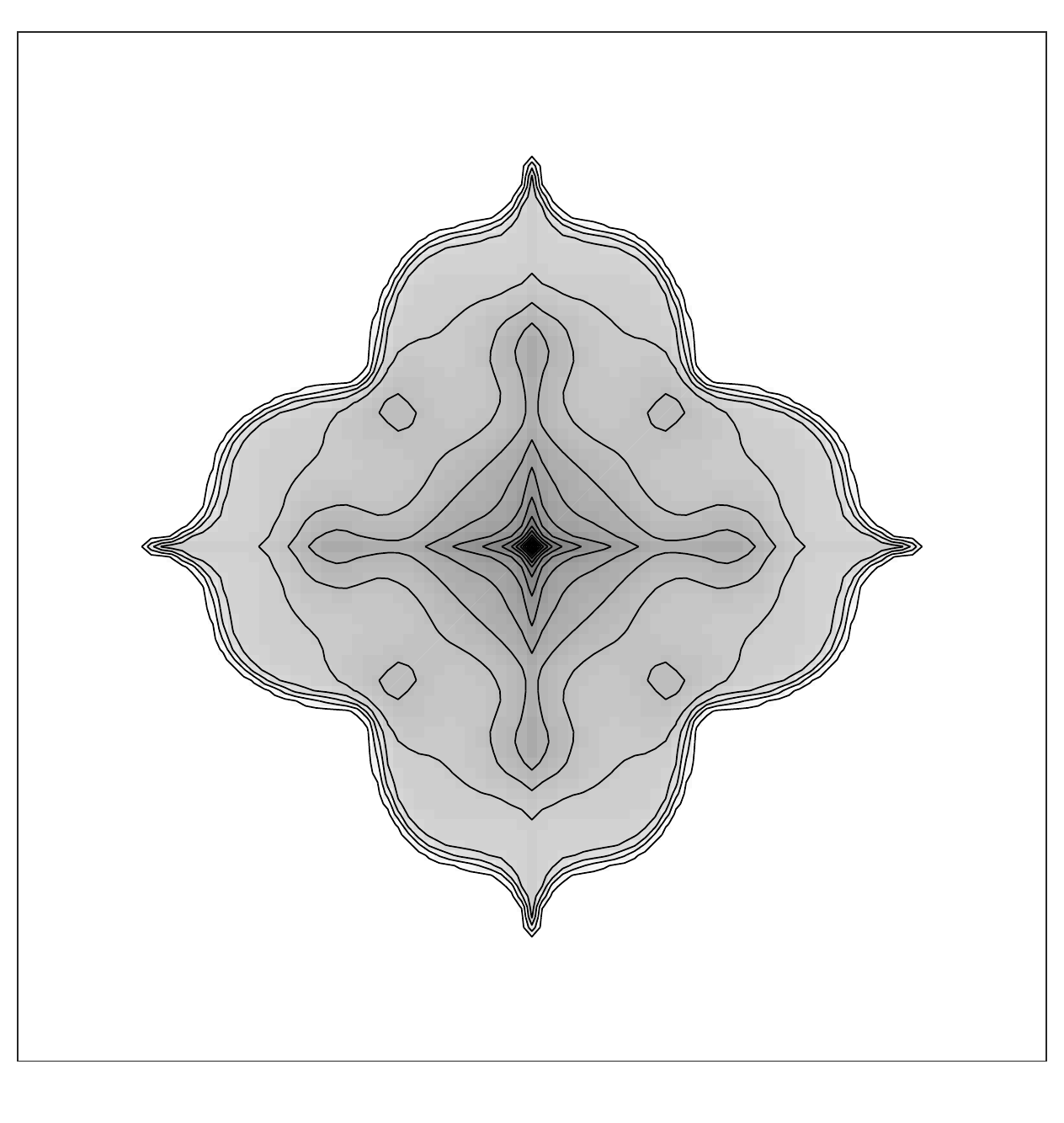}};
  \draw[gray,dashed,thick] (1.05,-6.54) circle (1.01cm);
  \node (ib_4) at (1.05,-6.54) {\tiny $\textcolor{white}{\star}$};
  \node (ib_4) at (1.,2.8) {\small 1D-PPU};
  \node (ib_4) at (-1.2,1.175) {\small $\theta = 0$};
  \node (ib_4) at (-1.2,-1.3) {\small $\theta = \pi/12$};
  \node (ib_4) at (-1.2,-3.9) {\small $\theta = \pi/8$};
  \node (ib_4) at (-1.2,-6.55) {\small $\theta = \pi/4$};
  \end{tikzpicture}
  \hspace{-0.2cm}
  \begin{tikzpicture}
  \node[anchor=south west,inner sep=0] at (0,0) {\includegraphics[scale=0.162]{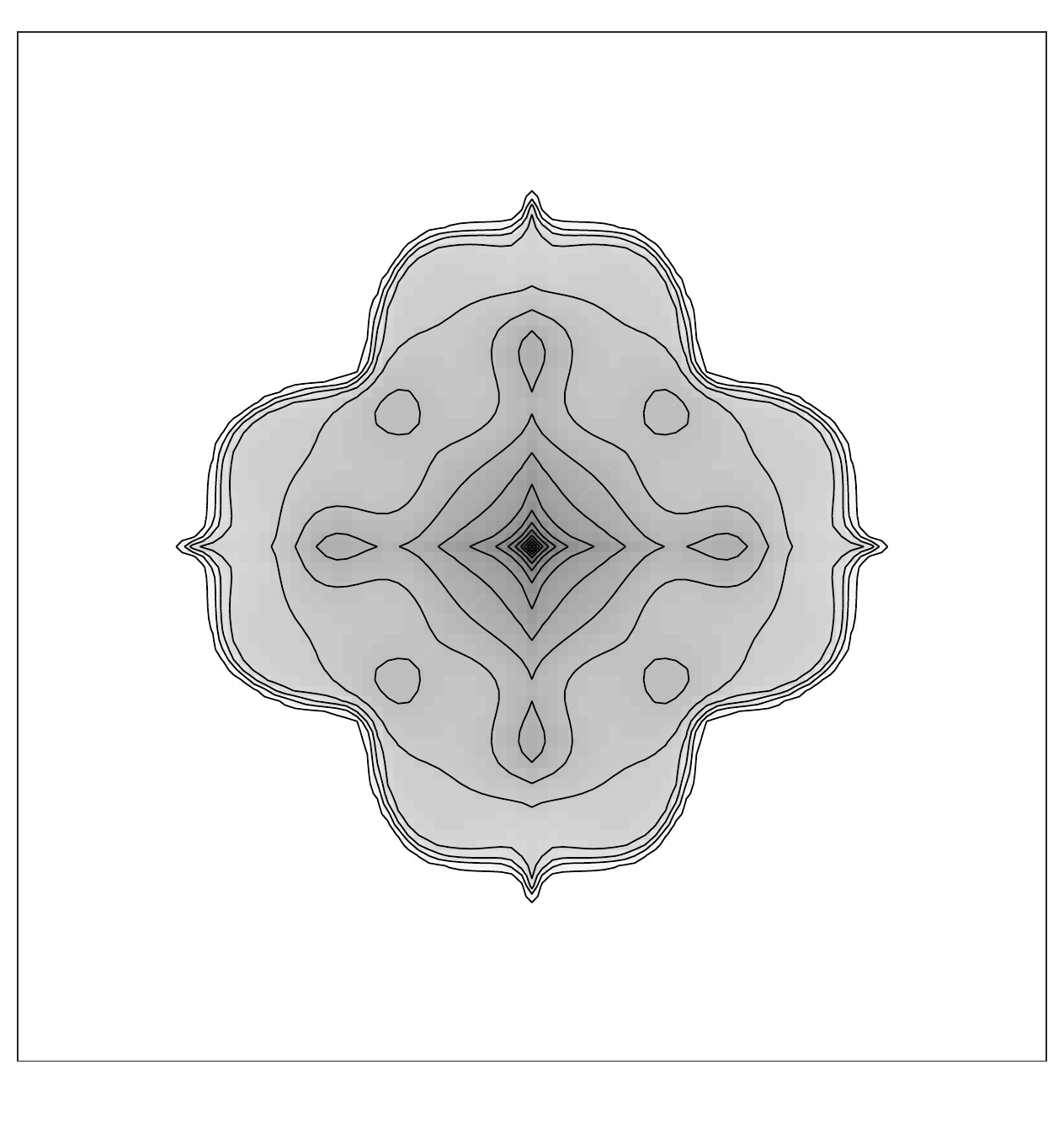}};
  \draw[gray,thick,dashed] (1.05,1.14) circle (1.01cm);
  \node (ib_4) at (1.045,1.135) {\tiny $\textcolor{white}{\star}$};
  \node[anchor=south west,inner sep=0, rotate=15] at (0.32,-2.65) {\includegraphics[scale=0.162]{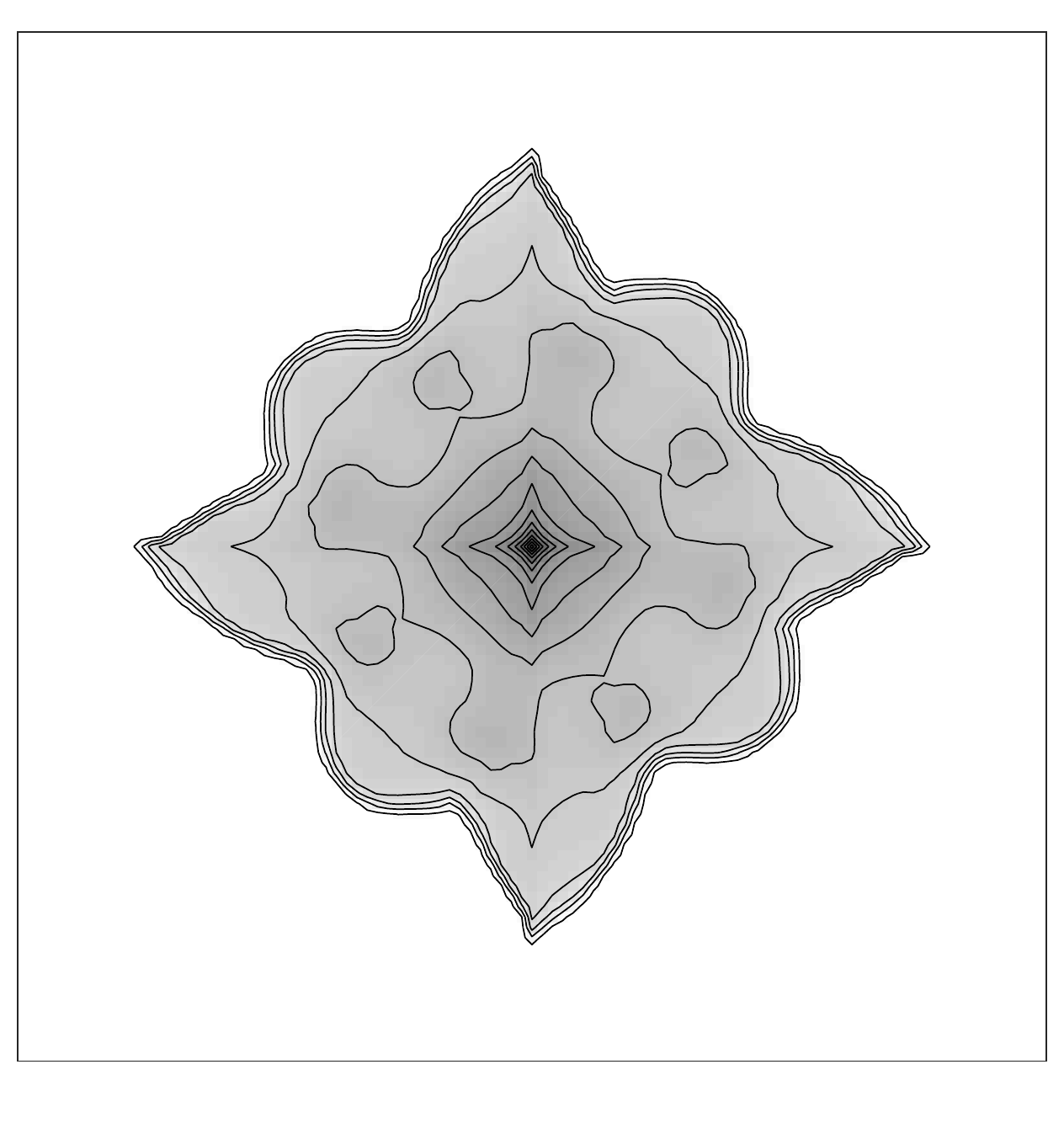}};
  \draw[gray,thick,dashed] (1.05,-1.28) circle (1.01cm);
  \node (ib_4) at (1.04,-1.28) {\tiny $\textcolor{white}{\star}$};
  \node[anchor=south west,inner sep=0, rotate=22.5] at (0.52,-5.245) {\includegraphics[scale=0.162]{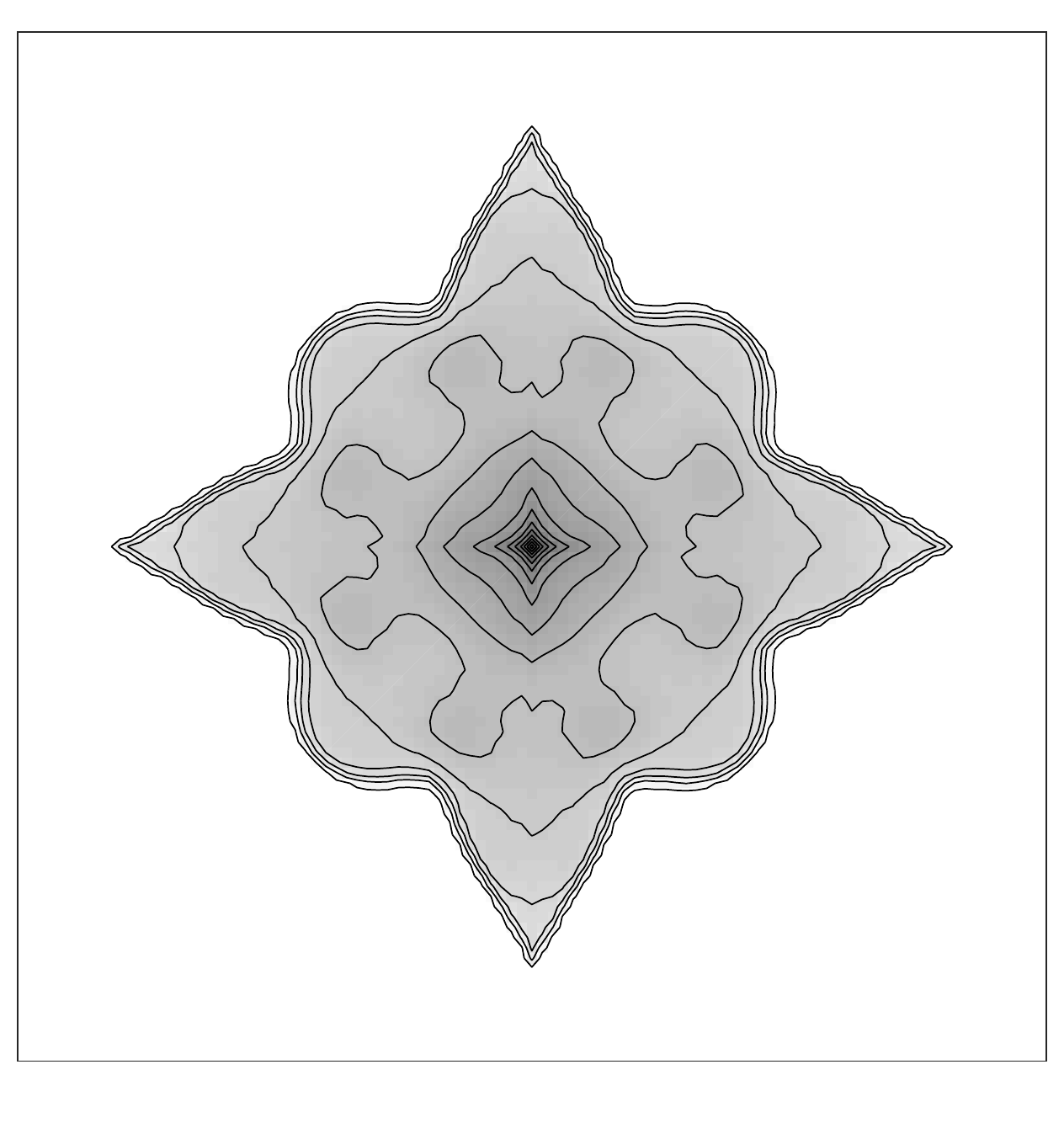}};
  \draw[gray,dashed,thick] (1.05,-3.8) circle (1.01cm);
  \node (ib_4) at (1.05,-3.8) {\tiny $\textcolor{white}{\star}$};
  \node[anchor=south west,inner sep=0, rotate=45] at (1.115,-8.08) {\includegraphics[scale=0.162]{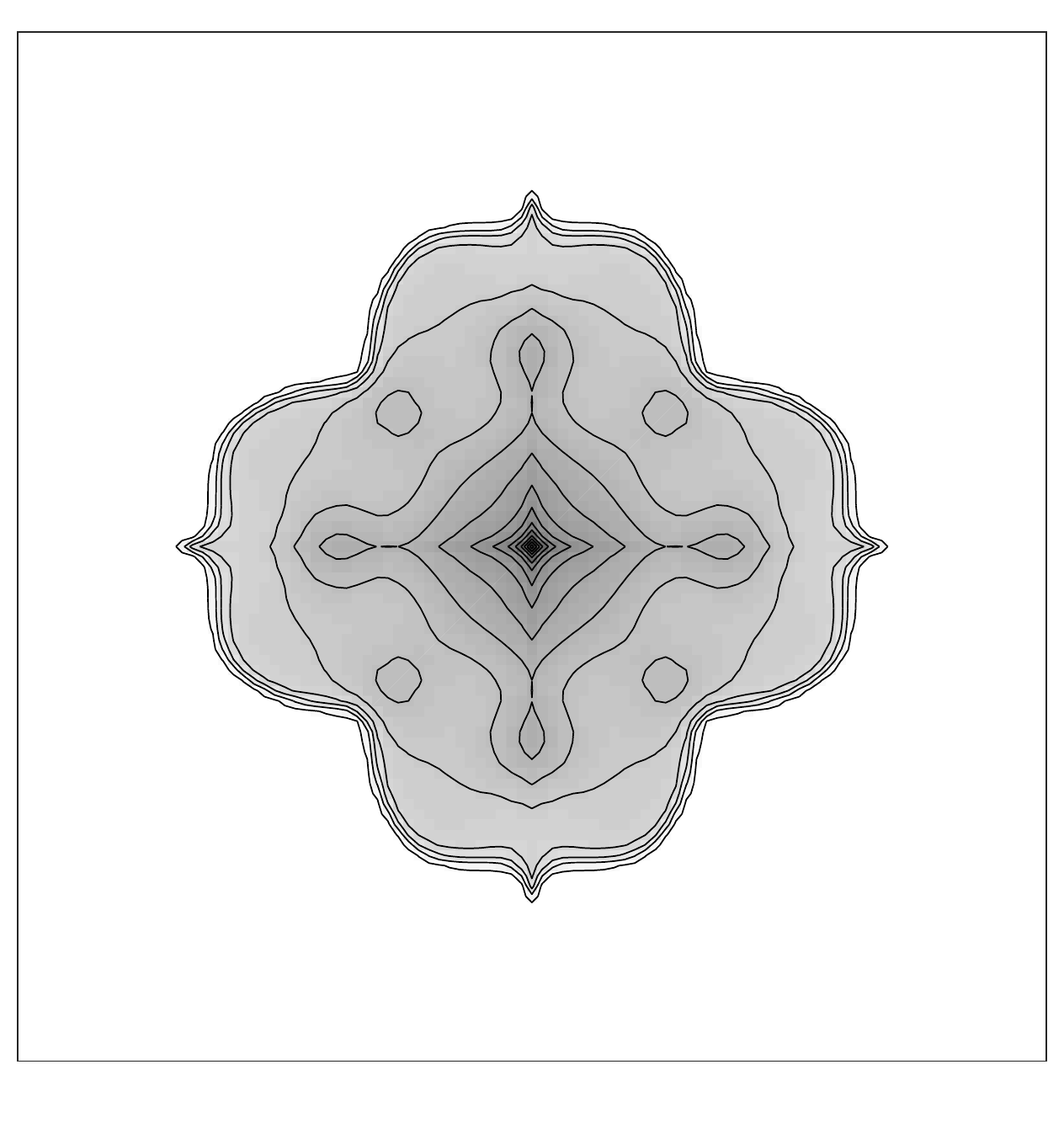}};
  \draw[gray,dashed,thick] (1.05,-6.54) circle (1.01cm);
  \node (ib_4) at (1.05,-6.54) {\tiny $\textcolor{white}{\star}$};
  \node (ib_4) at (1.,2.8) {\small 1D-IHU};
  \end{tikzpicture}
  \hspace{-0.2cm}
  \begin{tikzpicture}
  \node[anchor=south west,inner sep=0] at (0,0) {\includegraphics[scale=0.162]{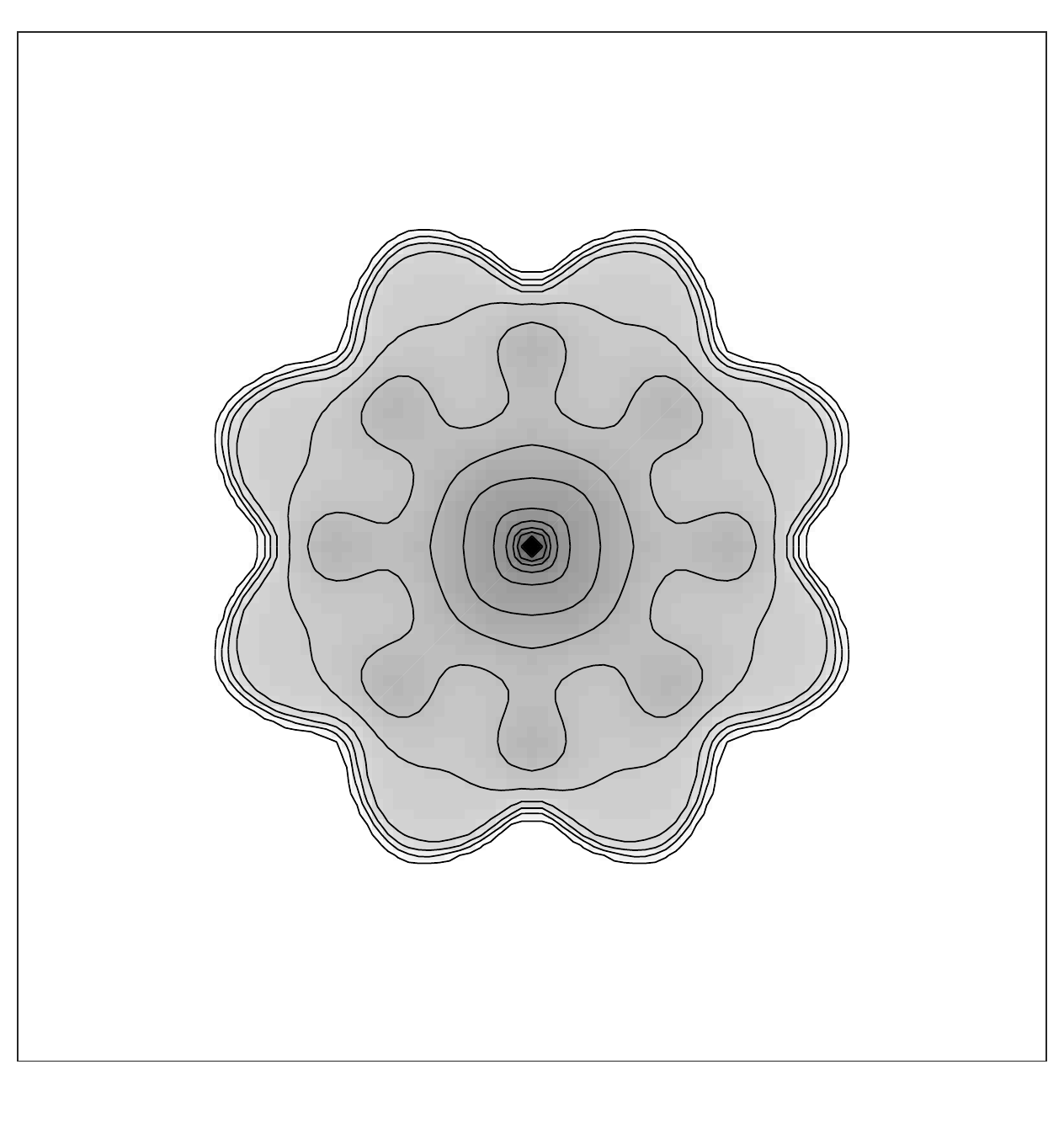}};
  \draw[gray,thick,dashed] (1.05,1.14) circle (1.01cm);
  \node (ib_4) at (1.045,1.135) {\tiny $\textcolor{white}{\star}$};
  \node[anchor=south west,inner sep=0, rotate=15] at (0.32,-2.65) {\includegraphics[scale=0.162]{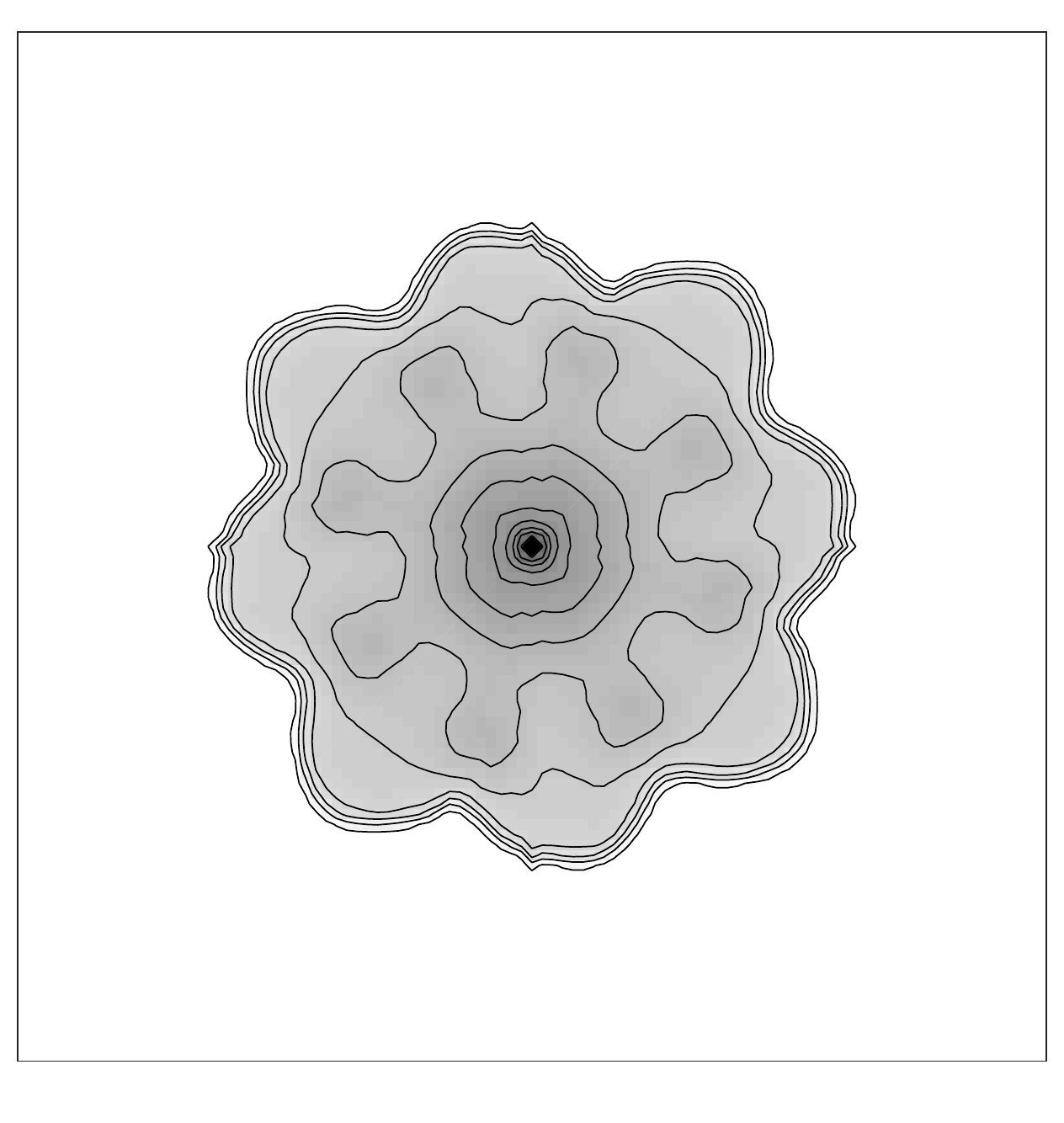}};
  \draw[gray,thick,dashed] (1.05,-1.28) circle (1.01cm);
  \node (ib_4) at (1.04,-1.28) {\tiny $\textcolor{white}{\star}$};
  \node[anchor=south west,inner sep=0, rotate=22.5] at (0.52,-5.245) {\includegraphics[scale=0.162]{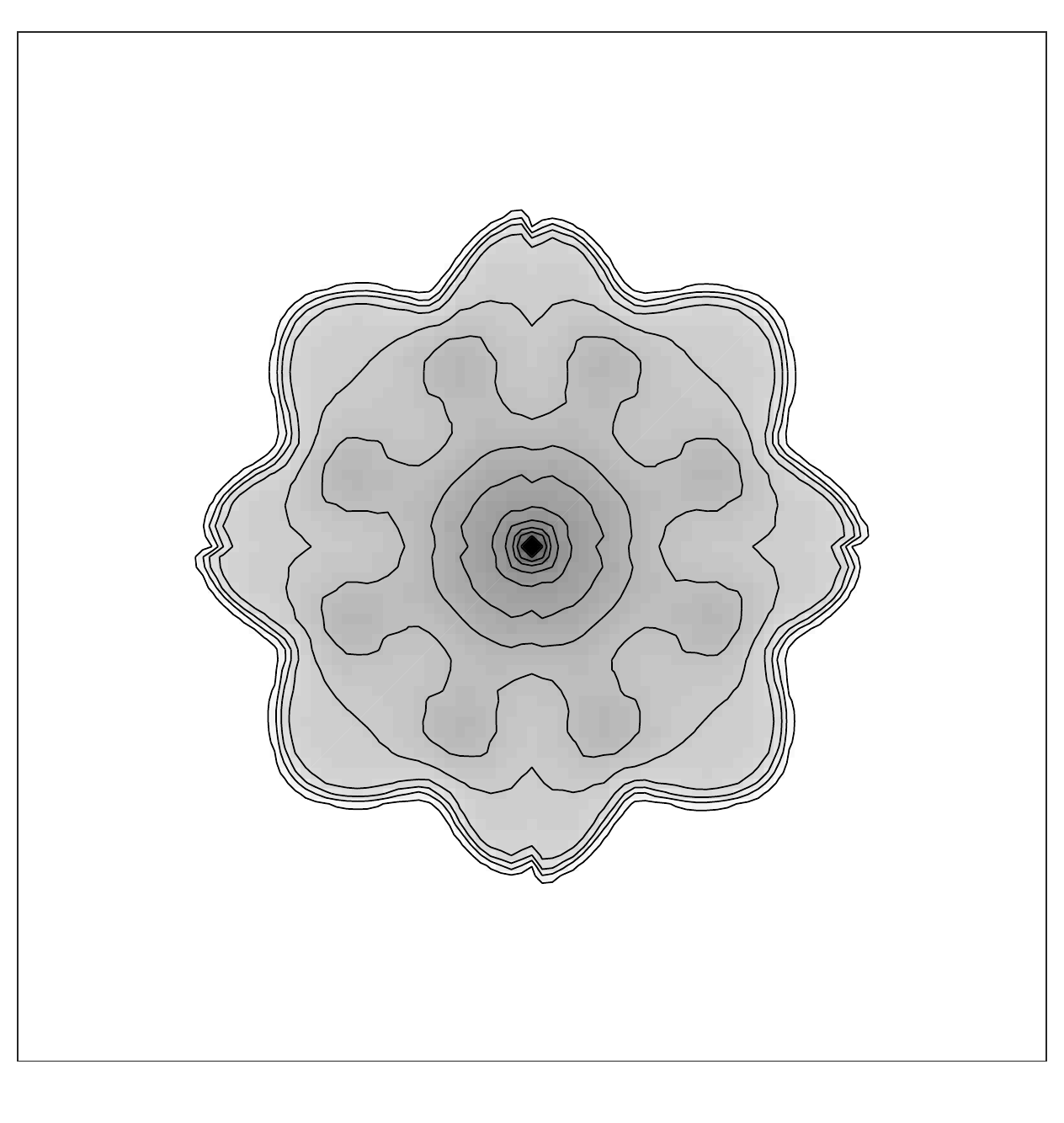}};
  \draw[gray,dashed,thick] (1.05,-3.8) circle (1.01cm);
  \node (ib_4) at (1.05,-3.8) {\tiny $\textcolor{white}{\star}$};
  \node[anchor=south west,inner sep=0, rotate=45] at (1.115,-8.08) {\includegraphics[scale=0.162]{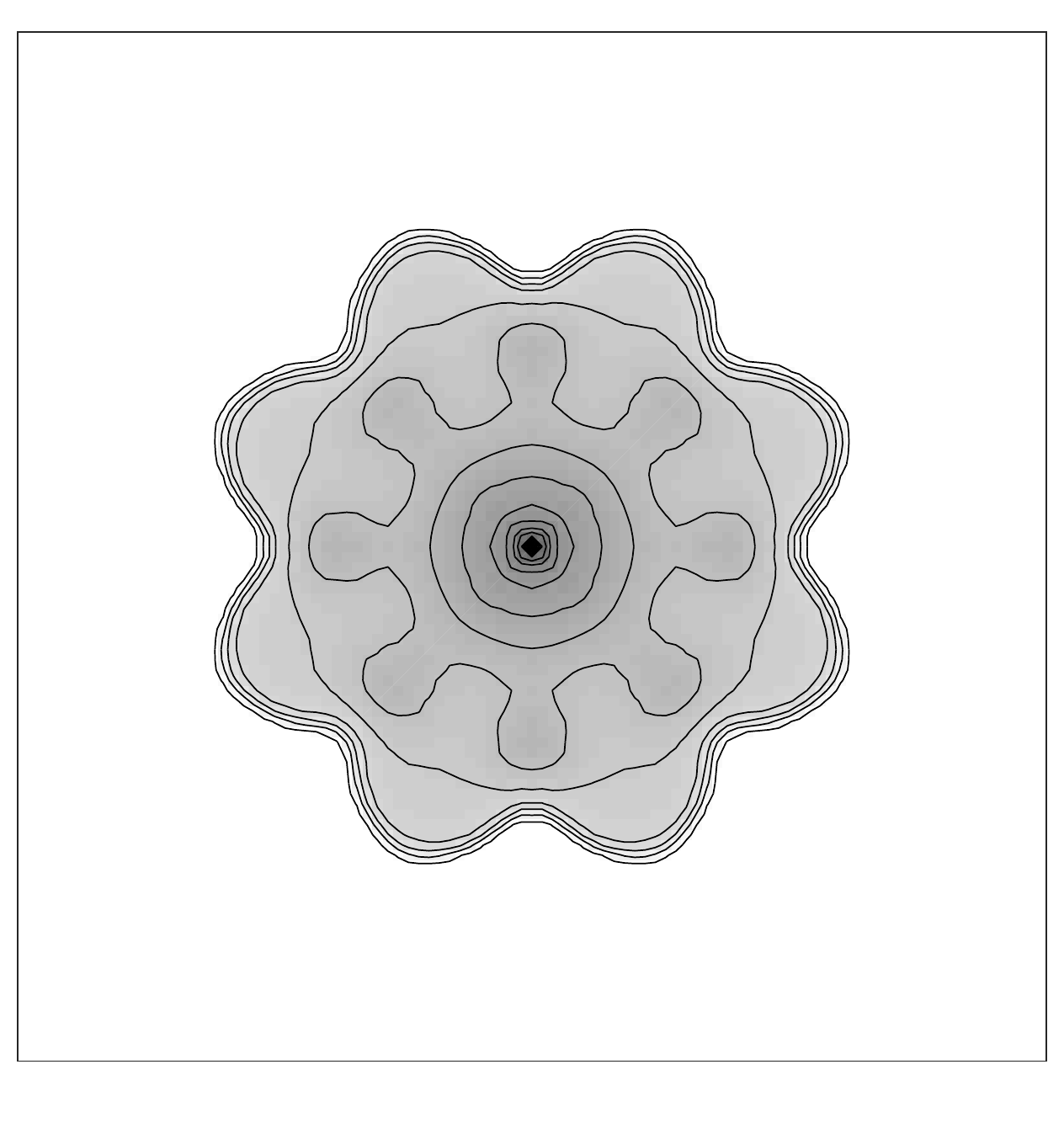}};
  \draw[gray,dashed,thick] (1.05,-6.54) circle (1.01cm);
  \node (ib_4) at (1.05,-6.54) {\tiny $\textcolor{white}{\star}$};
  \node (ib_4) at (1.05,2.8) {\small MultiD-PPU};
  \node (ib_4) at (1.05,2.451) {\small (SMU)};
  \end{tikzpicture}
  \hspace{-0.2cm}
  \begin{tikzpicture}
  \node[anchor=south west,inner sep=0] at (0,0) {\includegraphics[scale=0.162]{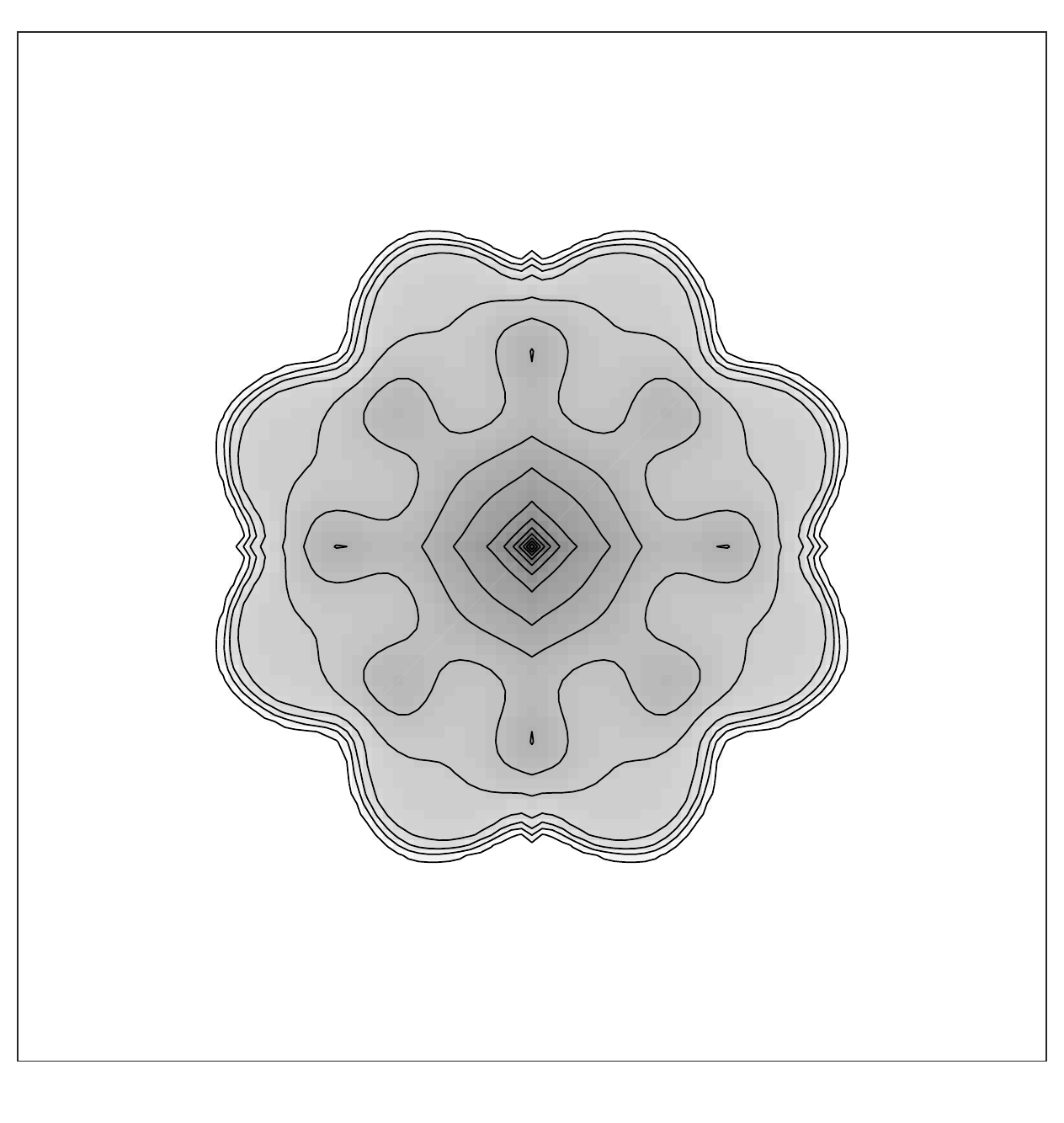}};
  \draw[gray,thick,dashed] (1.05,1.14) circle (1.01cm);
  \node (ib_4) at (1.045,1.135) {\tiny $\textcolor{white}{\star}$};
  \node[anchor=south west,inner sep=0, rotate=15] at (0.32,-2.65) {\includegraphics[scale=0.162]{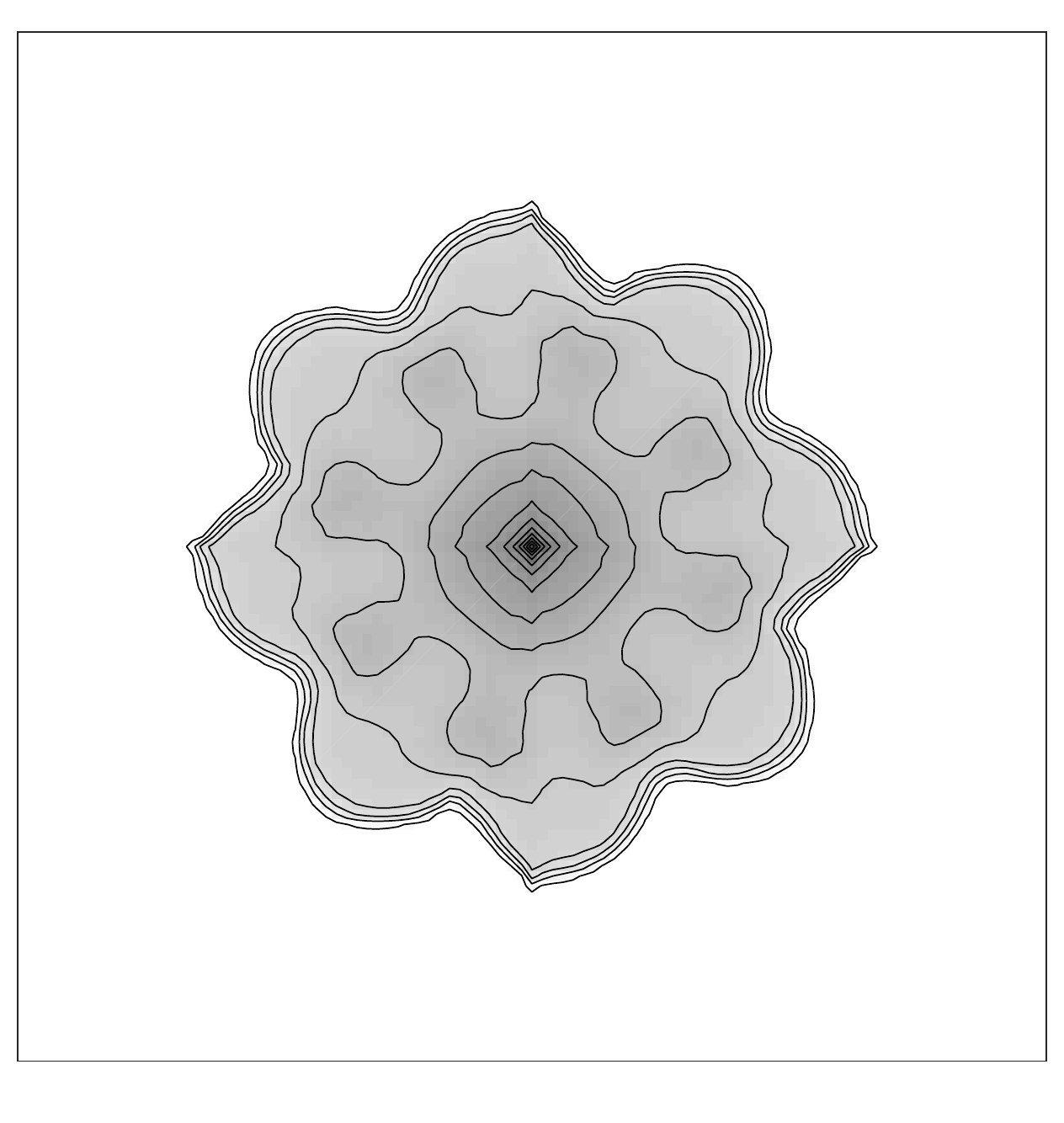}};
  \draw[gray,thick,dashed] (1.05,-1.28) circle (1.01cm);
  \node (ib_4) at (1.04,-1.28) {\tiny $\textcolor{white}{\star}$};
  \node[anchor=south west,inner sep=0, rotate=22.5] at (0.52,-5.245) {\includegraphics[scale=0.162]{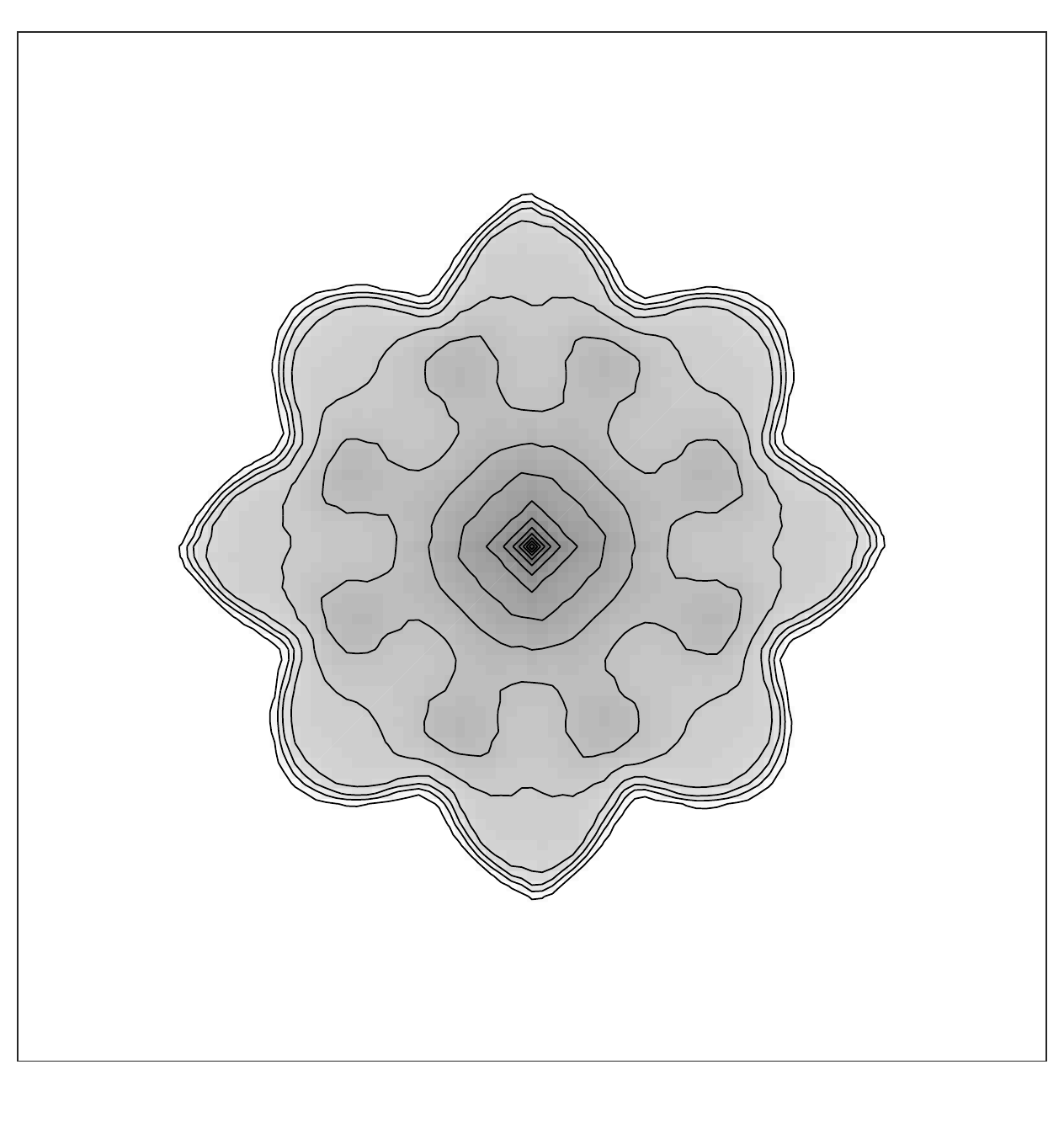}};
  \draw[gray,dashed,thick] (1.05,-3.8) circle (1.01cm);
  \node (ib_4) at (1.05,-3.8) {\tiny $\textcolor{white}{\star}$};
  \node[anchor=south west,inner sep=0, rotate=45] at (1.115,-8.08) {\includegraphics[scale=0.162]{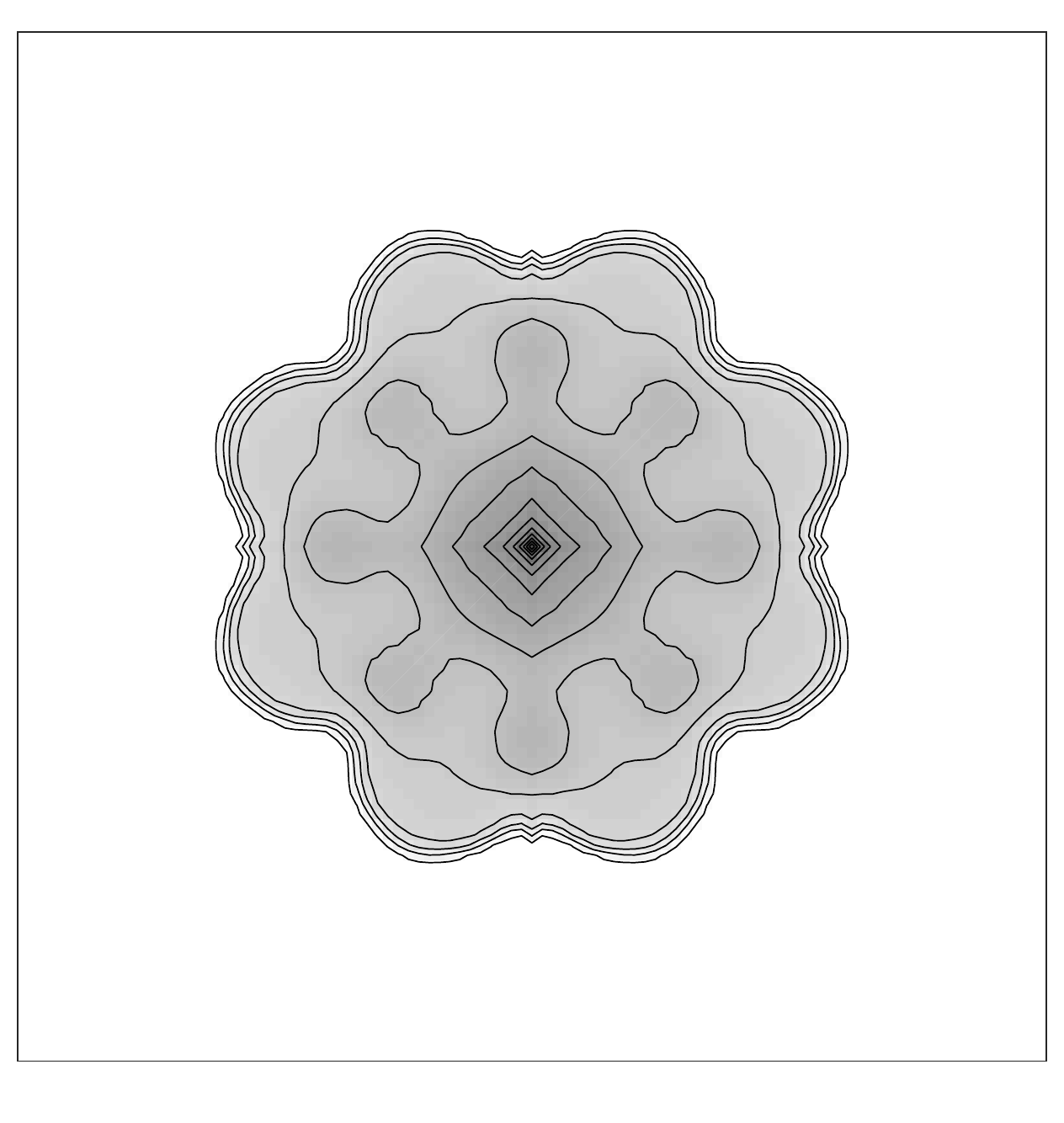}};
  \draw[gray,dashed,thick] (1.05,-6.54) circle (1.01cm);
  \node (ib_4) at (1.05,-6.54) {\tiny $\textcolor{white}{\star}$};
  \node (ib_4) at (1.05,2.8) {\small MultiD-IHU};
  \node (ib_4) at (1.05,2.451) {\small (SMU4)};
\end{tikzpicture}
  \caption{\label{fig:saturation_maps_second_numerical_example}Saturation maps for different values of the angle $\theta$
    after 0.06 PVI for the heterogeneous two-phase problem. The CFL number of these simulations is 21.1.
    The white stars show the location of the injector.
  }
\end{figure}

The nonlinear behavior of the four schemes is detailed in Table \ref{tab:nonlinear_behavior_heterogeneous_injection} for two constant
time step sizes, namely $\Delta t \approx 0.0021 \, \text{PVI}$ corresponding to CFL number of 21.1, and $\Delta t \approx 0.01 \, \text{PVI}$
for a CFL number of 106.1. The results are consistent with those of Section \ref{subsection_three_well_problem_with_buoyancy}. Specifically,
MultiD-IHU is the most efficient scheme in terms of Newton iterations, with a total reduction of 10.4\% compared to 1D-PPU for short
time steps, and of 13.7\% for larger time steps.

\begin{table}[!ht]
\scalebox{0.85}{
\centering
         \begin{tabular}{ccccccccc}
           \\ \toprule
            Angle     & \multicolumn{4}{c}{Small CFL}        & \multicolumn{4}{c}{Large CFL}  \\ 
            $\theta$  & 1D-PPU & 1D-IHU & MultiD-PPU & MultiD-IHU    & 1D-PPU & 1D-IHU & MultiD-PPU & MultiD-IHU \\ \toprule
              0        &   167     &   161     &     153       &    146           &    69    &    72   &       62     &  58 \\
              $\pi/12$ &   166     &   161     &     154       &    151           &    69    &    70   &       61     &  59 \\
              $\pi/8$  &   167     &   161     &     154       &    153           &    72    &    71   &       63     &  62 \\
              $\pi/6$  &   167     &   161     &     156       &    151           &    73    &    71   &       62     &  62 \\
              $\pi/4$  &   166     &   161     &     154       &    145           &    68    &    71   &       61     &  62 \\ \bottomrule 
         \end{tabular}}
 \caption{\label{tab:nonlinear_behavior_heterogeneous_injection} Total number of Newton iterations
for different angles at $T = 0.06 \, \text{PVI}$ in the heterogeneous injection problem with buoyancy.}
\end{table}

\subsection{\label{subsection_gravity_segregation_with_low_perm_barriers}Gravity segregation with low-permeability barriers}

Finally, we consider a $x-y$ domain of dimensions $[ -75 \, \text{ft}, 75 \, \text{ft} ] \times [ - 75  \, \text{ft}, 75 \, \text{ft}]$, 
in which low-permeability layers are placed to slow down the upward migration of the lighter non-wetting phase. The domain
is tilted with an angle of $\pi/3$ in the direction that is perpendicular to the low-permeability layers. We use a uniform
Cartesian grid consisting of $101 \times 101$ control volumes. In the disc of radius 
$r_0 = (150 - \Delta x)/2 \, \text{ft}$ that occupies the center of domain, the absolute permeability is $5 \times 10^{-9} \, \text{mD}$ in 
the horizontal layers and $50 \, \text{mD}$ everywhere else. Outside the disc, the absolute permeability is also equal to 
$5 \times 10^{-9} \, \text{mD}$. The absolute permeability field is shown in Fig.~\ref{fig:permeability_map_third_numerical_example}.
We use the same technique as in Sections \ref{subsection_three_well_problem_with_buoyancy} and \ref{subsection_second_numerical_example} 
to compute the location of the low-permeability layers as we rotate the grid. 

\begin{figure}[ht]
\centering
  \begin{tikzpicture}
    \node[anchor=south west,inner sep=0] at (0,0) {\includegraphics[scale=0.25]{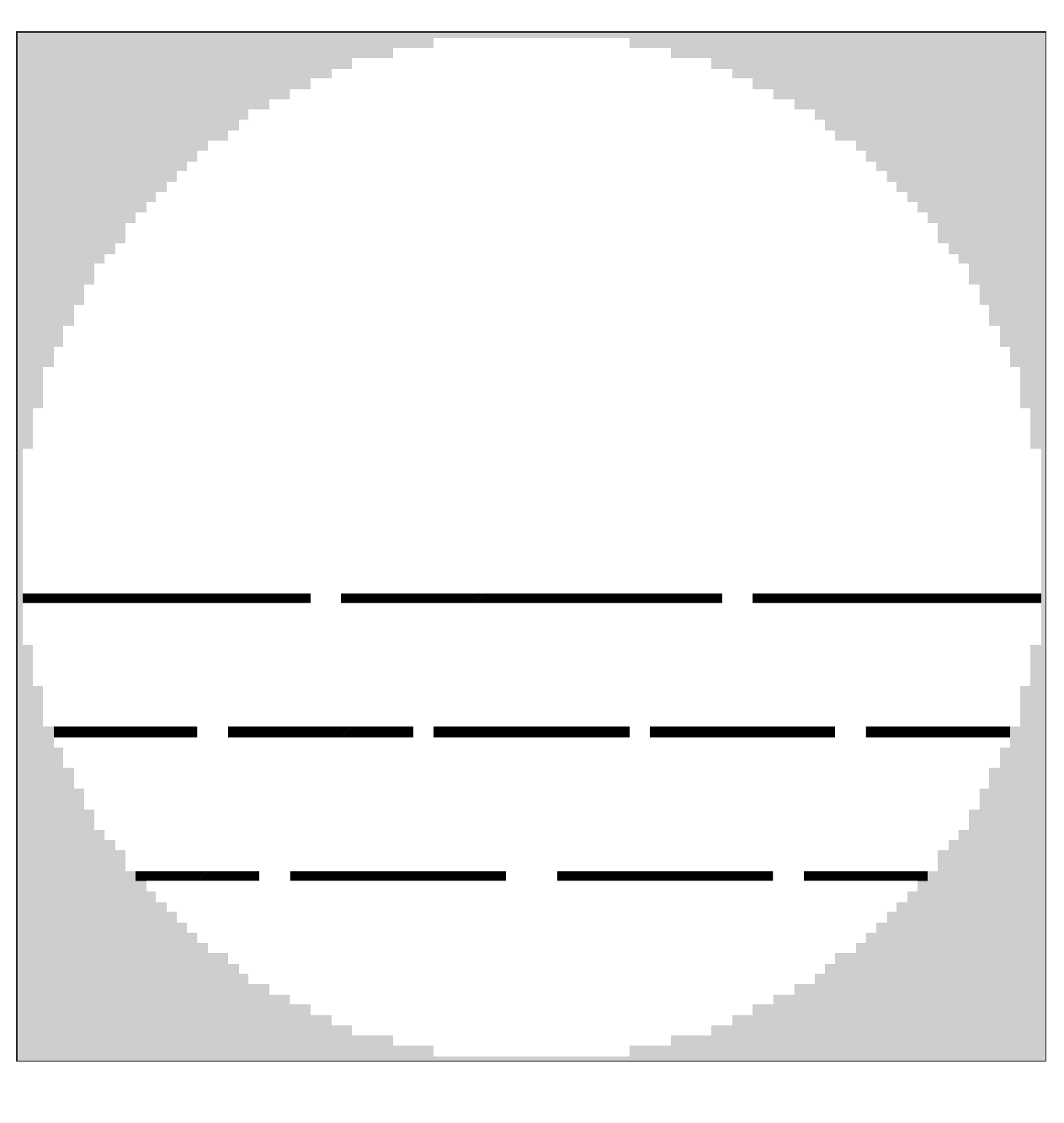}};
    \node (ib_4) at (1.6,0) {$x$};
    \node (ib_4) at (-0.2,1.75) {$y$};

    \node (a) at (4,3) {}; 
    \node (b) at (4,1.5) {}; 
    \node (ib_4) at (4.2,2.25) {$\vec{\boldsymbol{g}}$};
    \path[->,draw=black,thick] (a) edge node {} (b);  
  \end{tikzpicture}
\caption{\label{fig:permeability_map_third_numerical_example}Absolute permeability field with low-permeability
layers for an angle $\theta = 0$. Outside the disc of radius $r_0$ centered in $(x,y) = (0,0)$, the absolute permeability 
is set to $5 \times 10^{-9} \, \text{mD}$.}
\end{figure}
The non-wetting phase is initially saturating the bottom of the domain. Specifically, the initial saturation field is given by
\begin{equation}
S^{\text{init}}_{\textit{nw}}(x,y) = \left\{
\begin{array}{l l}
1  & \text{if } \, -75 < y' < -45 \, \, \text{and} \, \, \sqrt{ x^2 + y^2 } \leq r_0 \\
0  & \text{otherwise.} 
\end{array} \right.
\end{equation}
where the $y'$-axis is increasing from $-75 \, \text{ft}$ to
$75 \, \text{ft}$ when going from bottom to top. The phase densities
are given by
$\rho_{\textit{nw}} = 32 \, \text{lb}_m . \text{ft}^{-3}$ and
$\rho_w = 64 \, \text{lb}_m . \text{ft}^{-3}$.
The constant phase viscosities are chosen to be
$\mu_{\textit{nw}} = 2 \, \text{cP}$ and $\mu_w = 1 \, \text{cP}$.
Finally, we use Corey-type relative permeabilities such that
$k_{r\textit{nw}}(S) = (1-S)^2$ and $k_{rw}(S) = S^{1.5}$. 

Figure \ref{fig:saturation_maps_gravity_segregation_with_low_perm_layers} compares the saturation maps produced by the different 
schemes after 6000 days once the plume has reached the top
of the domain. As in the previous examples, the truly
multidimensional schemes drastically reduce the grid orientation
effect compared to the schemes based on two-point upwinding.
This is particularly visible in the upper part of the domain
where the shape of the plume predicted by 1D-PPU and 1D-IHU
varies significantly as the grid is rotated and becomes very
large for an angle $\theta = \pi/4$. For this example, the
shape of the plume is best preserved for all grid orientations
by MultiD-IHU.

\begin{figure}[ht]
\centering
  \begin{tikzpicture}
    \node[anchor=south west,inner sep=0] at (0,0) {\includegraphics[scale=0.162]{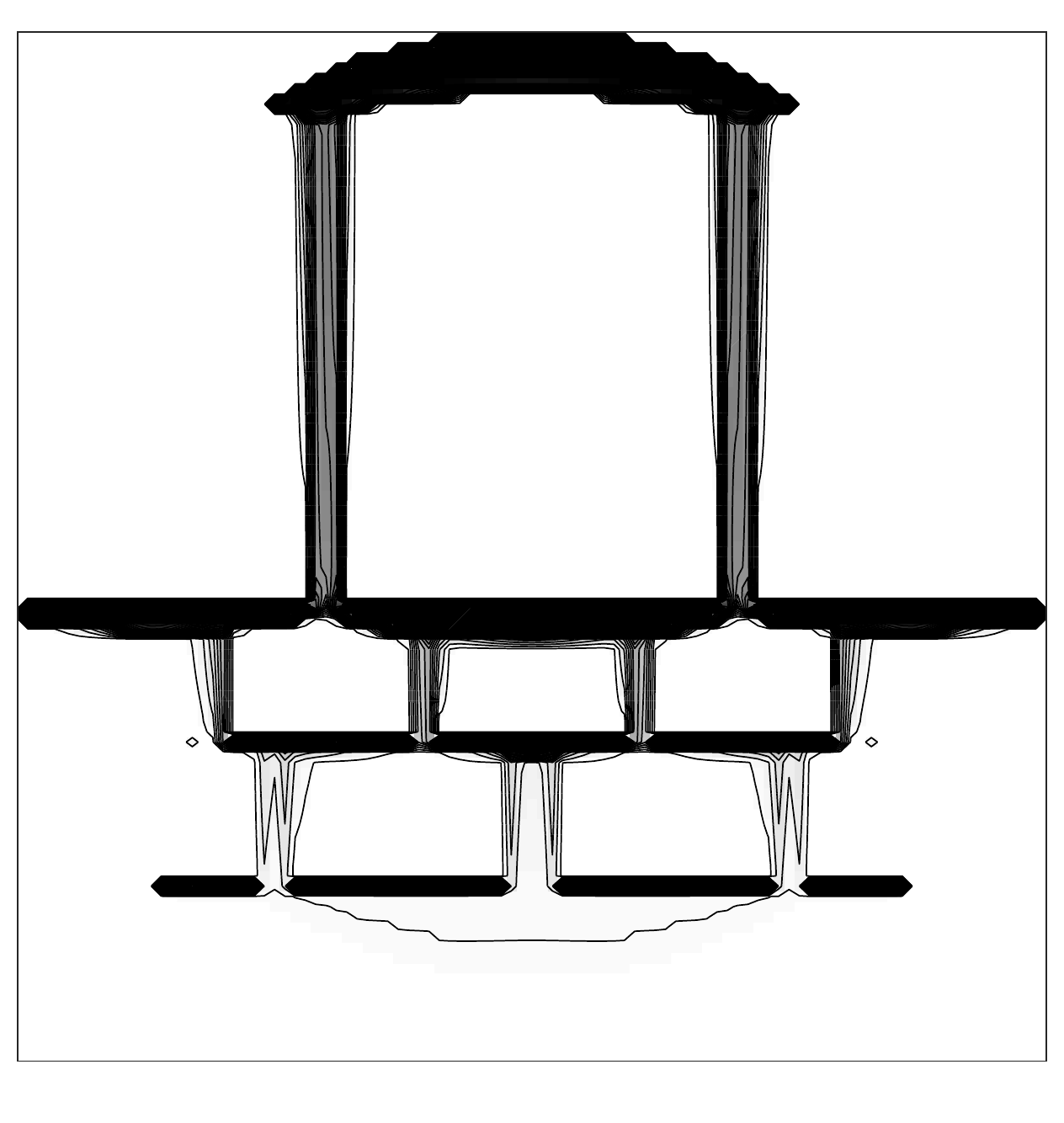}};
    \draw[gray,thick,dashed] (1.05,1.14) circle (1.01cm);
    \node[anchor=south west,inner sep=0, rotate=22.5] at (0.52,-2.945) {\includegraphics[scale=0.162]{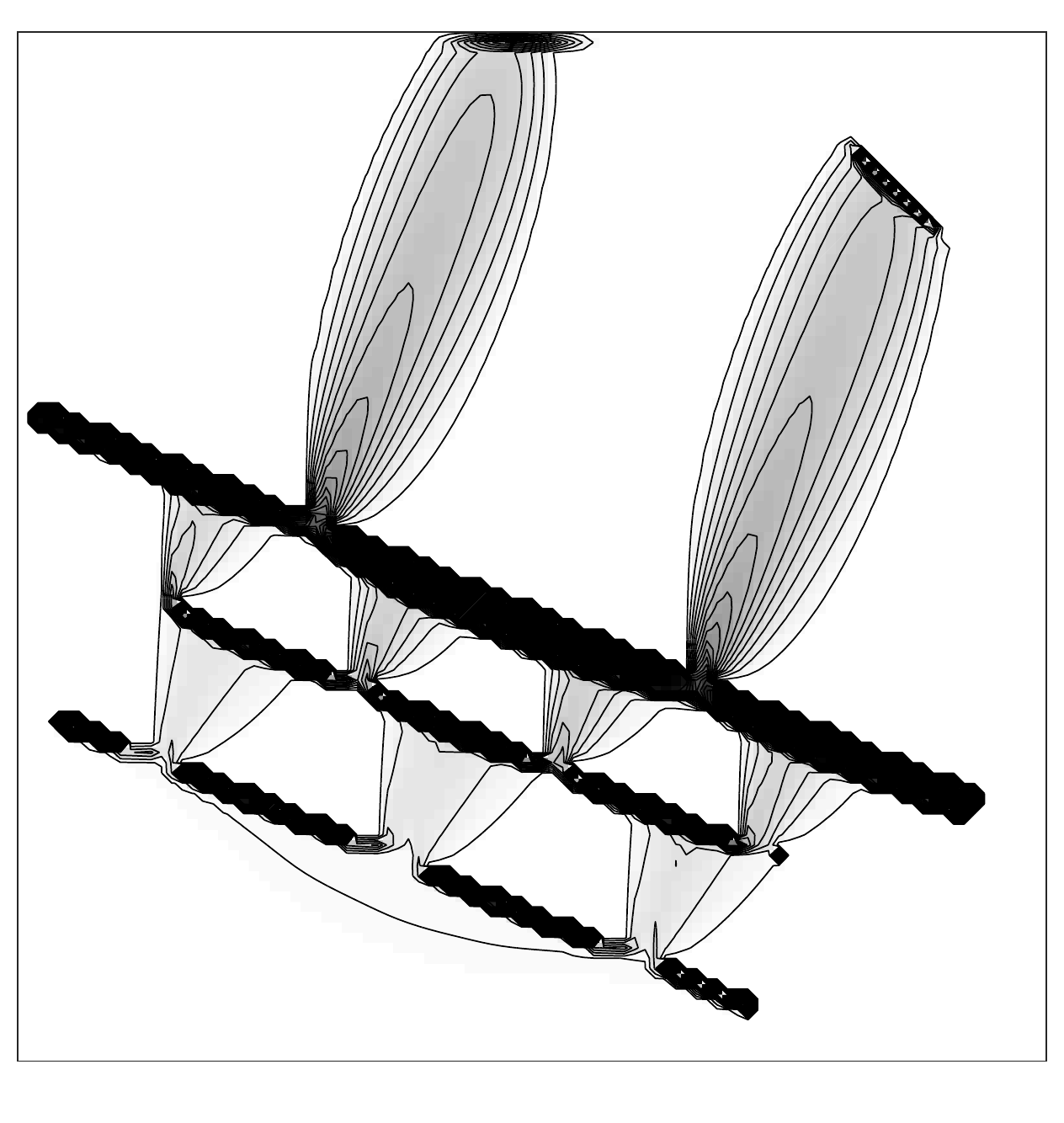}};
    \draw[gray,dashed,thick] (1.05,-1.5) circle (1.01cm);
    \node[anchor=south west,inner sep=0, rotate=45] at (1.115,-5.68) {\includegraphics[scale=0.162]{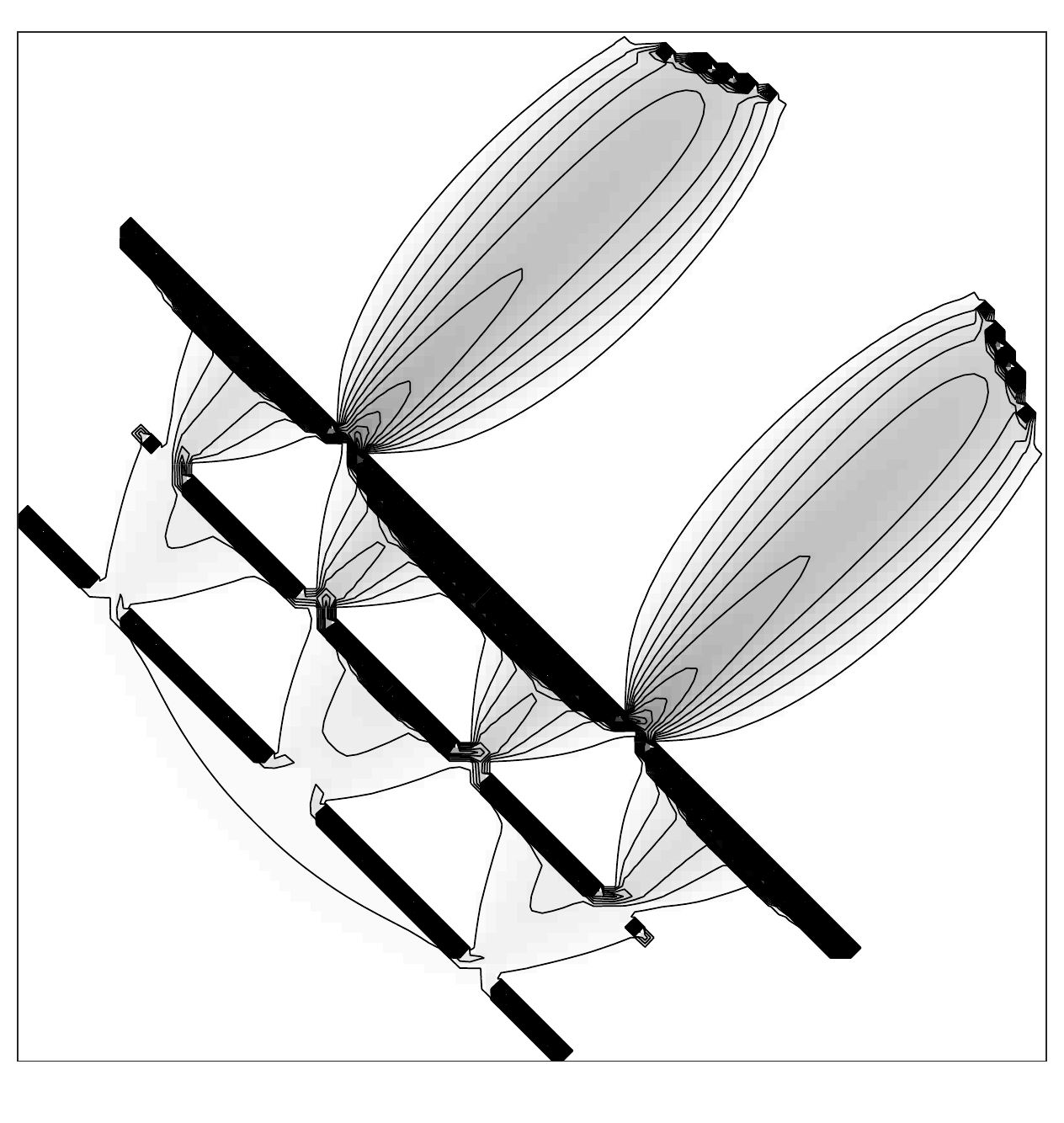}};
    \draw[gray,dashed,thick] (1.05,-4.14) circle (1.01cm);
    \node (ib_4) at (1.,2.8) {\small 1D-PPU};
    \node (ib_4) at (-1.2,1.15) {\small $\theta = 0$};
    \node (ib_4) at (-1.2,-1.45) {\small $\theta = \pi/8$};
    \node (ib_4) at (-1.2,-4.05) {\small $\theta = \pi/4$};
  \end{tikzpicture}
  \hspace{-0.2cm}
  \begin{tikzpicture}
    \node[anchor=south west,inner sep=0] at (0,0) {\includegraphics[scale=0.162]{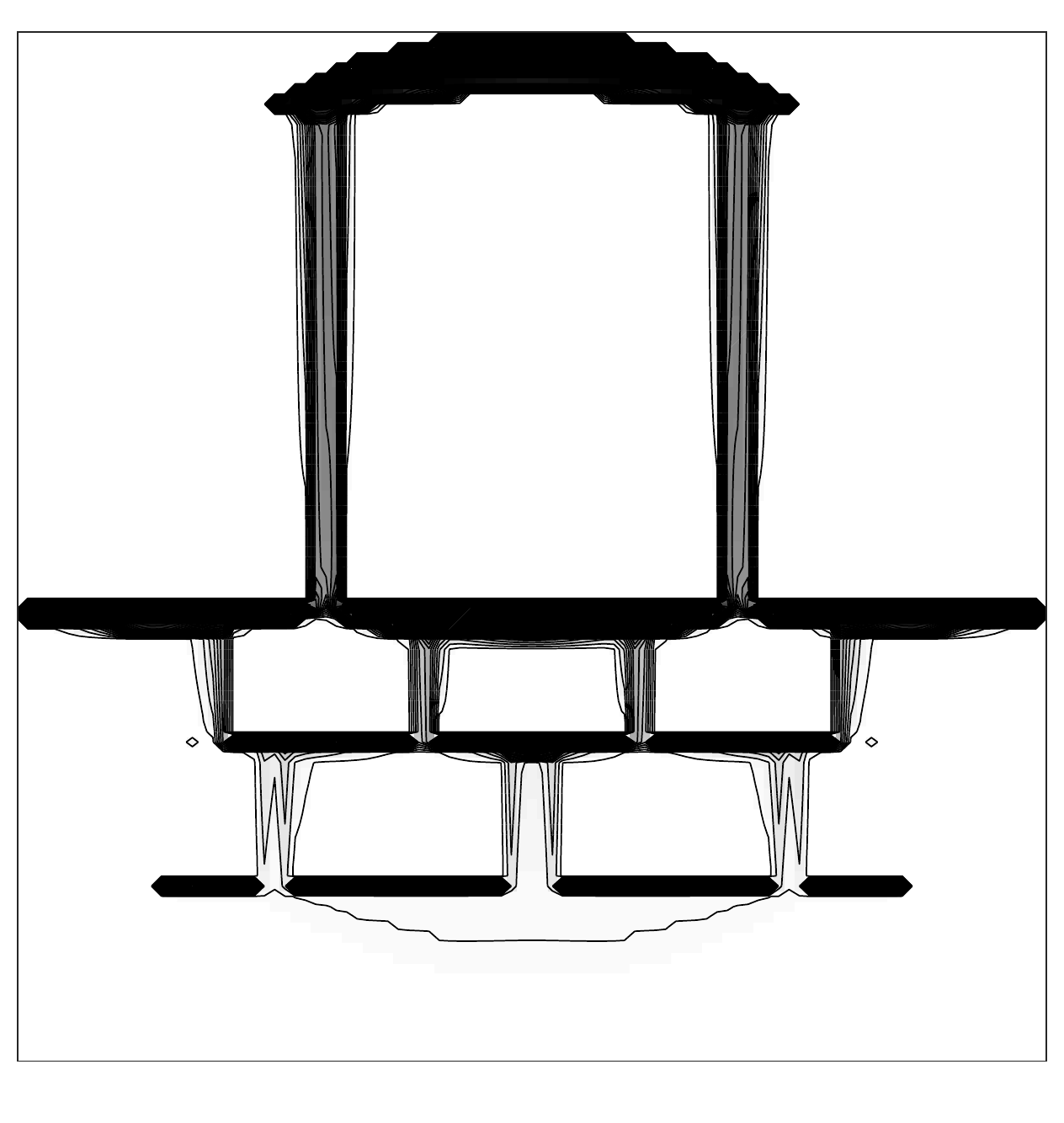}};
    \draw[gray,thick,dashed] (1.05,1.14) circle (1.01cm);
    \node[anchor=south west,inner sep=0, rotate=22.5] at (0.52,-2.945) {\includegraphics[scale=0.162]{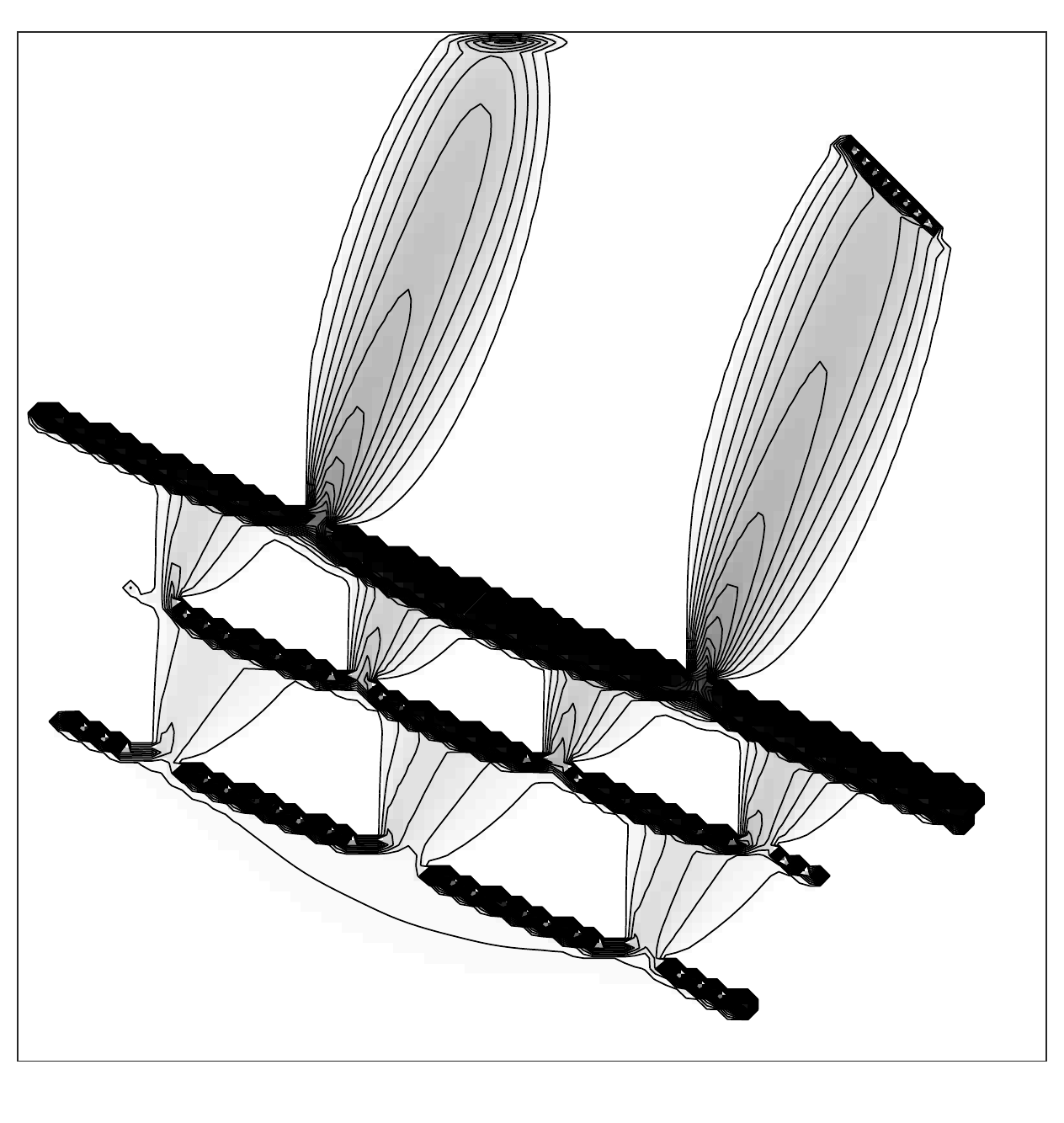}};
    \draw[gray,thick,dashed] (1.05,-1.5) circle (1.01cm);
    \node[anchor=south west,inner sep=0, rotate=45] at (1.115,-5.68) {\includegraphics[scale=0.162]{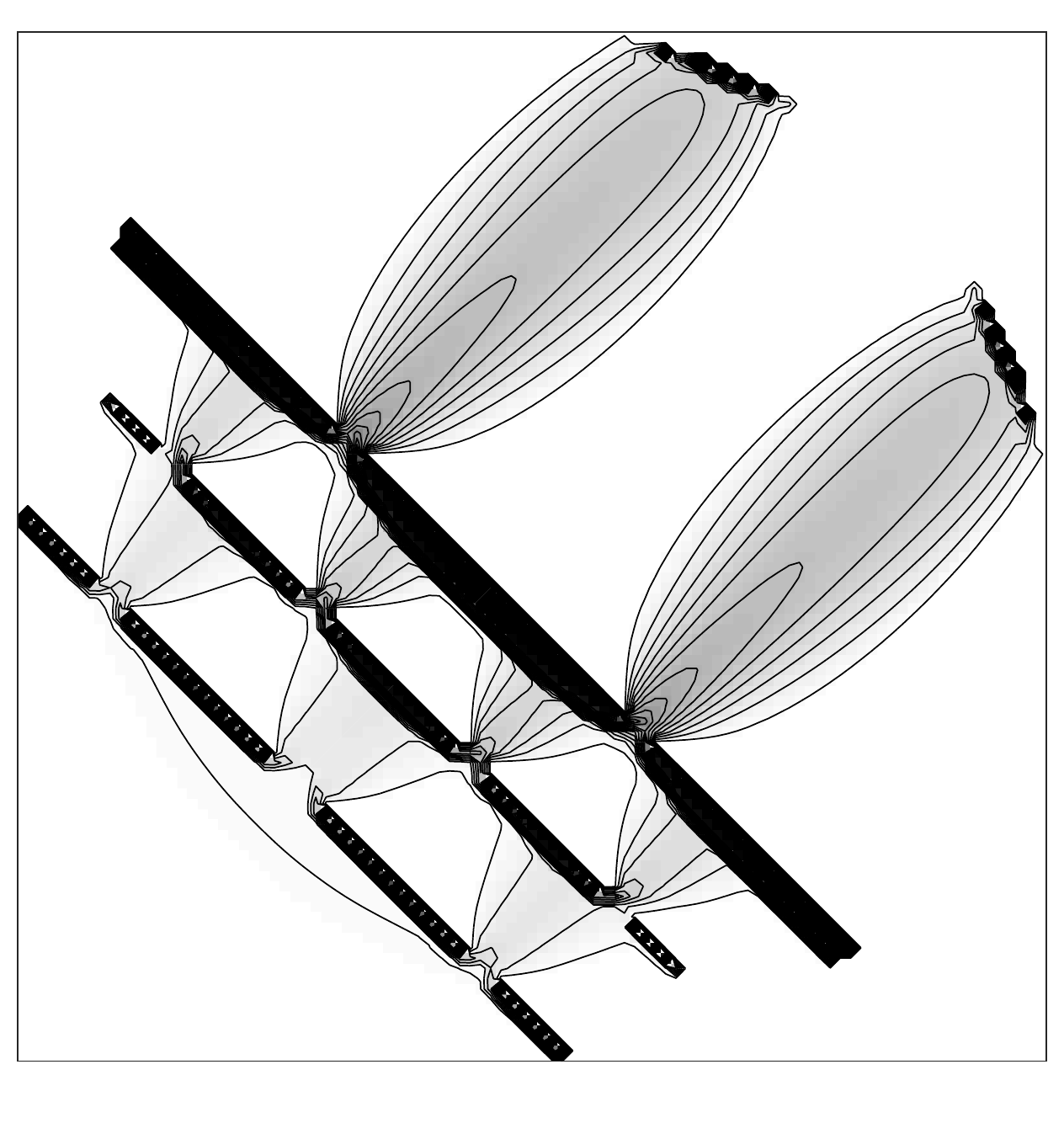}};
    \draw[gray,thick,dashed] (1.05,-4.14) circle (1.01cm);
    \node (ib_4) at (1.,2.8) {\small 1D-IHU};
  \end{tikzpicture}
  \hspace{-0.2cm}
  \begin{tikzpicture}
    \node[anchor=south west,inner sep=0] at (0,0) {\includegraphics[scale=0.162]{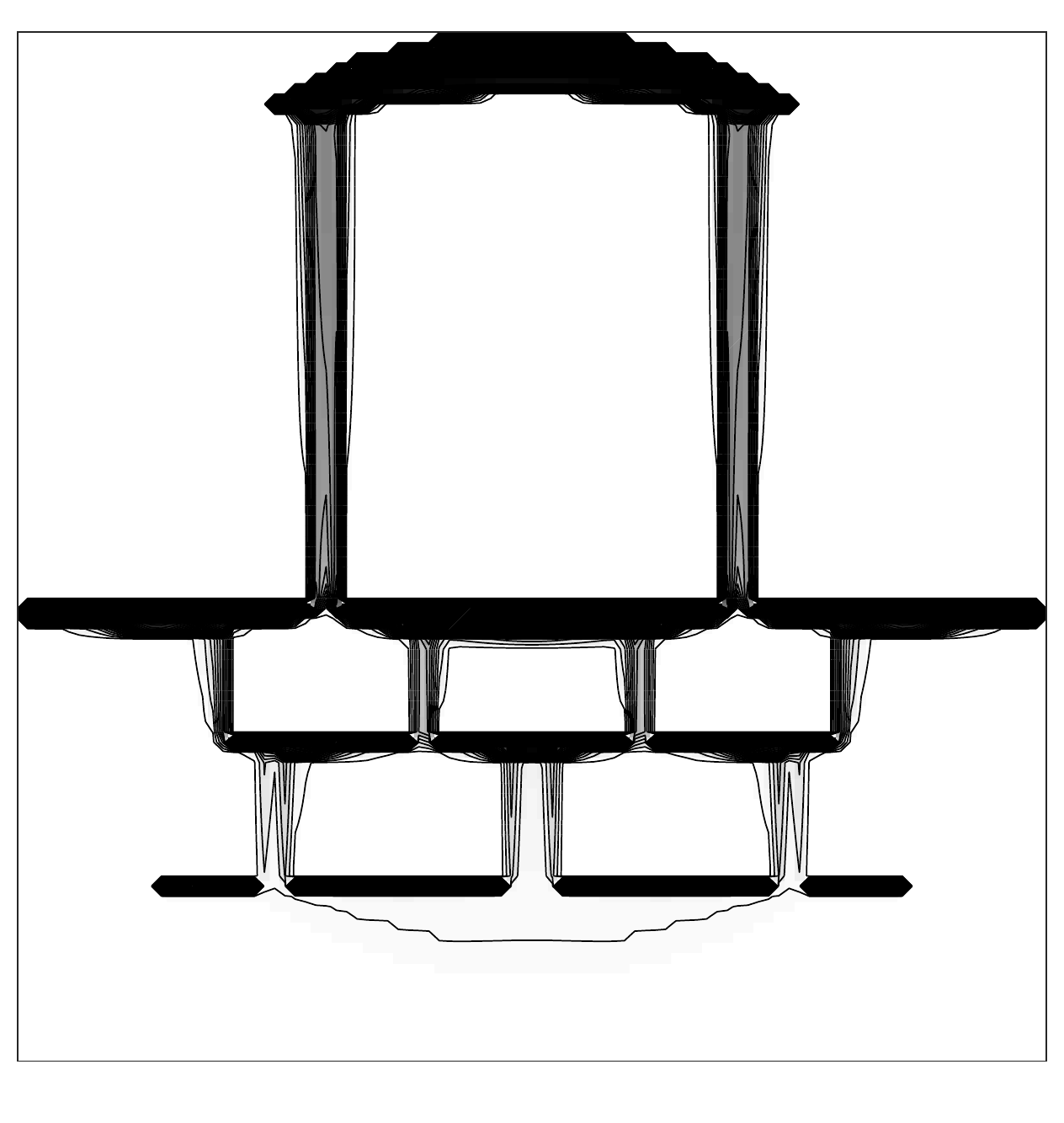}};
    \draw[gray,thick,dashed] (1.05,1.14) circle (1.01cm);
    \node[anchor=south west,inner sep=0, rotate=22.5] at (0.52,-2.945) {\includegraphics[scale=0.162]{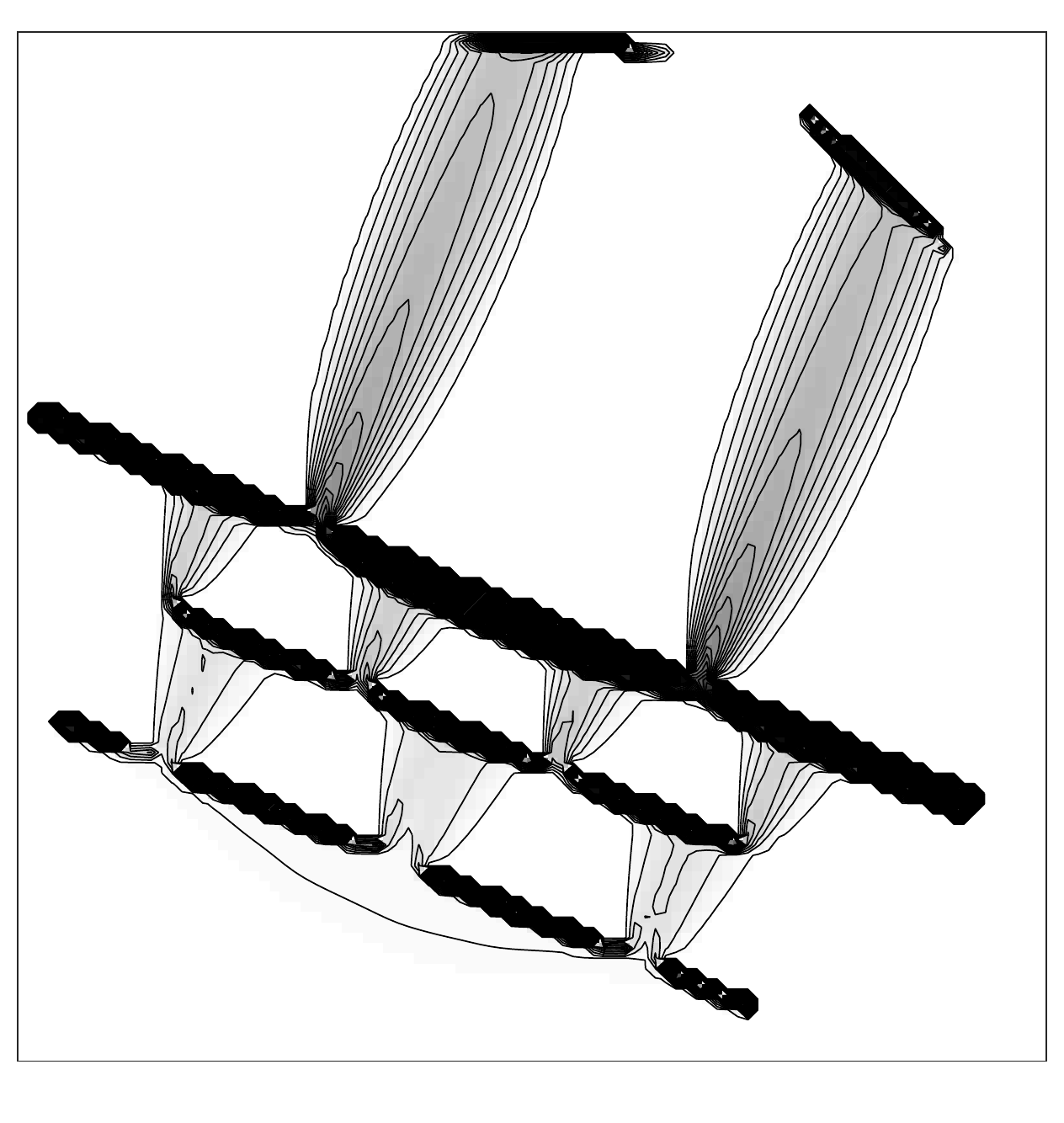}};
    \draw[gray,thick,dashed] (1.05,-1.5) circle (1.01cm);
    \node[anchor=south west,inner sep=0, rotate=45] at (1.115,-5.68) {\includegraphics[scale=0.162]{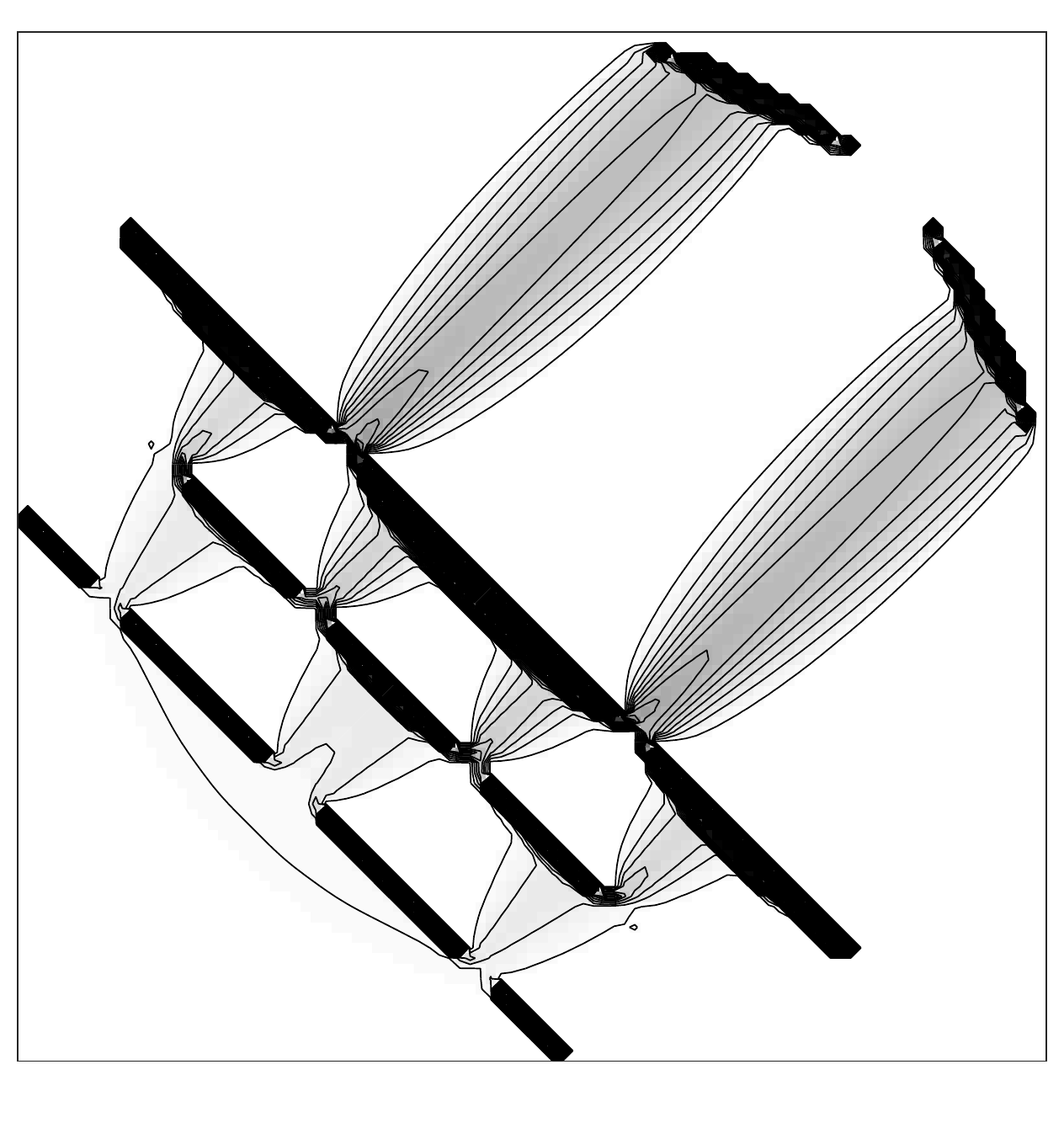}};
    \draw[gray,thick,dashed] (1.05,-4.14) circle (1.01cm);
    \node (ib_4) at (1.05,2.8) {\small MultiD-PPU};
    \node (ib_4) at (1.05,2.451) {\small (SMU)};
  \end{tikzpicture}
  \hspace{-0.2cm}
  \begin{tikzpicture}
    \node[anchor=south west,inner sep=0] at (0,0) {\includegraphics[scale=0.162]{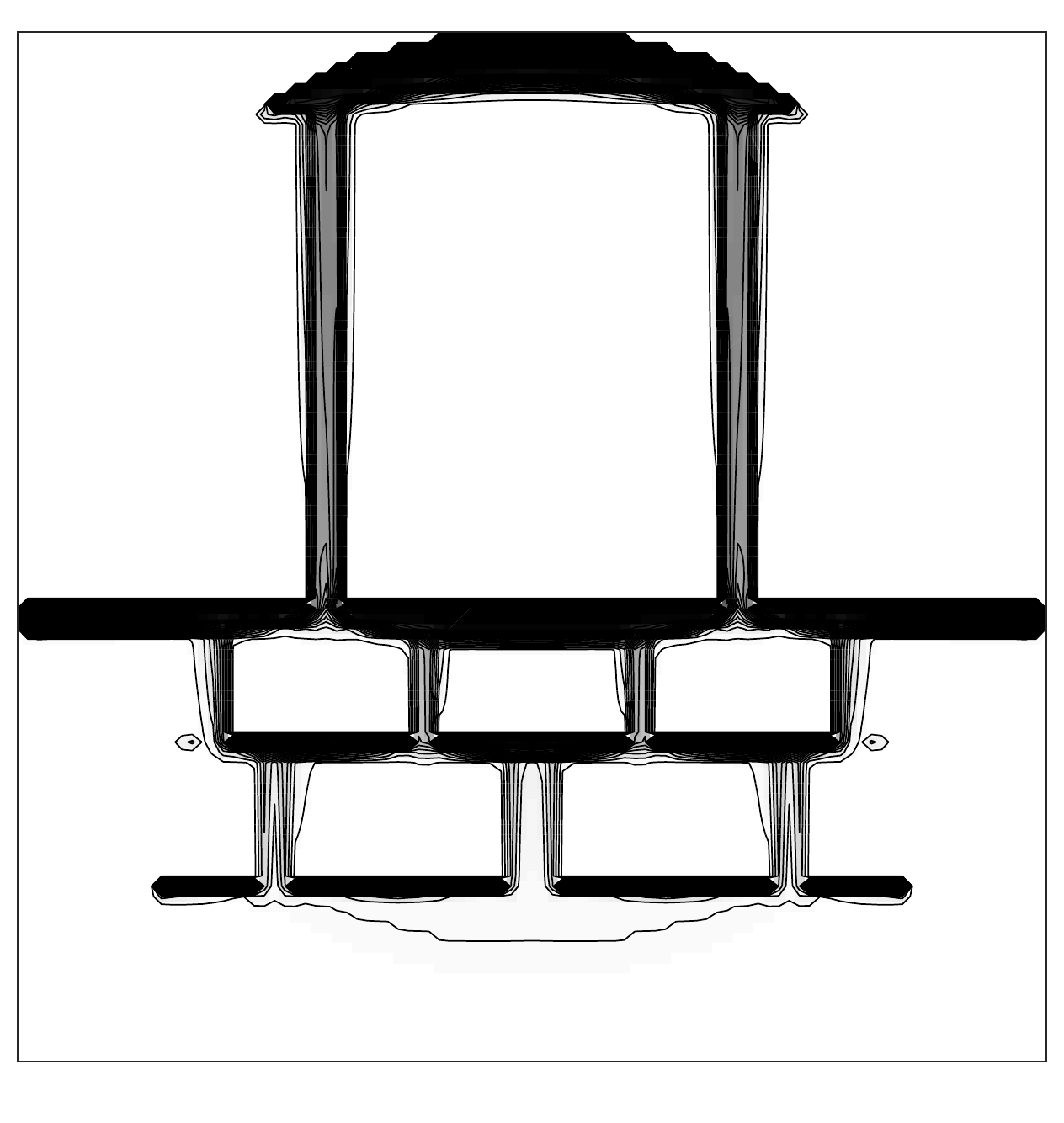}};
    \draw[gray,thick,dashed] (1.05,1.14) circle (1.01cm);
    \node[anchor=south west,inner sep=0, rotate=22.5] at (0.52,-2.945) {\includegraphics[scale=0.162]{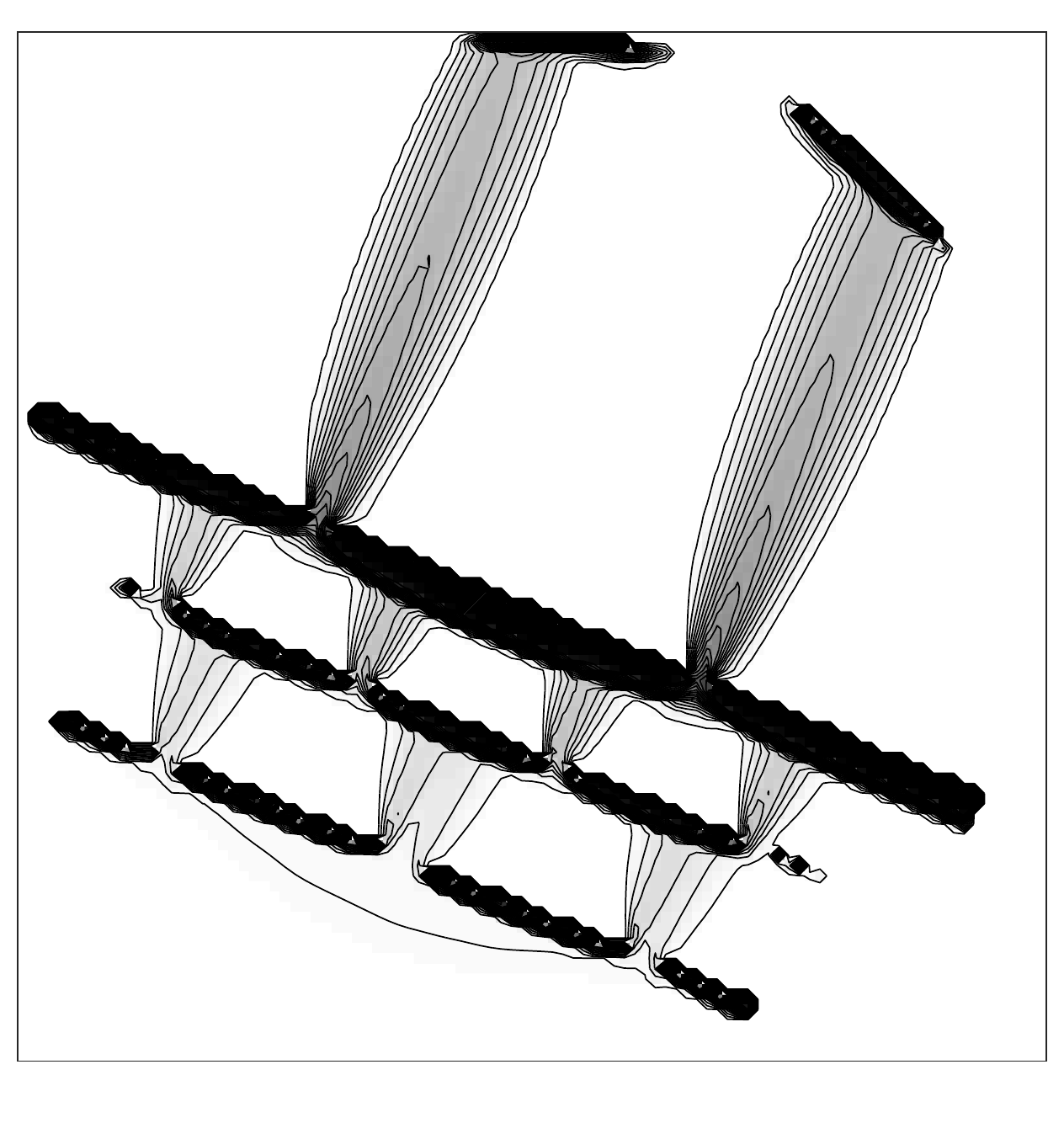}};
    \draw[gray,thick,dashed] (1.05,-1.5) circle (1.01cm);
    \node[anchor=south west,inner sep=0, rotate=45] at (1.115,-5.68) {\includegraphics[scale=0.162]{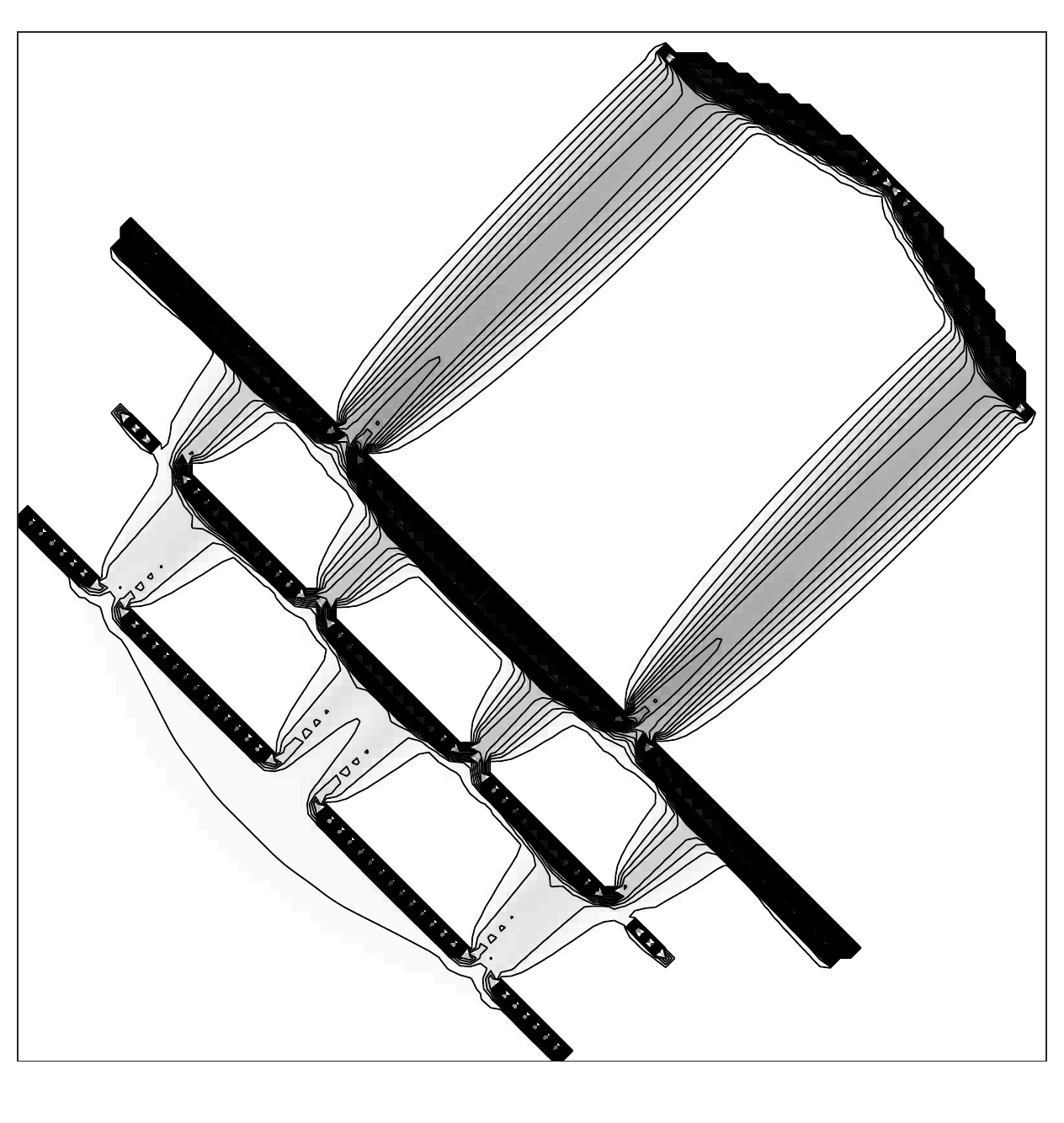}};
    \draw[gray,thick,dashed] (1.05,-4.14) circle (1.01cm);
    \node (ib_4) at (1.05,2.8) {\small MultiD-IHU};
    \node (ib_4) at (1.05,2.451) {\small (SMU4)};
  \end{tikzpicture}
  \vspace{-0.3cm}
\caption{\label{fig:saturation_maps_gravity_segregation_with_low_perm_layers}Saturation maps for different angles $\theta$ 
after 6000 days for the gravity segregation problem with low-permeability layers. The CFL number of these simulations is 8.9.
Outside the dashed circle, the control volumes have an absolute permeability set to $5 \times 10^{-9} \, \text{mD}$.
}
\end{figure}

The nonlinear iteration count is given in Table \ref{tab:nonlinear_behavior_gravity_segregation_low_perm_layers}. For 
this buoyancy-dominated example, the reduction in the number of nonlinear iterations with MultiD-PPU compared to 1D-PPU
is much smaller than in the previous sections. We also note that 1D-IHU exhibits a better nonlinear behavior than MultiD-PPU.
MultiD-IHU still achieves a very robust nonlinear behavior that remains insensitive to the orientation of the computational 
grid. For a CFL number of 8.9 (respectively, 17.8), MultiD-IHU reduces the total number of nonlinear iterations summed over all
cases by 19\% (respectively, 25.8\%) compared to 1D-PPU. It also reduces the number of Newton iterations compared to MultiD-PPU, 
and performs slightly better than 1D-IHU.

\begin{table}[!ht]
\scalebox{0.85}{
\centering
         \begin{tabular}{ccccccccc}
           \\ \toprule
             Angle     & \multicolumn{4}{c}{Smaller CFL}                & \multicolumn{4}{c}{Larger CFL}                    \\ 
             $\theta$  & 1D-PPU & 1D-IHU & MultiD-PPU & MultiD-IHU    & 1D-PPU    & 1D-IHU  & MultiD-PPU    & MultiD-IHU \\ \toprule
              0        &    253 &   238  &     257    &    241        &    199    &    187  &       202     &  182       \\
              $\pi/12$ &    307 &   239  &     292    &    235        &    244    &    171  &       231     &  167       \\
              $\pi/8$  &    290 &   234  &     285    &    231        &    225    &    174  &       222     &  163       \\
              $\pi/6$  &    300 &   236  &     290    &    233        &    226    &    171  &       228     &  158       \\
              $\pi/4$  &    281 &   229  &     277    &    229        &    212    &    167  &       214     &  151       \\ \bottomrule 
         \end{tabular}}
 \caption{\label{tab:nonlinear_behavior_gravity_segregation_low_perm_layers} Total number of Newton iterations for different angles after 6000 days 
in the gravity segregation problem with low-permeability layers. There was no time step cut for any of the schemes considered here.}
\end{table}

\section{\label{section_summary_of_results_and_conclusions}Summary of results and conclusions}

We have proposed a robust fully implicit finite-volume 
scheme for coupled multiphase flow and transport with buoyancy that is 
more accurate than commonly used first-order 1D upwinding schemes.
Specifically, we have presented a truly multidimensional 
extension of Implicit Hybrid Upwinding (IHU), referred to as MultiD-IHU, 
that reduces the numerical biasing due to the orientation of the underlying 
grid while retaining the monotonicity and robustness of IHU in the presence 
of competing viscous and buoyancy forces. This extension builds on the 
methodology proposed in \cite{kozdon2011multidimensional} for Phase-Potential 
Upwinding and is based on an extended stencil that adapts to the local flow
pattern for the evaluation of the phase mobility ratios.

We have tested the proposed scheme on two-phase flow numerical examples 
including competing viscous and buoyancy forces and characterized by the 
presence of countercurrent flow. We have shown that MultiD-IHU reduces 
the grid orientation effect as well as MultiD-PPU for the two first 
numerical examples, and leads to a more significant reduction than 
MultiD-PPU for the gravity segregation case with low-permeability layers. 
We have also demonstrated that MultiD-IHU retains the robust nonlinear 
behavior of 1D-IHU, even for large time steps. This is particularly the 
case for the third numerical example, in which MultiD-IHU reduces the 
number of nonlinear iterations by up 25.8\% compared to 1D-PPU and by 
up 25.2\% compared to MultiD-PPU. The computational gain due to the 
reduction in nonlinear iterations will help offset the added per-iteration 
cost associated with the larger stencils of multidimensional schemes.
Future work includes extending the methodology presented 
here to a robust, fully implicit high-resolution approach that would 
tend to a truly accurate scheme.

\section*{Acknowledgments}
The authors thank the SUPRI-B affiliates program at Stanford University, 
and in particular Prof. Hamdi Tchelepi for his insight and guidance. 
The first author acknowledges the financial support of Total during 
his PhD at Stanford University.

\appendix

\section{\label{app_properties_of_matrix_A}Properties of matrices $\boldsymbol{A}$ and $\boldsymbol{A}^{-1} \boldsymbol{B}$}

We study the properties of the matrices $\boldsymbol{A}$ and
$\boldsymbol{C} \equiv \boldsymbol{A}^{-1} \boldsymbol{B}$
defined in (\ref{definition_A}) and (\ref{definition_B}).
By construction, on each row, $\boldsymbol{A}$ has one positive unit
term on the diagonal and one non-positive off-diagonal term.
The definition of the SMU4 limiter given in 
(\ref{limiter}) guarantees that
\begin{equation}
  0 \leq \varphi^{\textit{SMU4}}(r) < 1 \quad \forall r \in [0,1].
  \label{property_smu4_limiter}
\end{equation}
Combining (\ref{property_smu4_limiter}) with the definition
of the weights in (\ref{definition_omega_v}), we obtain
\begin{equation}
  0 \leq \overline{\omega}^V_{k+1/2} < 1 \quad \forall k \in \{1, \dots, 4\}.
\label{batstoi_francois}
\end{equation}
Therefore, all the off-diagonal terms in $\boldsymbol{A}$ are
strictly smaller than 1, which implies that this matrix is strictly
diagonally dominant in the sense of
\begin{equation}
  |a_{ii}| > \sum_{j \neq i} |a_{ij}| \quad i \in \{1, \dots, 4\}.
  \label{inequality}
\end{equation}
$\boldsymbol{A}$ is therefore a non-singular M-matrix, and
$\boldsymbol{A}^{-1}$ is non-negative. Since
$\boldsymbol{B}$ is also non-negative by construction,
$\boldsymbol{A}^{-1} \boldsymbol{B}$ is non-negative.
To show that for each row of $\boldsymbol{C}$,
the sum of the entries is equal to one, i.e., $\sum_k c_{ik} = 1$
for $i \in \{1, \dots, 4\}$, we define the vector
$\boldsymbol{e} = [1, 1, 1, 1]^T$. Then,
$\boldsymbol{C} \boldsymbol{e} = \boldsymbol{e}$ is equivalent
to $\boldsymbol{A} \boldsymbol{e} = \boldsymbol{B} \boldsymbol{e}$,
which follows from the definitions
(\ref{definition_A}) and (\ref{definition_B}).

\section{\label{app_viscous_flux_monotonicity}Viscous flux monotonicity}

In this section and the next one, we show that the MultiD-IHU
scheme is monotone in the sense of (\ref{monotonicity_conditiona}).
We consider first the monotonicity of the viscous flux by analyzing
the sign of the derivatives with respect to saturation of 
\begin{equation}
\overline{V}_{\ell,k+1/2} - \overline{V}_{\ell,k-1/2}
=
\overline{\chi}_{\ell,k+1/2} \overline{u}_{T,k+1/2}
-
\overline{\chi}_{\ell,k-1/2} \overline{u}_{T,k-1/2},
\end{equation}
for a fixed total velocity field. Following
\cite{kozdon2011multidimensional}, we consider only three cases
since the other cases can be treated by symmetry. We first assume
that
$\overline{u}_{T,k+1/2} \leq 0$
and
$\overline{u}_{T,k-1/2} \geq 0$.
Using (\ref{linear_solution_to_define_overline}) and  matrix
$\boldsymbol{C} = \boldsymbol{A}^{-1} \boldsymbol{B}$, we have
\begin{equation}
\overline{V}_{\ell,k+1/2} - \overline{V}_{\ell,k-1/2}
=
\sum^{4}_{j=1} (c_{kj} \overline{u}_{T,k+1/2} - c_{(k-1)j} \overline{u}_{T,k-1/2}) \chi_{\ell,j}.
\label{case_1_v_k_minus_v_k-1}
\end{equation}
Using the assumption on the sign of the total velocities and the
non-negativity of
$\boldsymbol{C} \equiv \boldsymbol{A}^{-1} \boldsymbol{B}$
shown in \ref{app_properties_of_matrix_A}, 
(\ref{case_1_v_k_minus_v_k-1}) gives
\begin{equation}
\frac{ \partial ( \overline{V}_{\ell,k+1/2} - \overline{V}_{\ell,k-1/2} )}{\partial S_{\ell, j \neq k}} = 
(c_{kj} \overline{u}_{T,k+1/2} - c_{(k-1)j} \overline{u}_{T,k-1/2}) \frac{\partial \chi_{\ell,j}}{\partial S_{\ell, j \neq k}} \leq 0.
\end{equation}
In the second case, we assume that $\overline{u}_{T,k+1/2} \geq 0$
and $\overline{u}_{T,k-1/2} \leq 0$. Using the definition of the
upwinding of the mobility ratio, we can write
\begin{equation}
\overline{V}_{\ell,k+1/2} - \overline{V}_{\ell,k-1/2}
=
\chi_{\ell,k} \overline{u}_{T,k+1/2}
-
\chi_{\ell,k} \overline{u}_{T,k-1/2},
\label{case_2_v_k_minus_v_k-1}
\end{equation}
which yields
\begin{equation}
\frac{ \partial ( \overline{V}_{\ell,k+1/2}
-
\overline{V}_{\ell,k-1/2} )}{\partial S_{\ell, j \neq k}} = 
0.
\end{equation}
The third case is such that $\overline{u}_{T,k+1/2} \geq 0$
and $\overline{u}_{T,k-1/2} \geq 0$. Using the weighted average of
(\ref{definition_chi}), we write
\begin{align}
\overline{V}_{\ell,k+1/2} - \overline{V}_{\ell,k-1/2}
&= (1-\overline{\omega}^V_{k+1/2}) \chi_{\ell,k} \overline{u}_{T,k+1/2}
+ ( \overline{\omega}^V_{k+1/2} \overline{u}_{T,k+1/2}
- \overline{u}_{T,k-1/2} ) \overline{\chi}_{\ell,k-1/2} \nonumber \\
&= (1-\overline{\omega}^V_{k+1/2}) \chi_{\ell,k} \overline{u}_{T,k+1/2}
+ ( \overline{\omega}^V_{k+1/2} \overline{u}_{T,k+1/2}
-
\overline{u}_{T,k-1/2} ) \sum^{4}_{j = 1} c_{(k-1)j} \chi_{\ell,j}.
\label{case_3_v_k_minus_v_k-1}
\end{align}
Taking the derivatives with respect to saturation in
(\ref{case_3_v_k_minus_v_k-1}) gives
\begin{equation}
\frac{ \partial ( \overline{V}_{\ell,k+1/2} - \overline{V}_{\ell,k-1/2} )}{\partial S_{\ell, j \neq k}} = ( \overline{\omega}^V_{k+1/2} \overline{u}_{T,k+1/2} - \overline{u}_{T,k-1/2} ) c_{(k-1)j} \frac{\partial \chi_{\ell,j}}{\partial S_{\ell,j\neq k}} \leq 0,
\end{equation}
since the definition of the limiter guarantees that
$( \overline{\omega}^V_{k+1/2} \overline{u}_{T,k+1/2} - \overline{u}_{T,k-1/2} ) \leq 0$.

\section{\label{app_buoyancy_flux_monotonicity}Buoyancy flux monotonicity}

Next, we consider the monotonicity of the buoyancy term by
studying the sign of the derivative with respect to saturation of
\begin{equation}
\overline{G}_{\ell,k+1/2} - \overline{G}_{\ell,k-1/2} =
\overline{T}_{k+1/2} \overline{\psi}_{\ell,m,k+1/2} (\rho_{\ell} - \rho_m) \overline{\Delta z}_{k+1/2}
-
\overline{T}_{k-1/2} \overline{\psi}_{\ell,m,k-1/2} (\rho_{\ell} - \rho_m) \overline{\Delta z}_{k-1/2}.
\end{equation}
We assume that $\rho_{\ell} - \rho_m > 0$ and
$\overline{\Delta z}_{k+1/2} > 0$. The other cases can be treated
analogously and will be omitted for brevity. On a Cartesian grid,
we only have to study two configurations. In the first
configuration, we assume that
$\overline{\Delta z}_{k+3/2} \overline{\Delta z}_{k+1/2} > 0$
and $\overline{\Delta z}_{k+1/2} \overline{\Delta z}_{k-1/2} < 0$,
which yields
\begin{align}
\overline{G}_{\ell,k+1/2} &= \overline{T}_{k+1/2} \bigg( (1-\overline{\omega}^G_{k+1/2}) \displaystyle \frac{ \lambda_{\ell,k} \lambda_{m,k+1} }{ \lambda_{\ell,k} + \lambda_{m,k+1} }
+ \overline{\omega}^G_{k+1/2} \displaystyle \frac{ \lambda_{\ell,k} \lambda_{m,k+2} }{ \lambda_{\ell,k} + \lambda_{m,k+2} } \bigg) (\rho_{\ell} - \rho_m) \overline{\Delta z}_{k+1/2},  \\
\overline{G}_{\ell,k-1/2} &=  \overline{T}_{k-1/2} \bigg( (1-\overline{\omega}^G_{k-1/2}) \displaystyle \frac{ \lambda_{\ell,k} \lambda_{m,k-1} }{ \lambda_{\ell,k} + \lambda_{m,k-1} }
+ \overline{\omega}^G_{k-1/2}   \displaystyle \frac{ \lambda_{\ell,k} \lambda_{m,k-2} }{ \lambda_{\ell,k} + \lambda_{m,k-2} } \bigg) (\rho_{\ell} - \rho_m) \overline{\Delta z}_{k-1/2}.
\end{align}
Noting that $m \neq \ell$, the sign of the derivatives of the
mobility ratios is given by
\begin{equation}
\displaystyle \frac{\partial}{\partial S_{\ell,j\neq k}} \bigg( \frac{ \lambda_{\ell,k} \lambda_{m,j} }{ \lambda_{\ell,k} + \lambda_{m,j} } \bigg)
= \frac{ \lambda^2_{\ell,k} \displaystyle \frac{\partial \lambda_{m,j}}{\partial S_{\ell,j \neq k}} }{ ( \lambda_{\ell,k} + \lambda_{m,j} )^2 } \leq 0.
\label{sign_derivative_mobility_ratio}
\end{equation}
We now use the assumptions on the densities and the depth to obtain
\begin{equation}
\frac{ \partial ( \overline{G}_{\ell,k+1/2} - \overline{G}_{\ell,k-1/2} ) }{\partial S_{\ell, j \neq k } }  \leq 0.
\end{equation}
In the second configuration, we assume that
$\overline{\Delta z}_{k+3/2} \overline{\Delta z}_{k+1/2} < 0$
and $\overline{\Delta z}_{k+1/2} \overline{\Delta z}_{k-1/2} > 0$.
This case yields
\begin{align}
\overline{G}_{\ell,k+1/2}   &=  \overline{T}_{k+1/2} \bigg( (1-\overline{\omega}^G_{k+1/2}) \displaystyle \frac{ \lambda_{\ell,k} \lambda_{m,k+1} }{ \lambda_{\ell,k} + \lambda_{m,k+1} }
+
\overline{\omega}^G_{k+1/2}   \displaystyle \frac{ \lambda_{\ell,k-1} \lambda_{m,k+1} }{ \lambda_{\ell,k-1} + \lambda_{m,k+1} } \bigg) (\rho_{\ell} - \rho_m) \overline{\Delta z}_{k+1/2},  \\
\overline{G}_{\ell,k-1/2} &=  \overline{T}_{k-1/2} \bigg( (1-\overline{\omega}^G_{k-1/2}) \displaystyle \frac{ \lambda_{\ell,k-1} \lambda_{m,k} }{ \lambda_{\ell,k-1} + \lambda_{m,k} }
+ \overline{\omega}^G_{k-1/2}   \displaystyle \frac{ \lambda_{\ell,k-1} \lambda_{m,k+1} }{ \lambda_{\ell,k-1} + \lambda_{m,k+1} } \bigg) (\rho_{\ell} - \rho_m) \overline{\Delta z}_{k-1/2}.
\end{align}
Using these expressions, we obtain
\begin{align}
\overline{G}_{\ell,k+1/2} - \overline{G}_{\ell,k-1/2} &=  \bigg( \overline{T}_{k+1/2} (1-\overline{\omega}^G_{k+1/2}) \displaystyle \frac{ \lambda_{\ell,k} \lambda_{m,k+1} }{ \lambda_{\ell,k} + \lambda_{m,k+1} } \overline{\Delta z}_{k+1/2}  \nonumber \\
                                   &-  \overline{T}_{k-1/2} (1-\overline{\omega}^G_{k-1/2}) \displaystyle \frac{ \lambda_{\ell,k-1} \lambda_{m,k} }{ \lambda_{\ell,k-1} + \lambda_{m,k} }  \overline{\Delta z}_{k-1/2} \bigg) (\rho_{\ell} - \rho_m )
                                   \nonumber \\
                                  &+ (\overline{T}_{k+1/2} \overline{\omega}^G_{k+1/2} \overline{\Delta z}_{k+1/2} - \overline{T}_{k-1/2} \overline{\omega}^G_{k-1/2} \overline{\Delta z}_{k-1/2}) \displaystyle \frac{ \lambda_{\ell,k-1} \lambda_{m,k+1} }{ \lambda_{\ell,k-1} + \lambda_{m,k+1} } (\rho_{\ell} - \rho_m).
\label{case_2_g_k_minus_g_k-1}
\end{align}
Equation (\ref{symmetry_property}) guarantees that
$\overline{T}_{k+1/2} \overline{\omega}^G_{k+1/2} \overline{\Delta z}_{k+1/2} - \overline{T}_{k-1/2} \overline{\omega}^G_{k-1/2} \overline{\Delta z}_{k-1/2} = 0$
whenever
$\overline{\Delta z}_{k+3/2} \overline{\Delta z}_{k+1/2} < 0$
and
$\overline{\Delta z}_{k+1/2} \overline{\Delta z}_{k-1/2} > 0$.
Therefore the term in the third line of
(\ref{case_2_g_k_minus_g_k-1}) cancels and we can use the sign of the derivatives of the mobility ratios to obtain the result
\begin{equation}
\frac{ \partial ( \overline{G}_{\ell,k+1/2} - \overline{G}_{\ell,k-1/2} ) }{\partial S_{\ell, j \neq k } }  \leq 0.
\end{equation}

\bibliography{biblio}

\end{document}